%% file: blow_up_asymptotic_1.tex
\documentclass{article}
\usepackage{authblk}
\usepackage{fancyhdr}
\usepackage{amscd}
\usepackage[dvipdfmx]{graphicx}
\usepackage{amsmath}
\usepackage{amssymb}
\usepackage{amsthm}
\usepackage{mathrsfs}
\newtheorem{dfn}{Definition}[section]
\newtheorem{thm}[dfn]{Theorem}
\newtheorem{prop}[dfn]{Proposition}
\newtheorem{lem}[dfn]{Lemma}

\newtheorem{rem}[dfn]{Remark}

\newtheorem{ass}[dfn]{Assumption}
\newtheorem{ex}[dfn]{Example}

\newcommand{\KMc}[1]{{\color{black} #1}}
\newcommand{\KMd}[1]{{\color{black} #1}}
\newcommand{\KMe}[1]{{\color{black} #1}}
\newcommand{\KMf}[1]{{\color{black} #1}}

\newcommand{\KMg}[1]{{\color{black} #1}}
\newcommand{\KMh}[1]{{\color{black} #1}}
\newcommand{\KMi}[1]{{\color{black} #1}}
\newcommand{\KMj}[1]{{\color{black} #1}}
\newcommand{\KMk}[1]{{\color{black} #1}}
\newcommand{\KMl}[1]{{\color{black} #1}}

\newcommand{\KMm}[1]{{\color{black} #1}}
\newcommand{\KMn}[1]{{\color{black} #1}}

\numberwithin{equation}{section}

\usepackage{color}
\begin{document}

\title{Multi-order asymptotic expansion of blow-up solutions for autonomous ODEs. I - 
Method and Justification}

\author[1]{Taisei Asai}
\author[2]{Hisatoshi Kodani}
\author[2,3]{Kaname Matsue\thanks{(Corresponding author, {\tt kmatsue@imi.kyushu-u.ac.jp})}}
\author[2]{\\ Hiroyuki Ochiai}
\author[4,5]{Takiko Sasaki}

\affil[1]{
\normalsize{
Graduate School of Fundamental Science and Engineering, Waseda University, 3-4-1 Okubo, Shinjuku, Tokyo 169-8555, Japan
}
}
\affil[2]{Institute of Mathematics for Industry, Kyushu University, Fukuoka 819-0395, Japan}
\affil[3]{International Institute for Carbon-Neutral Energy Research (WPI-I$^2$CNER), Kyushu University, Fukuoka 819-0395, Japan}
\affil[4]{Department of Mathematical Engineering, Faculty of Engineering, Musashino University, 3-3-3 Ariake, Koto-ku, Tokyo 135-8181, Japan}
\affil[5]{Mathematical Institute, Tohoku University, Aoba, Sendai 980-8578, Japan}
\maketitle

\begin{abstract}
\KMh{In this paper, we provide a systematic methodology for calculating multi-order asymptotic expansion of blow-up solutions near blow-up for autonomous ordinary differential equations (ODEs).}
Under the specific form of the principal term of blow-up solutions for a class of vector fields, we extract algebraic objects determining all possible orders in the asymptotic expansions.
Examples for calculating concrete multi-order asymptotic expansions of blow-up solutions are finally collected.
\end{abstract}

{\bf Keywords:} blow-up solutions, asymptotic expansion
\par
\bigskip
{\bf AMS subject classifications : } 34A26, 34D05, 34E05, 34E10, 58K55

\tableofcontents

\section{Introduction}
\label{section-intro}

Asymptotic expansion is an important and useful tool for analyzing concrete behavior of functions or solutions of differential equations at infinity \KMh{and/or} near given points (e.g. \cite{BO1999}).
Its importance is widely understood in physical and \KMe{mathematical problems in analyzing functions.
A famous investigation} involving calculations of asymptotic expansion around given points, including singularities, \KMe{in the field of differential equations} is the {\em Painlev\'{e} test} (e.g. \cite{FP1991, K1992, KJH1997}), which roughly aims at verifying whether all solutions of differential equations of the form
\begin{equation}
\label{ODE-complex}
\frac{du}{dz} = f(z,u),\quad z\in \mathbb{C}
\end{equation}
are single valued and all possible \KMm{(movable)} singularities $z$ are poles.
Note that the differential equation whose solutions satisfy the above \KMh{properties} is often said to possess the {\em Painlev\'{e} property}.
In wide investigations for the criterion of the Painlev\'{e} property, a concept of {\em resonance} \KMn{(e.g. \cite{KJH1997})} has successfully played a key role in determining asymptotic expansion of solutions with their \KMh{dependence on initial points}.
The presence of resonance is algebraically verified and is applicable even when solutions \KMh{possess} {\em movable} singularities.
Conceptually, multi-order asymptotic expansion of solutions is systematically determined once all resonances (around a given singularity $z_0$, possibly at infinity $z_0 = \infty$) are determined.
\par
Our main interest in the present paper is {\em blow-up solutions} of (autonomous) ordinary differential equations (ODEs for short) depending on the real-time variable $t$:
\begin{equation}
\label{ODE-intro}
\frac{d{\bf y}}{dt} = f({\bf y}),\quad {\bf y}(\KMi{t_0}) = {\bf y}_0\in \mathbb{R}^n,
\end{equation}
where $f: \mathbb{R}^n\to \mathbb{R}^n$ is a smooth function.
A solution ${\bf y}(t)$ is said to \KMi{\em blow up} at $t_{\max} < \infty$ if \KMh{its modulus} diverges as $t\to t_{\max}-0$.
The finite value $t_{\max}$ is known as the {\em blow-up time}.
Properties of blow-up solutions such as the existence and asymptotic behavior are of great importance in the field of (ordinary, partial, delay, etc.) differential equations, and are widely investigated in decades  \KMn{from many aspects (e.g. \cite{BFGK2011, BK1988, CHO2007, FM2002, GV2002})}.
Asymptotic behavior of blow-up solutions is often referred to as determination of {\em blow-up rates}, which is characterized as the following form:
\begin{equation*}
{\bf y}(t) \sim {\bf u}(\theta(t))\quad \text{ as }\quad t\to t_{\max},
\end{equation*}
for a function ${\bf u}$, where 
\begin{equation*}
\theta(t) = t_{\max}-t
\end{equation*}
and $t_{\max}$ is assumed to \KMh{be} a finite value and to be known a priori.
In many studies, however, the function ${\bf u}$ is assumed to \KMh{consist of a} single term, and investigations of the concrete form of ${\bf u}$ as a function with multiple terms are limited. 
One of the reasons could be the difficulty \KMf{even} for deriving the principal term, namely the first term, of ${\bf u}$ except special cases (e.g. \cite{D1985}).
One would expect that the similar approach to the Painlev\'{e} test can be applied to determining multi-order asymptotic expansion of blow-ups. 
It should be mentioned here that the Painlev\'{e} test typically begins with the \KMg{anzatz} of solutions being the {\em Frobenius series}\footnote{
In the linear homogeneous system, it is proved by Fuchs that the system admits \KMn{solutions} expressed by the Frobenius series around {\em regular singular points}.
This ansatz is typically applied to linear inhomogeneous equations or nonlinear equations (cf. \cite{FP1991, K1992, KJH1997}).
There is also a technique for calculating asymptotic expansions of solutions around {\em irregular singular points} (e.g. \cite{BO1999}), which is omitted here since it is beyond our scope in the present study.
}, namely
\begin{equation}
\label{Frobenius-cpx}
\KMe{
u(z) = (z-z_0)^{-\KMi{\rho}}\sum_{m=0}^\infty a_m (z-z_0)^{m},\quad a_m\in \mathbb{C}\KMi{,}
}
\end{equation}
when \KMe{$z_0$} is a regular singular point (cf. \cite{BO1999}).
On the other hand, there is no guarantee that blow-up solutions always admit \KMe{the similar form to (\ref{Frobenius-cpx}), namely
\begin{equation}
\label{Frobenius}
y(t) = \theta(t)^{-\KMi{\rho}}\sum_{m=0}^\infty a_m \theta(t)^{m},\quad a_m\in \mathbb{R}.
\end{equation}
In particular,} no information about the successive terms of blow-up solutions is provided a priori in general.
Nevertheless, the blow-up time $t_{\max}$ is interpreted as a movable singularity in $\mathbb{C}$ \KMf{because}\KMn{, as far as the (complex) analytic continuation is admitted,} the system (\ref{ODE-intro}) is also considered as the system depending on complex time variable $z$ restricted to $\{t \equiv {\rm Re}\,z > 0\}$\KMi{,} and $t_{\max}$ depends on initial points ${\bf y}_0$, and the multi-order asymptotic expansion of blow-up solutions contributes \KMe{to understanding the blow-up behavior itself in detail, and} the classification of $t_{\max}$ as singularities, as investigated in the Painlev\'{e} test.
\par
In recent years, the third author and his collaborators have developed a framework to characterize blow-up solutions from the viewpoint of dynamics at infinity (e.g. \cite{Mat2018, Mat2019}) based on preceding works (e.g. \cite{EG2006}) and have derived machineries of computer-assisted proofs for the existence of blow-up solutions as well as their qualitative and quantitative features (e.g. \cite{LMT2021, MT2020_1, MT2020_2, TMSTMO2017}).
\KMh{As in} the present paper, \KMh{finite-dimensional} vector fields with scale invariance in an asymptotic sense, {\em asymptotic quasi-homogeneity} defined precisely in Definition \ref{dfn-AQH}, are mainly concerned.
In this framework, dynamics at infinity is appropriately characterized, and blow-up solutions are shown to be characterized \KMk{by} dynamical properties of invariant sets, such as equilibria, \lq\lq at infinity".
In particular, {\em hyperbolicity} of such invariant sets \KMk{induces} the leading asymptotic approximation of blow-up solutions of the form $a_0(t_{\max} - t)^{-\rho}$ \KMk{as} in (\ref{Frobenius}) uniquely determined by the original vector field.
By means of the terminology in the field of (partial) differential equations, such blow-ups are said to be {\em type-I}.

One would then expect that multi-order asymptotics of blow-up solutions can be clarified by the dynamical property of invariant sets \lq\lq at infinity".
In particular, some characteristics determining multi-order asymptotics, like the resonance in the Painlev\'{e} test, can be extracted from the dynamical property of blow-up solutions.
\par
\bigskip 
The present subject, consisting of two papers, aims at  providing a systematic methodology to calculate multi-order asymptotic expansions of blow-up solutions for autonomous ODEs with their justifications \KMk{(Part I)}, and investigating the correspondence of the characterization of these asymptotic expansions to dynamical properties of blow-up solutions  \KMk{(Part II)}.
We pay attention to multi-order asymptotics of type-I blow-ups and, with \KMm{general asymptotic series} of blow-up solutions under an appropriate assumption, derive the associated \KMn{system of } differential \KMn{equations which all asymptotic terms are inductively solved}.
\par
The main result in the present paper is, under the above setting, to observe that the following algebraic objects essentially determine \KMm{all terms} appeared in asymptotic expansions of blow-up solutions:
\begin{itemize}
\item The {\em balance law} determining coefficients of the leading terms of blow-up solutions (Definition \ref{dfn-balance}).
\item The {\em blow-up power-determining matrix} determined by the quasi-homogeneous part of vector fields and blow-up solutions (Definition \ref{dfn-blow-up-power-ev}).
\item The {\em blow-up power eigenvalues} defined by eigenvalues of the blow-up power-determining matrix (Definition \ref{dfn-blow-up-power-ev}).
\end{itemize}
The combination of standard integration of ODEs with the above algebraic objects, \KMn{lower order terms if exist,} and typical assumptions in asymptotic expansions provide a systematic algorithm to calculate asymptotic expansions of blow-up solutions of an arbitrary order.
In particular, all possible terms and parameter dependence of blow-up solutions are essentially determined by blow-up power eigenvalues, and hence these eigenvalues play a role like resonance in the Painlev\'{e} test.
\par

\par
\bigskip
The rest of this paper is organized as follows.
In Section \ref{section-preliminaries-1}, \KMj{the precise definition of the class of vector fields we mainly treat is presented\KMm{.}
We then introduce a function measuring the asymptotic degree of functions with respect to $\theta(t)$, which plays a role in estimating asymptotic relations among terms appeared in our asymptotic expansions of blow-up solutions.}
\par
In Section \ref{section-asym}, we derive a systematic methodology to calculate multi-order asymptotic expansion of \KMm{(type-I) blow-up solutions under a mild assumption.} 
Quasi-homogeneity of vector fields in an asymptotic sense extracts algebraic objects essential to determine asymptotic expansions, which shall be called {\em blow-up power eigenvalues} associated with blow-up solutions.
More precisely, we show that \KMn{the} vector fields along type-I blow-up solutions induce \KMn{the} matrices determined by the quasi-homogeneous part of vector fields and coefficients of principal terms of blow-up solutions, which are called {\em blow-up power-determining matrices}.
Eigenvalues of such matrices, which are \KMf{exactly blow-up power eigenvalues} in our terminology, and lower order terms of vector fields are then shown to determine all possible terms in asymptotic expansion of blow-up solutions.
Once \KMn{these objects} are determined, standard integrator of \KMe{linear} ODEs yields asymptotic expansions of blow-up solutions of \KMj{an arbitrary order} with the extraction of their parameter dependence in a formal sense.
Justification of such formal expansions as the asymptotic expansion of the original blow-up solutions is also discussed \KMj{under mild assumptions}.

In Section \ref{section-examples}, examples of multi-order asymptotic expansions of blow-up solutions are presented for showing the applicability of our present methodology.
Examples below aim at demonstrating the following significance, respectively:
\begin{itemize}
\item Section \ref{section-ex-1dim} collects the simplest cases: the one-dimensional ODEs.
We shall provide the basic strategy based on the presented approach to determine multi-order asymptotic expansions of blow-up solutions.
\item Section \ref{section-ex-2phase} demonstrates the asymptotic expansion of blow-up solutions for a {\em rational} vector field \KMn{containing a component independent of scaling.}
In particular, an asymptotically quasi-homogeneous vector field is treated.
\item Sections \ref{section-ex-Andrews1} and \ref{section-ex-Andrews2} demonstrate blow-up solutions such that {\em parameters in coefficients of vector fields can affect the order of \KMg{$\theta(t)$ in} asymptotic expansions}.
In Section \ref{section-ex-Andrews1}, a rational vector field is treated.
In Section \ref{section-ex-Andrews2}, the third order asymptotic expansion is also calculated.
\item Section \ref{section-ex-KK} treats {\em asymptotically quasi-homogeneous} vector fields possessing two different blow-up solutions \KMg{whose orders of $\theta(t)$ in the first terms are identical}.
\item Section \ref{section-ex-log} treats asymptotically quasi-homogeneous vector fields possessing the similar type of blow-ups to Section \ref{section-ex-KK}, while {\em blow-up power-determining matrix admits a nontrivial Jordan normal form}.
\KMg{This section also exhibits an example of the presence of qualitatively different asymptotic behavior of solutions among components. 
In the field of partial differential equations, this phenomenon is recognized as {\em nonsimultaneous blow-up} (e.g. \cite{Ha2016, Ha2017}).}
\end{itemize}
Note that the first terms of all blow-up solutions listed here are already calculated in \cite{Mat2018, Mat2019} or their references\KMj{, while a criterion for determining the first terms of blow-ups is briefly reviewed in Part II \cite{asym2}}.
\par
In Appendix, proofs of several statements in Section \ref{section-asym} are provided, \KMe{where} advanced machineries in the theory of dynamical systems, such as {\em time-variable compactifications for asymptotically autonomous systems} (e.g. \cite{WXJ2021}), are applied. 
\par
\KMj{Finally, several examples showing multi-order asymptotic expansions of blow-up solutions are revisited in Part II \cite{asym2}, where the correspondence to dynamical properties obtained through desingularized vector fields are also described.}


\section{Setting and the asymptotic degree of functions}
\input{degree}

\section{Asymptotic expansion of type-I blow-ups}
\input{asymptotic}

\section{Examples}
\input{examples_1}



\section*{Concluding Remarks}

In this paper, we have provided a systematic methodology to derive asymptotic expansions of blow-up solutions \KMi{in} arbitrary orders for autonomous ODEs possessing quasi-homogeneity in an asymptotic sense\KMm{, as well as their with their justification, and parameter dependence under mild assumptions}.
\KMk{Through our proposed methodology with several examples, we see that}
\begin{itemize}
\item only fundamental solver of linear ODEs is necessary to derive asymptotic expansions of an arbitrary order, at least in the formal sense;
\item all possible powers of $\theta(t)$ can be extracted {\em even if the series is not a power series};
\item dependence of powers of $\theta(t)$ on parameters in systems can be clearly extracted;
\item low singular or regular behavior of solutions, corresponding to $0$-valued components in roots of the balance law\KMj{,}
 can be also extracted through the present methodology, while they have not been extracted from dynamics \KMj{at infinity} in preceding results (cf. \cite{Mat2018} whose details are briefly reviewed in Part II \cite{asym2}).
\end{itemize}
We end this paper by providing several discussions towards the extension of our present arguments and perspectives of blow-up solutions\KMj{, as well as a short introduction to Part II \cite{asym2}}.

\subsection*{Beyond Type-I Blow-Ups}
The present methodology for asymptotic expansions is based on Assumption \ref{ass-fundamental}, that is, the blow-up is assumed to be type-I.
There are another type of blow-up behavior whose singularity as a function of $\theta(t)$ is stronger than type-I blow-ups, which are often called {\em type-II} blow-ups\footnote{
This terminology is well known in blow-up studies for partial differential equations (e.g. \cite{HV1994}).
}.
In \cite{Mat2019}, dynamical aspects of type-II blow-ups and {\em grow-ups}; divergence of solutions with $t_{\max} = \infty$, are studied, where {\em non-hyperbolic} equilibria or invariant sets \KMk{at infinity} play key roles in characterization of such asymptotic behavior.
We have seen in the present arguments that {\em hyperbolicity} of linearized matrices for systems of our interests are essential to determine type-I blow-up behavior and their asymptotic expansions, indicating that non-hyperbolicity can induce completely different situation in the study of asymptotic expansions.
Moreover, according to \cite{Mat2019}, type-II blow-up mechanism has various aspects and hence studies for individual systems should be collected towards a comprehensive understandings of asymptotic expansions of blow-ups beyond type-I.

\KMj{
\subsection*{Short Introduction to Part II \cite{asym2}}

We briefly introduce arguments in Part II \cite{asym2}.
As mentioned in Introduction, a dynamical characterization of blow-up solutions in terms of dynamics at infinity is proposed by the third author and his collaborators. 
In the present paper, on the other hand, a systematic methodology for calculating multi-order asymptotic expansions of blow-up solutions is proposed, implying characterization of blow-up solutions from different viewpoints (e.g., \cite{Mat2018}. See Introduction for more references).
Now we have two characterizations for identical blow-up solutions.
It is then natural to ask the correspondence between these two characterizations.
\par
In Part II \cite{asym2}, we address a correspondence of blow-up characterizations.
More precisely, we correspond roots of the balance law (Definition \ref{dfn-balance}) and eigenstructures of the blow-up power-determining matrices (Definition \ref{dfn-blow-up-power-ev}) to \lq\lq equilibria at infinity" and the associated eigenstructure of linearized systems.
The correspondence implies that algebraic objects determining asymptotic expansions characterize  dynamical properties of blow-ups, such as existence, persistence under perturbation of initial points, and blow-up rates.
As a corollary, we shall prove that {\em the balance law and blow-up power-determining matrix themselves provide a criterion of the existence of blow-up solutions which we are interested in} (remark that, throughout our arguments in the present paper, the existence of blow-ups is {\em assumed}).
The dynamical interpretation of our present methodology through arguments in Part II \cite{asym2} will give a new insight into characterizations of blow-up solutions.
}

\section*{Acknowledgements}
The essential ideas in the present paper are inspired in Workshop of Unsolved Problems in Mathematics 2021 sponsored by Japan Science and Technology Agency (JST). 
All authors appreciate organizers and sponsors of the workshop for providing us with an opportunity to create essential ideas of the present study.
KM was partially supported by World Premier International Research Center Initiative (WPI), Ministry of Education, Culture, Sports, Science and Technology (MEXT), Japan.
KM \KMk{was also} partially supported by JSPS Grant-in-Aid for Scientists (B) (No. JP21H01001).
\bibliographystyle{plain}
\bibliography{blow_up_asymptotic}

\appendix
\section{Tools}

\input{appendix_A}

\section{Proofs of results}
\input{appendix_B}


\end{document}

%% file: degree.tex
\label{section-preliminaries-1}

Our main concern is the Cauchy problem of an autonomous system of ODEs
\begin{equation}
\label{ODE-original}
{\bf y}' = \frac{d{\bf y}(t)}{dt}=f ({\bf y}(t)),\quad {\bf y}(\KMi{t_0}) = {\bf y}_0,
\end{equation}
where $t\in[\KMi{t_0},T)$ with $\KMi{t_0}<T\le\infty$, $f:\mathbb{R}^n\to\mathbb{R}^n$ is a $C^r$ function with $r\geq 2$ and ${\bf y}_0\in\mathbb{R}^n$.
In this section, we \KMj{introduce a class of ODEs which we mainly treat, and a function measuring the asymptotic degree of functions as $t\to t_{\max}$ with respect to $\theta(t)$.
We also provide several estimates of such degrees we mainly use later to obtain multi-order asymptotic expansions of blow-up solutions.}

\KMm{
\begin{rem}
In what follows, we shall use the following notations for describing asymptotic relations based on \cite{E1979}.
Consider two real continuous functions $f(\epsilon), g(\epsilon)$ defined on a small interval $(0,\epsilon_0]$ and asymptotic behavior as $\epsilon \to +0$.
\begin{itemize}
\item $f = O(g)$ iff there are constant $\bar \epsilon_0 \in (0,\epsilon_0]$ and $C > 0$ such that $|f(\epsilon)| \leq C|g(\epsilon)|$ for all $\epsilon \in (0,\bar \epsilon_0]$.
\item $f = o(g)$, or $f\ll g$ iff $\lim_{\epsilon\to +0}(f(\epsilon)/g(\epsilon))$ exists and is $0$.
\item $f\sim g$ iff $\lim_{\epsilon\to +0}(f(\epsilon)/g(\epsilon))$ exists and is $1$.
\item $f = O_s(g)$ iff $f = O(g)$ and $f \not = o(g)$.
\end{itemize}
\end{rem}
}

%
%
\subsection{Asymptotically quasi-homogeneous vector fields}
\label{sec:QH}

First of all, we review a class of vector fields in our present discussions.

\begin{dfn}[Asymptotically quasi-homogeneous vector fields, cf. \cite{D1993, Mat2018}]\rm
\label{dfn-AQH}
Let $\KMf{f_0}: \mathbb{R}^n \to \mathbb{R}$ be a function.
Let $\alpha_1,\ldots, \alpha_n$ \KMn{be nonnegative integers with $(\alpha_1,\ldots, \alpha_n) \not = (0,\ldots, 0)$ and $k > 0$.}
We say that $\KMf{f_0}$ is a {\em quasi-homogeneous function\footnote{
\KMn{In preceding studies, all $\alpha_i$'s and $k$ are typically assumed to be natural numbers.
In the present study, on the other hand, the above generalization is valid.
}
} of type $\KMh{\alpha = }(\alpha_1,\ldots, \alpha_n)$ and order $k$} if
\begin{equation*}
f_0( \KMh{s^{\Lambda_\alpha}{\bf x}} ) = s^k \KMf{f_0}( \KMh{{\bf x}} )\quad \text{ for all } {\bf x} = (x_1,\ldots, x_n)^T \in \mathbb{R}^n \text{ and } s >0,
\end{equation*}
where\KMk{\footnote{
Throughout the rest of this paper, \KMl{real positive numbers or functions to the matrices are denoted} in the similar manner.
}}
\begin{equation*}
\KMh{\Lambda_\alpha =  {\rm diag}\left(\alpha_1,\ldots, \alpha_n\right),\quad s^{\Lambda_\alpha}{\bf x} = (s^{\alpha_1}x_1,\ldots, s^{\alpha_n}x_n)^T.}
\end{equation*}
Next, let $X = \sum_{i=1}^n f_i({\bf x})\frac{\partial }{\partial x_i}$ be a continuous vector field on $\mathbb{R}^n$.
We say that $X$, or simply $f = \KMg{(f_1,\ldots, f_n)^T}$ is a {\em quasi-homogeneous vector field of type $\KMh{\alpha = }(\alpha_1,\ldots, \alpha_n)$ and order $k+1$} if each component $f_i$ is a quasi-homogeneous function of type $\KMh{\alpha}$ and order $k + \alpha_i$.
\par
Finally, we say that $X = \sum_{i=1}^n f_i({\bf x})\frac{\partial }{\partial x_i}$, or simply $f$ is an {\em asymptotically quasi-homogeneous vector field of type $\KMh{\alpha = }(\alpha_1,\ldots, \alpha_n)$ and order $k+1$ at infinity} if \KMf{there is a quasi-homogeneous vector field  $f_{\alpha,k} = \KMi{(f_{i; \alpha,k})_{i=1}^n}$ of type $\KMh{\alpha}$ and order $k+1$ such that}
\begin{equation*}
\KMm{f_i( s^{\Lambda_\alpha}{\bf x} ) - s^{k+\alpha_i} f_{i;\alpha,k}( {\bf x} ) = o(s^{k+\alpha_i})}
 \end{equation*}
\KMm{as $s\to +\infty$} uniformly for \KMf{${\bf x} = (x_1,\ldots, x_n)\in S^{n-1} \equiv \{{\bf x}\in \mathbb{R}^n \mid \sum_{i=1}^n x_i^2 = 1\}$}.
\end{dfn}

A fundamental property of quasi-homogeneous functions and vector fields is reviewed here.
\begin{lem}\label{temporary-label2}
A \KMl{$C^1$} quasi-homogenous function $\KMi{f_0}$ of type $(\alpha_1,\ldots,\alpha_n)$ and order $k$ satisfies the following differential equation:
\begin{equation}\label{temporary-label1}
\sum_{l=1}^n \alpha_l y_l \frac{\partial \KMi{f_0}}{\partial y_l}({\bf y}) = k \KMi{f_0}({\bf y}).
\end{equation}
This equation is rephrased as
\begin{equation*}
(\nabla_{\bf y}\KMi{f_0}({\bf y}))^T \Lambda_{\alpha} {\bf y} = k \KMi{f_0}({\bf y}).
\end{equation*}
\end{lem}

\begin{proof}
Differentiating the identity
\[
\KMi{f_0}(s^{\Lambda_\alpha}{\bf y}) = s^k \KMi{f_0}({\bf y})
\]
in $s$ and put $s=1$, we obtain
the desired equation \eqref{temporary-label1}.
\end{proof}
The same argument yields that, for any quasi-homogenenous function $\KMi{f_0}$ of type $\alpha = (\alpha_1,\ldots, \alpha_n)$ and order $k$, 
\begin{equation*}
\sum_{\KMk{l}=1}^n \alpha_\KMk{l} s^{\alpha_\KMk{l}}y_\KMk{l} \frac{\partial \KMi{f_0}}{\partial y_l}(s^{\Lambda_\alpha}{\bf y}) = ks^k \KMi{f_0}({\bf y})
\end{equation*}
and 
\begin{equation*}
\sum_{\KMk{l}=1}^n \alpha_\KMk{l} (\alpha_\KMk{l} - 1) s^{\alpha_\KMk{l}}y_\KMk{l} \frac{\partial \KMi{f_0}}{\partial y_l}(s^{\Lambda_\alpha}{\bf y}) + \sum_{j,l = 1}^n \alpha_j \alpha_l (s^{\alpha_j}y_j) (s^{\alpha_l}y_l) \frac{\partial^2 \KMi{f_0}}{\partial y_j \partial y_l}(s^{\Lambda_\alpha}{\bf y})  = k(k-1)s^k \KMi{f_0}({\bf y})
\end{equation*}
for any ${\bf y}\in \mathbb{R}^n$.
In particular, each partial derivative satisfies
\begin{equation}
\label{order-QH-derivatives}
\frac{\partial f_0}{\partial y_l}(s^{\Lambda_\alpha}{\bf y}) = O\left(s^{k-\alpha_\KMi{l}} \right),\quad 
\frac{\partial^2 f_0}{\partial y_j \partial y_l}(s^{\Lambda_\alpha}{\bf y}) = O\left(s^{k-\alpha_j - \alpha_\KMi{l}} \right)
\end{equation}
as $s\to 0, \infty$ for any ${\bf y}\in \mathbb{R}^n$\KMl{, as long as $f_0$ is $C^2$ in the latter case}.
In particular, \KMn{for any fixed ${\bf y}\in \mathbb{R}^n$,} both derivatives are $O(1)$ as $s\to 1$.
\KMm{In other words, we have quasi-homogeneous relations for partial derivatives in the above sense.}

\KMh
{
\begin{lem}
\label{lem-identity-QHvf}
A quasi-homogeneous vector field $f=(f_1,\ldots,f_n)$ of 
type $\KMh{\alpha = }(\alpha_1,\ldots,\alpha_n)$ and order $k+1$ satisfies the following differential equation:
\begin{equation}
\label{temporary-label3}
\sum_{l=1}^n \alpha_l y_l \frac{\partial f_i}{\partial y_l}({\bf y}) = (k+\alpha_i) f_i({\bf y}) \qquad (i=1,\ldots,n).
\end{equation}
\KMi{This} equation can be rephrased as
\begin{equation}\label{temporay-label4}
(D f)({\bf y}) \Lambda_\alpha \mathbf{y} = \left( k I+ \Lambda_\alpha \right) f({\bf y}).
\end{equation}
\end{lem}
\begin{proof}
By Lemma~\ref{temporary-label2}, 
we obtain \eqref{temporary-label3}.
For \eqref{temporay-label4},
we recall that $Df=(\frac{\partial f_i}{\partial y_l})$ is Jacobian matrix, 
\KMi{while} $kI + \Lambda_\alpha$ is the diagonal matrix with diagonal entries $k+\alpha_i$.
Finally $\mathbf{y} = (y_1,\ldots,y_n)^T$ is the column vector,
so that the left-hand side of \eqref{temporay-label4} is the product of two matrices and one column vector.
\end{proof}
}

Throughout successive sections, consider an (autonomous) $C^r$ vector field (\ref{ODE-original}) with $r\geq 2$,
where $f: \mathbb{R}^n \to \mathbb{R}^n$ is asymptotically quasi-homogeneous of type $\alpha = (\alpha_1,\ldots, \alpha_n)$ and order $k+1$ at infinity.

\subsection{The asymptotic degree of functions as $t\to t_{\max}$}
\label{section-degree}
\KMj{
Our concern here is to introduce a quantity to measure the asymptotic behavior of functions as $t\to t_{\max}-0$ so that a given series of functions, in particular describing blow-up solutions, is verified to be an asymptotic series as $t\to t_{\max}-0$.
To this end, we shall introduce the asymptotic degree of continuous scalar- and vector-valued functions as $t\to t_{\max}-0$.
It should be a generalization of degree of polynomials and rational functions.
The following notion indeed generalizes such well-known degree {\em only in a neighborhood of $t_{\max}$ with $t < t_{\max}$}.
}

\begin{dfn}[Asymptotic degree $\deg_\theta$]\rm
For a given scalar-valued function $h = h(t)$, let
\begin{equation}
\label{Ih}
\KMm{I_h := \{\gamma \in \mathbb{R} \mid h(t) = o(\theta(t)^\gamma) \text{ as } t\to t_{\max}-0\}}
\end{equation}
\KMm{and}
\begin{equation}
\label{ord}
\KMm{\deg_\theta(h) := \sup I_h.}
\end{equation}
For a given $n$-vector-valued function ${\bf h} = {\bf h}(t) \equiv (h_1(t), \ldots, h_n(t))^T$, define
\begin{equation*}
\KMj{\deg_\theta}({\bf h}) := \min_{i=1,\ldots, n} \KMj{\deg_\theta}(h_i).
\end{equation*}
\end{dfn}
\KMj{
The value $\deg_\theta$ measures the behavior of functions with respect to the power function $\theta(t)^\gamma$ as $t\to t_{\max}-0$, and generalizes the degree of {\em polynomial} functions of $\theta(t)$.
Indeed, $\deg_\theta$ is defined as {\em suprema} of $\gamma$ satisfying the corresponding asymptotic estimates because various elementary functions such as logarithmic ones are also taken into account. 
For example, as $s\to +0$, $f_0(s) = \ln s$ is $o(s^a)$ for any $a<0$, whereas $f_0(s) \not = O(1)$.
On the other hand, the above definition is consistent with the asymptotic relation of the identically zero function\footnote{
This consistency is also valid for \KMi{exponentially decaying} functions, which can be applied to extending the present results in future works.
}. 
Indeed, if $h(t) \equiv 0$, the definition implies that $\KMj{\deg_\theta}(h) = \infty$ and is consistent with the asymptotic relation $\lim_{t\to t_{\max}-0}h(t)\theta(t)^{-\gamma} = 0$ for any $\gamma \in \mathbb{R}$.
}

\subsubsection{Fundamental properties}
Fundamental properties of $\KMj{\deg_\theta}$ used in the successive arguments are summarized here.
\KMm{First, we have the following properties by definition.
\begin{lem}
\label{lem-deg-basic}
Let $I_h$ be the set given in (\ref{Ih}) for a scalar function $h(t)$.
\begin{enumerate}
\item \KMn{If $I_h \not = \mathbb{R}$, it} is the interval of the form $(-\infty, \gamma_h)$ or $\KMn{(-\infty, \gamma_h)}$ for some $\gamma_h \in \mathbb{R}$.
\item If $f(t) \leq g(t)$ holds for all $t$ near and less than $t_{\max}$, then $\deg_\theta(g) \leq \deg_\theta(f)$.
\end{enumerate}
\end{lem}
\begin{proof}
\begin{enumerate}
\item $\gamma\in I_h$ and $\gamma' < \gamma$ imply $\gamma' \in I_h$.
\item The assumption implies $I_g \subset I_f$, and hence the statement holds.
\end{enumerate}
\end{proof}}

\KMm{The degree of several elementary functions are collected \KMm{below}, which are fully applied to estimating the orders of functions in asymptotic expansions of blow-up solutions.}
\begin{ex}
\label{ex-ord}
The results below follow from the definition.
\begin{itemize}
\item $\KMj{\deg_\theta}(c) = 0$ for any constant $c \not = 0$.
\item $\KMj{\deg_\theta}(0) = +\infty$.
\item $\KMj{\deg_\theta}(\theta(t)^\gamma) = \gamma$.
\item $\ln \theta(t) = o(\theta(t)^\gamma)$  for any \KMk{$\gamma < 0$}, while $\lim_{t\to t_{\max}-0}\ln \theta(t)= -\infty$. 
Therefore $\KMj{\deg_\theta}(\ln \theta(t)) = 0$.
\item $\cos (\lambda \ln \theta(t)) = O_s(1)$ and hence $\deg_\theta (\cos (\lambda \ln \theta(t))) = 0$.
\item $\sin (\lambda \ln \theta(t)) = O_s(1)$ and hence $\deg_\theta (\sin (\lambda \ln \theta(t))) = 0$.
\end{itemize}
\end{ex}

\KMm{Next, we summarize basic several asymptotic relations based on the l'H{\^o}pital's theorem so that readers who are not familiar with asymptotic relations can easily understand the subsequent arguments.}

\KMi{
\begin{lem}[cf. \cite{BO1999}]
\label{lem-asym-poly-log}
For a given $x_0\in \mathbb{R}$ and a real-valued function $f = f(x)$ continuous in $(x_0-\delta, x_0)$ for some $\delta > 0$, \KMl{let ${\bf F}_f(x)$ be a primitive function of $f(x)$. Then} 
the following asymptotic relations hold as $x\to x_0-0$;
if $f(x) \sim a(x_0-x)^{-b}(\ln (x_0-x))^M$ with $M\in \mathbb{Z}_{\geq 0}$,
\begin{enumerate}
\item ${\bf F}_f(x) \sim [-a/(1-b)](x_0-x)^{1-b}(\ln (x_0-x))^M$ if $b > 1$. 
\item assuming $b < 1$ and 
if ${\bf F}_f(x)$ of $f(x)$ is chosen to vanish as $x\to x_0-0$, then ${\bf F}_f(x) \sim [-a/(1-b)](x_0-x)^{1-b}(\ln (x_0-x))^M$, 
otherwise ${\bf F}_f(x) \sim C$, where $C$ is a \KMm{non-zero constant}.
\par
In the \KMl{former} case, we further obtain $\KMm{{\bf F}_f(x)} =  o((x_0-x)^{1-b \KMk{-} \epsilon})$ for any $\epsilon > 0$ as $x\to x_0-0$.
\item $\int_x^{x_0} f(\eta) d\eta \equiv \displaystyle{\left( \lim_{x\to x_0-0}{\bf F}_f(x) \right)} - {\bf F}_f(x) \sim [a/(1-b)](x_0-x)^{1-b}(\ln (x_0-x))^M$ if $b < 1$.
\item ${\bf F}_f(x) \sim \frac{a}{M+1}(\ln (x_0-x))^{M+1}$ if $b=1$.
\end{enumerate}
\end{lem}
}

\begin{proof}
The l'H{\^o}pital's theorem and the assumption yield
\begin{equation*}
{\bf F}_f(x) \sim a\int^x (x_0-\eta)^{-b}(\ln (x_0-\eta))^M d\eta \quad \text{ as }\quad x\to x_0-0,
\end{equation*}
\KMl{provided that both sides diverge, or converge to $0$.}
The right-hand side is given as 
\begin{equation*}
\KMk{a}\left\{ \sum_{l=0}^M \frac{M!}{(M-l)!} \left(\frac{-1}{1-b} \right)^{l+1} (\ln (x_0-x))^{M-l} \right\} (x_0-x)^{1-b}
+ C
\end{equation*}
with the integral constant $C$, provided $b\not = 1$.
In particular, we have
\begin{equation*}
a\int^x (x_0-\eta)^{-b}(\ln (x_0-\eta))^M d\eta \sim \frac{\KMk{-a}}{1-b}(x_0-x)^{1-b}(\ln (x_0-x))^M + C.
\end{equation*}
\KMl{Note that the right-hand side converges to $C$ as $x\to x_0-0$ if $b < 1$.}
\par
\KMl{Now we prove our statements.}
If $b > 1$, the function $(x_0-x)^{1-b}(\ln (x_0-x))^M$ diverges as $x\to x_0-0$ and hence 
\begin{equation*}
a\int^x (x_0-\eta)^{-b}(\ln (x_0-\eta))^M d\eta \sim \frac{\KMk{-a}}{1-b}(x_0-x)^{1-b}(\ln (x_0-x))^M,
\end{equation*}
which shows the first statement.
On the other hand, if $b < 1$, then
\begin{equation*}
a\int^x (x_0-\eta)^{-b}(\ln (x_0-\eta))^M d\eta \sim C
\end{equation*}
provided $C\not = 0$, because $(x_0-x)^{1-b}(\ln (x_0-x))^M$ converges to $0$ as $x\to x_0-0$ and hence the constant $C$ is dominant as $x\to x_0-0$, while $(x_0-x)^{1-b}(\ln (x_0-x))^M$ becomes dominant when \KMl{${\bf F}_f(x)\to 0$ as $x\to x_0-0$}.
Moreover, we have
\begin{equation*}
\lim_{x\to x_0-0}\frac{(x_0-x)^{1-b}(\ln (x_0-x))^M}{(x-x_0)^{1-b \KMl{-}\epsilon}} = \lim_{x\to x_0-0}\frac{(\ln (x_0-x))^M}{(x_0-x)^{\KMl{-}\epsilon}} = 0
\end{equation*}
for any $\epsilon > 0$.
As a summary, we have the second statement.
\par
\bigskip
Consider the integral $\int_x^{x_0} f(\eta) d\eta$ instead\KMl{, assuming $b < 1$}.
By the l'H{\^o}pital's theorem, we have
\begin{equation*}
\int_x^{x_0} f(\eta) d\eta \sim a\int_x^{x_0} (x_0-\eta)^{-b}(\ln (x_0-\eta))^M d\eta\quad \text{ as }\quad x\to x_0-0.
\end{equation*}
The integral of the right-hand side is
\begin{equation*}
a\int_x^{x_0} (x_0-\eta)^{-b}(\ln (x_0-\eta))^M d\eta = 
-\KMk{a} \left\{ \sum_{l=0}^M \frac{M!}{(M-l)!} \left(\frac{-1}{1-b} \right)^{l+1} (\ln (x_0-x))^{M-l} \right\} (x_0-x)^{1-b},
\end{equation*}
because $b < 1$.
Substituting this identity into the asymptotic relation, we have
\begin{equation*}
\int_x^{x_0} f(\eta) d\eta \sim \frac{a}{1-b}(x_0-x)^{1-b}(\ln (x_0-x))^M\quad \text{ as }\quad x\to x_0-0,
\end{equation*}
namely, we have the third statement.
\par
\bigskip
We go back to \KMl{the primitive function ${\bf F}_f(x)$}. 
When $b=1$, direct calculation yields
\begin{equation*}
a\int^x (x_0-\eta)^{-1}(\ln (x_0-\eta))^M d\eta = \frac{-a}{M+1}(\ln (x_0-x))^{M+1} + C
\end{equation*}
with a constant $C$. 
Therefore we have
\begin{equation*}
a\int^x (x_0-\eta)^{-1}(\ln (x_0-\eta))^M d\eta \sim \frac{-a}{M+1}(\ln (x_0-x))^{M+1}\quad \text{ as }\quad x\to x_0-0
\end{equation*}
and hence the last statement is proved.
\end{proof}

Based on the above arguments, we have the following basic estimates of $\deg_\theta$ involving integrals, depending on powers.
\begin{lem}
\label{lem-theta-integral}
For any $t_0 < t_{\max}$, the following results hold.
\begin{enumerate}
\item If $\gamma \leq -1$, 
\begin{equation*}
\deg_{\theta}\left( \int_{t_0}^t \theta(\eta)^{\gamma}d\eta \right) = \gamma + 1.
\end{equation*}
\item If $\gamma > -1$, 
\begin{equation*}
\deg_{\theta}\left( \int_{t}^{t_{\max}}\theta(\eta)^{\gamma}d\eta \right) = \gamma + 1.
\end{equation*}
\item If $\gamma > -1$, there is a constant $C\in \mathbb{R}$ such that
\begin{equation*}
\deg_\theta \left( \int_{t_0}^t \KMn{\theta(\eta)^\gamma } d\eta + C\right) \geq \KMn{\gamma} + 1.
\end{equation*}
\end{enumerate}
\end{lem}

\begin{proof}
\begin{enumerate}
\item If $\gamma < -1$,
\begin{align*}
\frac{1}{\theta(t)^{\gamma+1}}\int_{t_0}^t \theta(\eta)^{\gamma}d\eta
&=\frac{1}{\theta(t)^{\gamma+1}}\frac{-1}{\gamma+1}(\theta(t)^{\gamma+1}-\theta(t_0)^{\gamma+1})\\
&=\frac{-1}{\gamma+1}(1-(\theta(t)/ \theta(t_0))^{-(\gamma+1)})\\
&\to \frac{-1}{\gamma+1}\quad \text{as}\quad t \to t_{\max}-0.
\end{align*}
If $\gamma = -1$, 
\begin{equation*}
\int_{t_0}^t \theta(\eta)^{\gamma}d\eta = \ln \frac{\theta(t_0)}{\theta(t)},
\end{equation*}
which implies, for any $\tilde \gamma > 0$,
\begin{equation*}
\theta(t)^{\tilde \gamma} \int_{t_0}^t \theta(\eta)^{\gamma}d\eta = \theta(t)^{\tilde \gamma} \ln \frac{\theta(t_0)}{\theta(t)} \to 0
\end{equation*}
holds as $t\to t_{\max}-0$. 
\item If $\gamma > -1$, 
\begin{align*}
\frac{1}{\theta(t)^{\gamma+1}} \int_{t}^{t_{\max}}\theta(\eta)^{\gamma}d\eta = \frac{1}{\theta(t)^{\gamma+1}}\frac{1}{\gamma+1}(\theta(t)^{\gamma+1}-\theta(t_{\max})^{\gamma+1})
=\frac{1}{\gamma+1}.
\end{align*}
\item Similar to the previous arguments, we have
\begin{align*}
\int_{t_0}^{t}\theta(\eta)^{\gamma}d\eta = \frac{1}{\gamma+1}(\theta(t_0)^{\gamma+1}-\theta(t)^{\gamma+1}) \to \frac{\theta(t_0)^{\gamma+1}}{\gamma+1}
\end{align*}
as $t\to t_{\max}-0$, which yields
\begin{equation*}
\deg_{\theta}\left( \int_{t_0}^t \theta(\eta)^{\gamma}d\eta \right) = \begin{cases}
\gamma + 1 & \text{if } \theta(t_0) = 0, \\
0 & \text{otherwise}.
\end{cases}
\end{equation*}
\end{enumerate}
Let $\Theta_\gamma(t)$ be a primitive function of $\theta(t)^\gamma$.
Note that, for any constant $C\in \mathbb{R}$, $\Theta_\gamma(t) + C$ is also a primitive function of $\theta(t)^{\gamma}$, and that $C$ can be chosen so that the corresponding primitive function $\Theta_\gamma^0(t)$ converges to $0$ as $t\to t_{\max}-0$.
In the present case, we see that
\begin{equation}
\label{primitive-theta-0}
\deg_{\theta}\left( \Theta_\gamma^0(t) \right) = \gamma + 1\quad \text{ letting }\quad \Theta_\gamma^0(t) = -(\gamma+1)^{-1}\theta(t)^{\gamma+1}.
\end{equation}
\end{proof}

The fundamental properties of $\KMj{\deg_\theta}$ are stated below.
\begin{prop}
\label{prop-ord-fundamental}
Consider real-valued continuous functions $f_0(t), g_0(t)$ defined in \KMm{$[t_0, t_{\max})$} with $t_{\max} < \infty$.
Let $\gamma_{f_0} = \KMj{\deg_\theta}(f_0(t))$ and $\gamma_{g_0} = \KMj{\deg_\theta}(g_0(t))$.
\begin{enumerate}
\item If $f_0(t) \sim a\theta(t)^{-b}(\ln \theta(t))^M$ \KMm{with $a\not = 0$} as $t\to t_{\max}-0$, 
we have
\begin{equation*}
\KMj{\deg_\theta}\left(\int^t f_0(\eta) d\eta\right) = \begin{cases}
	0 & \text{if $b < 1$ and the integral constant $C$ is not $0$,}\\
	1-b & \text{otherwise.}
\end{cases}
\end{equation*}
\item $\deg_\theta (f_0(t) + g_0(t)) \geq \min\{\gamma_{f_0}, \gamma_{g_0}\}$.
\item $\deg_\theta (f_0(t)g_0(t)) \geq \gamma_{f_0} + \gamma_{g_0}$.
\item If $\gamma_{f_0} \leq -1$, then $\deg_\theta\left(\int_{t_0}^t f_0(\eta) d\eta \right) \geq \gamma_{f_0}+1$.
\item If $\gamma_{f_0} > -1$, then $\KMj{\deg_\theta}\left(\int_t^{t_{\max}} f_0(\eta) d\eta\right) \geq \gamma_{f_0}+1$.
\item \KMl{If $\gamma_{f_0} > -1$, for the primitive function ${\bf F}_{f_0}(t)$ of $f_0(t)$ satisfying\footnote{
\KMl{If $\gamma_{f_0} > -1$, the improper integral converges to a finite value. 
In particular, the limit $\displaystyle{\lim_{t\to t_{\max}-0}}{\bf F}_{f_0}(t)$ exists.}
} $\displaystyle{\lim_{t\to t_{\max}-0}}{\bf F}_{f_0}(t) = 0$, we have $\deg_\theta\left({\bf F}_{f_0}(t) \right) \geq \gamma_{f_0}+1$.}
\end{enumerate}
\end{prop}

\KMm{
\begin{rem}
When $\gamma_{f_0} > -1$ we have
\begin{equation*}
\deg_\theta \left( \int_{t_0}^t f_0(\eta) d\eta  \right) = 0
\end{equation*}
unless integral constants are chosen appropriately (cf. Lemmas \ref{lem-asym-poly-log}-2 and \ref{lem-theta-integral}-3). 
\end{rem}}

\begin{proof}
Statement 1 is just the paraphrase of Lemma \ref{lem-asym-poly-log}.
\par
Next we prove Statement 2.
If $\gamma\in I_{f_0}\cap I_{g_0}$, then
\begin{equation*}
\lim_{t\to t_{\max}} \frac{f_0(t)+g_0(t)}{\theta(t)^{\gamma}} = \lim_{t\to t_{\max}} \frac{f_0(t)}{\theta(t)^{\gamma}} + \lim_{t\to t_{\max}} \frac{g_0(t)}{\theta(t)^{\gamma}} = 0 + 0 = 0,
\end{equation*}
which implies that $\gamma \in I_{f_0 + g_0}$.
In other words, $I_{f_0}\cap I_{g_0} \subset I_{f_0 + g_0}$.
The statement then follows by taking the supremum on both sides.
\par
\bigskip
Next we prove Statement 3. 
\KMm{Let $\gamma_{f_0} \in I_{f_0}$ and $\gamma_{g_0} \in I_{g_0}$.
Our claim is to prove $\gamma_{f_0} + \gamma_{g_0} \in I_{f_0 g_0}$, which follows from the relation
\begin{equation*}
\lim_{t\to t_{\max}} \frac{f_0(t) g_0(t)}{ \theta(t)^{\gamma_{f_0} + \gamma_{g_0}}} 
= \lim_{t\to t_{\max}} \frac{f_0(t)}{ \theta(t)^{\gamma_{f_0}}} \lim_{t\to t_{\max}} \frac{g_0(t)}{ \theta(t)^{\gamma_{g_0}}} = 0\cdot 0 = 0,
\end{equation*}
which implies that
\begin{equation*}
I_{f_0} + I_{g_0} \equiv \{ \gamma_1 + \gamma_2 \mid \gamma_1\in I_{f_0},\, \gamma_2\in I_{g_0}\} \subset I_{f_0g_0}.
\end{equation*}
The statement then follows by taking the supremum on both sides.}

\par
\bigskip
Next, we prove Statement 4.
\KMm{For all $\gamma < \gamma_{f_0}$, the relation
\begin{equation*}
\lim_{t\to t_{\max}-0}\frac{f_0(t)}{\theta(t)^{\gamma}} = 0
\end{equation*}
implies that
\begin{equation*}
|f_0(t)| \leq \KMn{\theta(t)^\gamma}\quad \text{ near }\quad t = t_{\max}.
\end{equation*}
Then it follows from Lemma \ref{lem-deg-basic} that
\begin{equation*}
\deg_{\theta}\left( \int_{t_0}^t f_0(\eta) d\eta \right) \geq \deg_{\theta}\left( \int_{t_0}^t |f_0(\eta)| d\eta \right) \geq \deg_{\theta}\left( \int_{t_0}^t \theta(t)^\gamma d\eta \right) = \gamma + 1,
\end{equation*}
which proves the statement.
Note that the rightmost equality follows from Lemma \ref{lem-theta-integral}.}
\par
\bigskip
Next, we prove Statement 5.
\KMm{For all $\gamma \in (-1, \gamma_{f_0})$, the relation
\begin{equation*}
\lim_{t\to t_{\max}-0}\frac{f_0(t)}{\theta(t)^{\gamma}} = 0
\end{equation*}
implies that
\begin{equation*}
\deg_{\theta}\left( \int_t^{t_{\max}} f_0(\eta) d\eta \right) \geq \deg_{\theta}\left( \int_t^{t_{\max}} |f_0(\eta)| d\eta \right) \geq \deg_{\theta}\left( \int_t^{t_{\max}} \theta(t)^\gamma d\eta \right) = \gamma + 1,
\end{equation*}
which proves the statement.}
\par
\bigskip
\KMm{
Finally, we prove the statement 6.
Consider $\gamma < \gamma_{f_0}$. 
Let $\Theta_\gamma^0(t)$ be the primitive function of $\theta(t)^{\gamma}$ given in (\ref{primitive-theta-0}).
By assumption, the l'H{\^o}pital's theorem can be applied to ${\bf F}_{f_0}$ and $\Theta_\gamma^0(t)$, indicating
\begin{equation*}
\lim_{t\to t_{\max}-0}\frac{{\bf F}_{f_0}(t)}{\Theta_\gamma^0(t)} = \lim_{t\to t_{\max}-0}\frac{f_0(t)}{\theta(t)^{\gamma}} = 0.
\end{equation*}
Because $\gamma < \gamma_{f_0}$ can be chosen arbitrarily, we have $\deg_\theta({\bf F}_{f_0}(t)) \geq \gamma_{f_0} + 1$.
}
\end{proof}

\subsubsection{Degree and the fundamental matrices of Euler-type homogeneous ODEs}
\KMj{
As seen in Section \ref{section-asym}, the fundamental matrix of the following homogeneous system is fully applied to asymptotic expansions of blow-up solutions:
\begin{equation*}
\frac{d{\bf v}}{dt} = \theta(t)^{-1}A{\bf v},\quad A\in M_n(\mathbb{R}).
\end{equation*}
The fundamental matrix is given by $\theta(t)^{-A} = \exp(-A\ln \theta(t)) = (\theta(t)^A)^{-1}$.}
Any matrix $A\in M_n(\mathbb{R})$ and the associated fundamental matrix $\theta(t)^{A}$ has the following real Jordan normal forms, respectively:
\begin{align}
\notag
P^{-1}AP &= J_{\KMm{1}}(\lambda_1) \oplus \cdots \oplus J_{\KMm{d_r}}(\lambda_{d_r}) \oplus K_{\KMm{1}}(\lambda_{re,1}, \lambda_{im,1}) \oplus \cdots \oplus K_{\KMm{d_c}}(\lambda_{re, d_c}, \lambda_{im, d_c}),\\
\label{Jordan-theta}
P^{-1}\theta(t)^A P &= \theta(t)^{J_{\KMm{1}}(\lambda_1)} \oplus \cdots \oplus \theta(t)^{J_{\KMm{d_r}}(\lambda_{d_r})} \oplus \theta(t)^{K_{\KMm{1}}(\lambda_{re,1}, \lambda_{im,1})} \oplus \cdots \oplus \theta(t)^{K_{\KMm{d_c}}(\lambda_{re, d_c}, \lambda_{im, d_c})},
\end{align}
where 
\begin{align*}
&B_1 \oplus B_2 \oplus \cdots \oplus B_m \equiv \begin{pmatrix}
B_1 & & & O \\
& B_2 & & \\
& & \ddots & \\
O & & & B_m
\end{pmatrix},\quad 
B^{\oplus s} = \underbrace{B\oplus \cdots \oplus B}_{s}
\end{align*}
for any real squared matrices $B, B_1, \ldots, B_m$, and the above real Jordan blocks are expressed by
\begin{align}
\label{Jordan-real}
\KMm{J_l}(\lambda) &= \lambda I_{\KMm{m_l}} + N_{\KMm{m_l}},\quad
	N_{\KMm{m_l}} = \begin{pmatrix}
0 & 1 & & & O \\
 & 0 & 1 & & \\
 &  & 0 & \ddots & \\
 &  &  & \ddots & 1 \\
O &  &  & & 0 \\
\end{pmatrix}\in M_{\KMm{m_l}}(\mathbb{R}),\\
\label{Jordan-complex}
K_{\KMm{l}}(\lambda_{re}, \lambda_{im}) &= R(\lambda_{re}, \lambda_{im})^{\oplus \KMm{s_l}} + N_{2\KMm{s_l}}^2,\quad
R(\lambda_{re}, \lambda_{im}) = \begin{pmatrix}
\lambda_{re} & \lambda_{im}\\
-\lambda_{im} & \lambda_{re}
\end{pmatrix}.
\end{align}
Here $I_m$ denotes the $m$-dimensional identity matrix.
The scalar function $\theta(t)$ to the above \KMm{matrices} are given as follows:
\begin{align*}
\theta(t)^{\KMm{J_l}(\lambda)} &= 
\theta(t)^{\lambda} \left(I_{\KMm{m_l}} + \ell(t)N_{\KMm{m_l}} + \frac{\ell(t)^2}{2!}N_{\KMm{m_l}}^2 + \cdots + \frac{\ell(t)^{\KMm{m_l}-1}}{(\KMm{m_l}-1)!}N_{\KMm{m_l}}^{\KMm{m_l}-1}\right),\\
\theta(t)^{K_{\KMm{l}}(\lambda_{re}, \lambda_{im})} &= \theta(t)^{\lambda_{re}} R(\cos(\lambda_{im}\ell(t)), \sin(\lambda_{im}\ell(t)))^{\oplus \KMm{s_l}} \\
	&\quad \cdot \left(I_{2\KMm{s_l}} + \ell(t)N_{2\KMm{s_l}}^2 + \frac{\ell(t)^2}{2!}N_{2\KMm{s_l}}^4 + \cdots + \frac{\ell(t)^{\KMm{s_l}-1}}{(\KMm{s_l}-1)!}N_{2\KMm{s_l}}^{2(\KMm{s_l}-1)}\right),\\
\ell(t) &= \KMk{\ln (\theta(t))}.
\end{align*}
This Jordan normal form shows that the matrix $\theta(t)^A$, as well as $\theta(t)^{-A}$, consists of functions $\theta(t)^\gamma$, $\ln \theta(t)$, $\cos (\lambda \ln \theta(t))$ and $\sin (\lambda \ln \theta(t))$.
The degree $\KMj{\deg_\theta}$ of functions appeared in the above fundamental matrix are \KMm{collected in Example \ref{ex-ord}.}
Combining with Proposition \ref{prop-ord-fundamental}, we obtain the following inequalities.
\begin{prop}
\label{prop-ord-fund-matrix}
For any real $m$-dimensional and $2m$-dimensional vector-valued continuous functions ${\bf a}_m(t)$ and ${\bf b}_{2m}(t)$ defined in $(t_0, t_{\max})$ with $t_{\max} < \infty$ and any constant matrix $P\in M_m(\mathbb{R})$, we have
\begin{align}
\label{ord-estimate-const}
\deg_\theta\left( P{\bf a}_m(t) \right) &\geq \KMj{\deg_\theta}\left( {\bf a}_m(t) \right),\\
\label{ord-estimate-Jordan-R}
\deg_\theta \left( \theta(t)^{ \KMm{J}(\lambda)}{\bf a}_m(t) \right) &\geq \lambda + \deg_\theta \left( {\bf a}_m(t) \right),\\
\label{ord-estimate-Jordan-C}
\deg_\theta \left( \theta(t)^{\KMm{K}(\lambda_{re}, \lambda_{im})}{\bf b}_{2m}(t) \right) &\geq \lambda_{re} + \KMj{\deg_\theta}\left( {\bf b}_{2m}(t) \right),
\end{align}
\KMm{where $J(\lambda)$ and $K(\lambda_{re}, \lambda_{im})$ are $m$-dimensional and $2m$-dimensional Jordan matrices of the form (\ref{Jordan-real}) and (\ref{Jordan-complex}), respectively.}
\end{prop}

\begin{proof}
By Example \ref{ex-ord}, the values of $\KMj{\deg_\theta}$ for each entry of the constant matrix $P$ is either $+\infty$ or $0$.
The first statement then follows from the combination of statements in Proposition \ref{prop-ord-fundamental}.
Similarly, the values of $\KMj{\deg_\theta}$ for each entry in $\theta(t)^{\KMm{J}(\lambda)}$ are identically either $+\infty$ or $\lambda$, while those for each entry in $\theta(t)^{\KMm{K}(\lambda_{re}, \lambda_{im})}$ are identically either $+\infty$ or $\lambda_{re}$.
The conclusion then follows from the combination of statements in Proposition \ref{prop-ord-fundamental}.
\end{proof}

We further consider the asymptotic degree of the fundamental matrix on invariant subspaces of $A$.
In general, eigenvalues of $A$ with positive (resp. negative) real parts associate the invariant projector $P_+$ (resp. $P_-$) such that $AP_- = P_- A$ (resp. $AP_+ = P_+A$) and that ${\rm Image}\,(P_+)$ (resp.  ${\rm Image}\,(P_-)$) is generated by (generalized) eigenvectors $A$ associated with eigenvalues with (resp. negative) real parts\footnote{
The invariant projector onto the invariant subspace associated with purely imaginary eigenvalues is also provided in general, but it is out of our interest in the present argument.
}.
Now $\mathbb{R}^n$ possesses the following decomposition into the invariant subspaces associated with $\{\lambda \in {\rm Spec}(A) \mid {\rm Re}\,\lambda > 0\}$, $\{\lambda \in {\rm Spec}(A) \mid {\rm Re}\,\lambda = 0\}$ and $\{\lambda \in {\rm Spec}(A) \mid {\rm Re}\,\lambda < 0\}$, respectively \KMf{(e.g. \cite{Cha2011, Rob})}:
\begin{equation}
\label{dec-Rn}
\mathbb{R}^n = \mathbb{E}_A^u \oplus  \mathbb{E}_A^c \oplus \mathbb{E}_A^s.
\end{equation}
Let $P_+$ be the invariant projector onto $\mathbb{E}_A^u$ and $P_-$ be the invariant projector onto $\mathbb{E}_A^s$.
In particular, we have
\begin{equation*}
AP_\pm = P_\pm A,\quad \theta(t)^A P_{\pm} = P_{\pm}\theta(t)^A
\end{equation*}
for any $t < \KMm{t_{\max}}$ under our considerations.
\par
Our interest here is the asymptotic degree of vector-valued functions of the form $\theta(t)^{-A}P_\pm {\bf h}(t)$.
The following estimates are mainly used to construct asymptotic expansions of blow-ups.

\begin{prop}
\label{prop-deg-integral-theta-F}
Let ${\bf h}(t)$ be a continuous $n$-vector-valued function with $\deg_\theta({\bf h}(t)) = \gamma > -1$.
Then the following inequalities hold:
\begin{align}
\label{deg-integral-theta-F-minus}
\deg_\theta\left(P_-  \left( {\bf C} +  \int_{t_0}^t \left( \frac{\theta(\eta)}{\theta(t)} \right)^{A}  {\bf h}(\eta) d\eta \right) \right) &\geq \min \left\{ \min_{\substack{\lambda\in {\rm Spec}(A),\\ {\rm Re}\,\lambda < 0}}(- {\rm Re}\,\lambda), \gamma + 1 \right\},\\
\label{deg-integral-theta-F-plus}
\deg_\theta\left( \int_{t}^{t_{\max}} \left( \frac{\theta(\eta)}{\theta(t)} \right)^{A} P_+ {\bf h}(\eta) d\eta \right) &\geq \gamma + 1,
\end{align}
where ${\bf C}\in \mathbb{R}^n$ is a constant vector.
Finally, \KMl{let $P^{-1}AP = \Lambda$ be the Jordan \KMm{normal form} of the matrix $A$ and
${\bf H}_{\bf h}(t)$ be the primitive function of $\left(\frac{\theta(t)}{\theta(t_0)}\right)^\Lambda P^{-1} P_- {\bf h}(t)$ such that}
\begin{equation}
\label{primitive-zero-power}
\KMl{\lim_{t\to t_{\max} - 0}({\bf H}_{\bf h}(t))_i = 0\quad \text{whenever}\quad \deg_\theta \left( \left\{ \left(\frac{\theta(t)}{\theta(t_0)}\right)^\Lambda P^{-1} P_- {\bf h}(t) \right\}_i \right) > -1.
}
\end{equation}
Then we have
\begin{equation}
\label{deg-integral-theta-F-minus-2}
\deg_\theta\left( \KMl{ P \left(  \frac{\theta(t)}{\theta(t_0)}\right)^{-\Lambda} P_- {\bf H}_{\bf h}(t)} \right) \geq \gamma + 1.
\end{equation}
\end{prop}

\begin{proof}
First, for fixed $\eta \in [t_0, t]$, apply the \KMm{normalization} to $(\theta(\eta) / \theta(t))^{A}$ following (\ref{Jordan-theta}):
\begin{align*}
P^{-1}\left(\frac{\theta(\eta)}{\theta(t)}\right)^{A} P &= \left(\frac{\theta(\eta)}{\theta(t)}\right)^{J_{\KMm{1}}(\lambda_1)} \oplus \cdots \oplus \left(\frac{\theta(\eta)}{\theta(t)}\right)^{J_{\KMm{d_r}}(\lambda_{d_r})}\\
	&\oplus \left(\frac{\theta(\eta)}{\theta(t)}\right)^{K_{\KMm{1}}(\lambda_{re,1}, \lambda_{im,1})} \oplus \cdots \oplus \left(\frac{\theta(\eta)}{\theta(t)}\right)^{K_{\KMm{d_c}}(\lambda_{re, d_c}, \lambda_{im, d_c})}.
\end{align*}
Letting $\KMm{P_l^{(r)}}$ be the invariant projector onto the invariant subspace associated with the real eigenvalue $\lambda_l$, we have
\begin{align}
\notag
\KMl{\left(\frac{\theta(\eta)}{\theta(t)}\right)^{A}} \KMm{P_l^{(r)}} {\bf h}(\eta) &= \KMl{P}\left( 0\oplus \cdots \oplus 0 \oplus \left(\frac{\theta(\eta)}{\theta(t)}\right)^{ \KMm{J_l} (\lambda_l)} \oplus 0 \oplus \cdots \oplus 0\right) \KMl{P^{-1}} \KMm{P_l^{(r)}} {\bf h}(\eta)\\
\label{h-decomposed}
	&\equiv \KMl{P} \left(\frac{\theta(\eta)}{\theta(t)}\right)^{\overline{\KMm{J_l}(\lambda_l)}} \KMl{P^{-1}} \KMm{P_l^{(r)}} {\bf h}(\eta).
\end{align}
Similarly, letting $\KMm{P_l^{(c)}}$ be the invariant projector onto the invariant subspace associated with the complex conjugate eigenvalues $\{\KMm{\lambda_{re, l} \pm \sqrt{-1} \lambda_{im, l}}\}$, we have
\begin{align*}
\KMl{ \left(\frac{\theta(\eta)}{\theta(t)}\right)^{A}} \KMm{P_l^{(c)}} {\bf h}(\eta) &= P \left( 0\oplus \cdots \oplus 0 \oplus \left(\frac{\theta(\eta)}{\theta(t)}\right)^{K_{\KMm{l}}(\lambda_{re, l}, \lambda_{im, l})} \oplus 0 \oplus \cdots \oplus 0\right) \KMl{P^{-1}} \KMm{P_l^{(c)}} {\bf h}(\eta)\\
	&\equiv \KMl{P} \left(\frac{\theta(\eta)}{\theta(t)}\right)^{\overline{K_{\KMm{l}}(\lambda_{re, l}, \lambda_{im, l})}} \KMl{P^{-1}} \KMm{P_l^{(c)}} {\bf h}(\eta).
\end{align*}

Now consider the integral
\begin{equation}
\label{integral-ml}
\int_{t_0}^t \left(\frac{\theta(\eta)}{\theta(t)}\right)^{A} \KMm{P_l^{(r)}} {\bf h}(\eta) d\eta.
\end{equation}
The integral is also written as
\begin{align}
\notag
\int_{t_0}^t \left(\frac{\theta(\eta)}{\theta(t)}\right)^{A} \KMm{P_l^{(r)}} {\bf h}(\eta) d\eta 
	&= \int_{t_0}^{t} P \left(\frac{\theta(\eta)}{\theta(t)}\right)^{\overline{\KMm{J_l}(\lambda_l)}} P^{-1} \KMm{P_l^{(r)}} {\bf h}(\eta) d\eta \\ 
\label{integral-ml-2}
	&=P \theta(t)^{-\overline{\KMm{J_l}(\lambda_l)}} \int_{t_0}^{t} \theta(\eta)^{\overline{\KMm{J_l}(\lambda_l)}} P^{-1} \KMm{P_l^{(r)}} {\bf h}(\eta) d\eta.
\end{align}
Moreover, if the improper integral 
\begin{equation*}
\mathbb{I}_1 = \int_{t_0}^{t_{\max}} \theta(\eta)^{\overline{\KMm{J_l}(\lambda_l)}} P^{-1} \KMm{P_l^{(r)}} {\bf h}(\eta) d\eta 
\end{equation*}
converges, then the integral is also written as
\begin{align}
\notag
\int_{t_0}^t & \left(\frac{\theta(\eta)}{\theta(t)}\right)^{A} \KMm{P_l^{(r)}} {\bf h}(\eta) d\eta \\
\notag
	&= P \theta(t)^{-\overline{\KMm{J_l}(\lambda_l)}} \int_{t_0}^{t_{\max}} \theta(\eta)^{\overline{\KMm{J_l}(\lambda_l)}} P^{-1} \KMm{P_l^{(r)}} {\bf h}(\eta) d\eta 
	- P \theta(t)^{-\overline{\KMm{J_l}(\lambda_l)}} \int_{t}^{t_{\max}} \theta(\eta)^{\overline{\KMm{J_l}(\lambda_l)}} P^{-1} \KMm{P_l^{(r)}} {\bf h}(\eta) d\eta\\
\label{integral-ml-3}
	&\equiv P\theta(t)^{-\overline{\KMm{J_l}(\lambda_l)}} (\mathbb{I}_1 + \mathbb{I}_2(t)).
\end{align}
\KMk{Now we have}
\begin{equation}
\label{ineq-deg_lambda_l}
\deg_\theta\left( \KMm{\theta(t)^{\Lambda} P^{-1}} \KMm{P_l^{(r)}} {\bf h}(t) \right) \geq \lambda_{l} + \gamma,\quad  
\deg_\theta\left( \KMm{\theta(t)^{\Lambda} P^{-1}} \KMm{P_l^{(c)}} {\bf h}(t) \right) \geq {\rm Re}\,\lambda_{re,l} + \gamma.
\end{equation}
\KMk{by using Proposition \ref{prop-ord-fund-matrix}.}
There are two cases depending on the value
\begin{equation}
\label{nu_ml}
\KMm{\nu_l} := \deg_\theta \left(\theta(\KMk{t})^{\overline{\KMm{J_l}(\lambda_l)}} P^{-1} \KMm{P_l^{(r)}} {\bf h}(t) \right).
\end{equation}
From (\ref{ineq-deg_lambda_l}), we have $\KMm{\nu_l} \geq \lambda_l + \gamma$.
\begin{description}
\item[Case 1. $\KMm{\nu_l} \leq -1$.] 
\end{description}
In this case, the integral (\ref{integral-ml}) with the equivalent form (\ref{integral-ml-2}) is directly estimated through Proposition \ref{prop-ord-fundamental}-4 to obtain
\begin{align*}
\KMk{ \deg_\theta \left( \int_{t_0}^t \theta(\eta)^{\overline{\KMm{J_l}(\lambda_l)}} P^{-1} \KMm{P_l^{(r)}} {\bf h}(\eta) d\eta  \right) } &\geq \KMm{\nu_l} + 1 \\
	&\geq \lambda_l + \gamma + 1.
\end{align*}
Therefore
\begin{align}
\label{integral-ml-special-choice}
\KMk{ \deg_\theta \left( P \theta(t)^{-\overline{\KMm{J_l}(\lambda_l)}} \left\{ \int_{t_0}^t \theta(\eta)^{\overline{\KMm{J_l}(\lambda_l)}} P^{-1} \KMm{P_l^{(r)}} {\bf h}(\eta) d\eta  \right\} \right) }
&\geq \gamma + 1
\end{align}
by Proposition \ref{prop-ord-fund-matrix}.
In particular, for any constant vector ${\bf C}$, we have
\begin{align*}
\deg_\theta \left( P \theta(t)^{-\overline{\KMm{J_l}(\lambda_l)}} \left\{ {\bf C} + \int_{t_0}^t \theta(\eta)^{\overline{\KMm{J_l}(\lambda_l)}} P^{-1} \KMm{P_l^{(r)}} {\bf h}(\eta) d\eta \right\} \right) &\geq \min \{-\lambda_l, \gamma + 1\}
\end{align*}
from Proposition \ref{prop-ord-fundamental}-2.
\begin{description}
\item[Case 2. $\KMm{\nu_l} > -1$.] 
\end{description}
In this case, the improper integral $\mathbb{I}_1$ converges and the integral (\ref{integral-ml}) with the equivalent form (\ref{integral-ml-3}) is considered, instead of (\ref{integral-ml-2}).
Proposition \ref{prop-ord-fundamental} yields
\begin{align}
\notag
\deg_\theta\left( P\theta(t)^{-\overline{\KMm{J_l}(\lambda_l)}} (\mathbb{I}_1 + \mathbb{I}_2(t)) \right)
	&\geq \min\left\{ \deg_\theta \left( P\theta(t)^{-\overline{\KMm{J_l}(\lambda_l)}} \mathbb{I}_1 \right),  \deg_\theta \left( P\theta(t)^{-\overline{\KMm{J_l}(\lambda_l)}}\mathbb{I}_2(t)\right) \right\},\\
\notag
\deg_\theta \left( P\theta(t)^{-\overline{\KMm{J_l}(\lambda_l)}} \mathbb{I}_1 \right) &\geq -\lambda_l,\\
\notag
\deg_\theta \left( \mathbb{I}_2(t) \right) &\geq \KMm{\nu_l} + 1\quad (\text{from Proposition \ref{prop-ord-fundamental}-5})\\
\notag
	&\geq \lambda_l + \gamma + 1,\\
\label{integral-ml-4}
\deg_\theta \left( P\theta(t)^{-\overline{\KMm{J_l}(\lambda_l)}}\mathbb{I}_2(t)\right) &\geq \gamma + 1.
\end{align}
Combining two cases, we have
\begin{equation*}
\left(\frac{\theta(t)}{\theta(t_0)}\right)^{-A} \KMm{P_l^{(r)}} \left( {\bf C} + \int_{t_0}^t  \left(\frac{\theta(\eta)}{\theta(t_0)}\right)^{A} {\bf h}(\eta) d\eta \right) \geq \min\{ -\lambda_l, \gamma + 1\}
\end{equation*}
for any constant vector ${\bf C}\in \mathbb{R}^n$, where we have used the commutativity
\begin{equation*}
\KMm{P_l^{(r)}} \left(\frac{\theta(\eta)}{\theta(t_0)}\right)^{A} =  \left(\frac{\theta(\eta)}{\theta(t_0)}\right)^{A} \KMm{P_l^{(r)}}
\end{equation*}
for any $\eta\in [t_0, t]$, thank to the fact that $\KMm{P_l^{(r)}}$ is an invariant projector.
The similar estimates follow replacing $\lambda_l$ and $\KMm{P_l^{(r)}}$ by $\lambda_{re,l}$ and $\KMm{P_l^{(c)}}$, respectively, for the case of complex conjugate eigenvalues.
The inequality (\ref{deg-integral-theta-F-minus}) then follows from the property of projections
\begin{equation*}
P_- = \sum_{l; \lambda_l < 0} \KMm{P_l^{(r)}} +\sum_{l; {\rm Re}\, \lambda_{re,l} < 0} \KMm{P_l^{(c)}}.
\end{equation*}
As for $P_+$, the inequality (\ref{deg-integral-theta-F-plus}) follows from (\ref{integral-ml-4}) and
\begin{equation*}
P_+ = \sum_{l; \lambda_l > 0} \KMm{P_l^{(r)}} +\sum_{l; {\rm Re}\, \lambda_{re,l} > 0} \KMm{P_l^{(c)}}.
\end{equation*}
\KMl{
Finally consider the last inequality (\ref{deg-integral-theta-F-minus-2}).
Our main interest here is (\ref{h-decomposed}) replacing $\theta(t)$ by $\theta(t)/\theta(t_0)$.
Let $\KMm{\nu_l}$ be the asymptotic degree defined in (\ref{nu_ml}).
If $\KMm{\nu_l} < -1$, our statement is exactly (\ref{integral-ml-special-choice}).
Note that we have used the fact that ${\rm Re}\, (-\lambda_l) > 0 \geq \KMm{\nu_l} +1$ in the present setting.
}
\par
Now consider the case when $\KMm{\nu_l} > -1$.
In this case, the improper integral
\begin{equation*}
\lim_{t\to t_{\max} - 0} \int_{t_0}^{t}  \left( \frac{\theta(\eta)}{\theta(t_0)}\right)^{\overline{\KMm{J_l}(\lambda_l)}} P^{-1}\KMm{P_l^{(r)}} {\bf h}(\eta)  d\eta
\end{equation*}
converges.
Let $\overline{{\bf H}_{{\bf h}; \KMm{l} }}(t)$ be the primitive function of $\left( \theta(t) / \theta(t_0)\right)^{\overline{\KMm{J_l}(\lambda_l)}} P^{-1} \KMm{P_l^{(r)}} {\bf h}(t)$ such that (\ref{primitive-zero-power}) holds, which is $n$-vector-valued.
Proposition \ref{prop-ord-fundamental}-6 yields that
\begin{align*}
\deg_\theta\left( \overline{{\bf H}_{{\bf h}; \KMm{l} }}(t) \right) &\geq \KMm{\nu_l} + 1
\end{align*}
and hence
\begin{align*}
\deg_\theta \left( P\left(\frac{\theta(t)}{\theta(t_0)}\right)^{-\overline{\KMm{J_l}(\lambda_l)}} P_- \overline{{\bf H}_{{\bf h}; \KMm{l}}}(t) \right) &\geq \KMm{\nu_l} -\lambda_l + 1 \geq \gamma + 1,
\end{align*}
which completes our proof through the composition of the projection $P_-$.
\end{proof}

%% file: asymptotic.tex
\label{section-asym}

Now we move to provide a systematic methodology for deriving asymptotic \KMk{expansions} of blow-up solutions.
\KMk{When blow-up solutions are supposed to be given a priori, one of the main issues in blow-up studies is to characterize rough asymptotic behavior of blow-up solutions, well-known as {\em the blow-up rates}.}
Our aim here is to obtain successive terms of such blow-up solutions towards the precise description of blow-up behavior.

\subsection{The fundamental setting}
\KMj{In the present study, we only pay attention to blow-up solutions ${\bf y}(t)$ of (\ref{ODE-original}) possessing the following asymptotic behavior.}
\begin{ass}
\label{ass-fundamental}
The asymptotically quasi-homogeneous system (\ref{ODE-original}) of type $\alpha$ and the order $k+1$ admits a solution 
\begin{equation*}
{\bf y}(t) = (y_1(t), \ldots, y_n(t))^T
\end{equation*}
 which blows up at $t = t_{\max} < \infty$ with the asymptotic behavior
\begin{equation}
\label{blow-up-behavior}
y_i(t) \KMk{ \sim c_i \theta(t)^{-\alpha_i / k}}, \quad t\to t_{\max}-0,\quad i=1,\ldots, n,
\end{equation}
\KMk{for some constants $c_i\in \mathbb{R}$.}
\end{ass}
Under Assumption \ref{ass-fundamental}, we shall write the blow-up solution ${\bf y}(t)$ as
\KMk{${\bf y}(t) = \theta(t)^{-\frac{1}{k}\Lambda_\alpha} {\bf Y}(t)$, equivalently}
\begin{equation}
\label{blow-up-sol}
y_i(t) = \theta(t)^{-\alpha_i / k} Y_i(t),\quad {\bf Y}(t) = (Y_1(t), \ldots, Y_n(t))^T,
\end{equation}
and determine the concrete form of the factor ${\bf Y}(t)$.
\par
As the first step, decompose the vector field $f$ into two terms as follows:
\begin{equation*}
f({\bf y}) = f_{\alpha, k}({\bf y}) + f_{\rm res}({\bf y}),
\end{equation*}
where $f_{\alpha, k}$ is the quasi-homogeneous component of $f$ and $f_{\rm res}$ is the residual (\KMe{i.e.,} lower-order) terms.
The componentwise expressions are 
\begin{equation*}
f_{\alpha, k}({\bf y}) = (f_{1;\alpha, k}({\bf y}), \ldots, f_{n;\alpha, k}({\bf y}))^T,\quad f_{\rm res}({\bf y}) = (f_{1;{\rm res}}({\bf y}), \ldots, f_{n;{\rm res}}({\bf y}))^T\KMe{,}
\end{equation*}
\KMe{respectively.
}
Substituting (\ref{blow-up-sol}) into (\ref{ODE-original}), we have
\begin{align*}
\frac{\alpha_i}{k}\theta(t)^{-\alpha_i / k-1} Y_i + \theta(t)^{-\alpha_i / k} Y_i' &= f_{i; \alpha, k}({\bf y}) + f_{i;{\rm res}}({\bf y})\\
	&= \theta(t)^{-(k+\alpha_i) / k} f_{i; \alpha, k}({\bf Y}) + f_{i;{\rm res}}({\bf y}),
\end{align*}
where we have used the quasi-homogeneity of $f_{\alpha, k}({\bf y})$ in the last equality.
Multiplying $(t_{\max}-t)^{\frac{\alpha_i}{k} + 1}$ for each $i$, we finally have
\begin{equation*}
\theta(t) \frac{d}{dt}Y_i = -\frac{\alpha_i}{k}Y_i + f_{i; \alpha, k}({\bf Y}) +  \theta(t)^{\frac{\alpha_i}{k}+1} f_{i;{\rm res}}(\theta(t)^{-\KMf{ \frac{1}{k}\Lambda_\alpha }} {\bf Y})
\end{equation*}
with the vector form
\begin{align}
\label{blow-up-basic-0}
\theta(t)\frac{d}{dt} {\bf Y} &= - \KMf{ \frac{1}{k}\Lambda_\alpha } {\bf Y} + f_{\alpha, k}({\bf Y}) + \theta(t)^{\KMf{ \frac{1}{k}\Lambda_\alpha }+\KMg{I}} f_{{\rm res}}(\theta(t)^{-\KMf{ \frac{1}{k}\Lambda_\alpha }} {\bf Y}),\\
\notag
\theta(t)^{\KMf{ \frac{1}{k}\Lambda_\alpha }} &\equiv {\rm diag}\left( \theta(t)^{\alpha_1/k},\ldots, \theta(t)^{\alpha_n/k}\right).
\end{align}
Another form equivalent to (\ref{blow-up-basic-0}) is 
\begin{equation}
\label{blow-up-basic}
\frac{d}{dt}{\bf Y} = \theta(t)^{-1}\left\{ -\KMf{ \frac{1}{k}\Lambda_\alpha }{\bf Y} + f_{\alpha, k}({\bf Y}) \right\} + \theta(t)^{\KMf{ \frac{1}{k}\Lambda_\alpha }} f_{{\rm res}}( \theta(t)^{-\KMf{ \frac{1}{k}\Lambda_\alpha }} {\bf Y}).
\end{equation}
While the residual term $f_{i;{\rm res}}(\theta(t)^{-\KMf{ \frac{1}{k}\Lambda_\alpha }} {\bf Y})$ can diverge as $t\to t_{\max} - 0$, the product $\theta(t)^{\KMc{\KMf{ \frac{1}{k}\Lambda_\alpha } + \KMg{I}}} f_{{\rm res}}(\theta(t)^{-\KMf{ \frac{1}{k}\Lambda_\alpha }} {\bf Y})$ goes to zero as $t\to t_{\max}-0$ from the asymptotic quasi-homogeneity of $f$.
\par
The system (\ref{blow-up-basic}) is actually nonautonomous, but the most singular term $\theta(t)^{-1}\left\{ - \KMf{\frac{1}{k}\Lambda_\alpha }{\bf Y} + f_{\alpha, k}({\bf Y}) \right\}$ as $t\to t_{\max}- 0$ has the common $t$-dependence $\theta(t)^{-1}$.
We thus focus on this singular term and consider the expansion of vector fields so that the linearization at an appropriate point and \KMi{normalization} of the system at the point can be applied to deriving the function ${\bf Y}(t)$ step by step.

\subsection{Balance law}
First note that the asymptotic behavior (\ref{blow-up-behavior}) indicates the fact that
the function ${\bf Y}(t)$ is as smooth as $f$ in $t < t_{\max}$ and there is a non-zero vector ${\bf Y}_0\in \mathbb{R}^n$ satisfying
\begin{equation}
\label{form-Y}
{\bf Y}(t) = {\bf Y}_0 + \tilde {\bf Y}(t),\quad {\bf Y}_0 = (Y_{0,1}, \ldots, Y_{0,n})^T,\quad \lim_{t\to t_{\max}-0}\tilde {\bf Y}(t) = {\bf 0}\in \mathbb{R}^n.
\end{equation}
The constant term ${\bf Y}_0$ determines the complete form of the lowest order term of ${\bf Y}(t)$. 
This term can be determined as a root of \KMi{the} time-independent system associated with (\ref{blow-up-basic-0}).
Taking the limit $t\to t_{\max}-0$ in (\ref{blow-up-basic-0}), we observe with continuity of the vector field and ${\bf Y} = {\bf Y}(t)$ that the system must satisfy
\begin{equation}
\label{0-balance}
\KMf{\frac{1}{k}\Lambda_\alpha }{\bf Y}_0 = f_{\alpha, k}({\bf Y}_0),
\end{equation}
provided that $\theta(t) \frac{d{\bf Y}}{dt} = o(\theta(t))$ holds as $t\to t_{\max}$, which have to be justified later\footnote{
By Theorem \ref{thm-smoothness-Y}, this asymptotic relation is proved to be valid.
\KMk{See Remark \ref{rem-justification-balance-0}.}
}.

\begin{dfn}\rm
\label{dfn-balance}
We call the equation (\ref{0-balance}) \KMf{{\em a balance law}} for the blow-up solution ${\bf y}(t)$.
\end{dfn}

Note that (nonzero) roots of the balance law are not always uniquely determined.
In other words, the equation (\ref{0-balance}) can have \KMc{multiple different solutions}.

\subsection{\KMc{Multi-order asymptotic expansions}}

We shall determine the non-constant term $\tilde {\bf Y}(t)$.
Assume that a nonzero root ${\bf Y}_0$ of the balance law (\ref{0-balance}) is obtained.
We do not know the concrete form of $\tilde {\bf Y}(t)$ at present\footnote{
We would try to express $\tilde {\bf Y}(t)$ as an asymptotic {\em power series}, but there is no guarantee that $\tilde {\bf Y}(t)$ can be expressed as power series, as mentioned in Introduction.
}. 
We then expand $\tilde {\bf Y}(t)$ as the {\em general asymptotic series}
\begin{align}
\notag
{\bf Y}(t) &= {\bf Y}_0 + \tilde {\bf Y}(t),\\
\label{Y-asym}
\tilde {\bf Y}(t) &= \sum_{j=1}^\infty {\bf Y}_j(t),\quad {\bf Y}_{j}(t) \ll {\bf Y}_{j-1}(t)\quad (t\to t_{\max}-0),\quad j=1,2,\ldots,\\
\notag
{\bf Y}_j(t) &= (Y_{j,1}(t), \ldots, Y_{j,n}(t))^T,\quad j=1,2,\ldots
\end{align}
and solve ${\bf Y}_j(t)$ through the balance \KMf{law} for each $j$ inductively as the general framework of local asymptotic analysis for ordinary differential equations (cf. \cite{BO1999}), where the asymptotic relation ${\bf Y}_{j}(t) \ll {\bf Y}_{j-1}(t)$ for vector-valued functions \KMf{are} considered {\em componentwise}.
Note that all the arguments below are valid only near $t = t_{\max}$ with $t < t_{\max}$.

\subsubsection{\KMc{The blow-up power-determining matrix and the blow-up power eigenvalues}}
Apply the Taylor's theorem \KMf{to} the right-hand side of (\ref{blow-up-basic}) at ${\bf Y} = {\bf Y}_0$:
\begin{align}
\label{asym-eq}
&\frac{d}{dt} \tilde {\bf Y} = \theta(t)^{-1} \left[ \left( -\KMf{ \frac{1}{k}\Lambda_\alpha } + Df_{\alpha, k}({\bf Y}_0) \right)\tilde {\bf Y} + R_{\alpha, k}({\bf Y}) \right]
	+  \theta(t)^{\KMf{ \frac{1}{k}\Lambda_\alpha }}  f_{{\rm res}}\left(\theta(t)^{-\KMf{ \frac{1}{k}\Lambda_\alpha }}  {\bf Y} \right),\\
	\notag
&R_{\alpha, k}({\bf Y}) = f_{\alpha,k}({\bf Y}) - \left\{f_{\alpha,k}({\bf Y}_0) + Df_{\alpha,k}({\bf Y}_0)\tilde {\bf Y}\right\},
\end{align}
where all terms involving the balance law (\ref{0-balance}) are cancelled.
The remainder of the Taylor expansion of $f_{\alpha, k}$ is expressed as $R_{\alpha, k}$ and note that $R_{\alpha, k}({\bf Y}_0) = 0$.
\par
Now \KMc{we derive a governing system} for the second term ${\bf Y}_1$ eliminating all possible \KMe{negligible} terms.
From (\ref{Y-asym}), the principal term of the left-hand side \KMi{will} be $d{\bf Y}_1/dt$, which have to be verified later \KMi{(see the sentence before Remark \ref{rem-order-Yj})}. 
On the other hand, the principal terms in the right-hand side \KMi{will} involve the linear part $ (-\KMf{ \frac{1}{k}\Lambda_\alpha } + Df_{\alpha, k}({\bf Y}_0))\tilde {\bf Y}$ and $f_{{\rm res}}\left(\theta(t)^{-\KMf{ \frac{1}{k}\Lambda_\alpha }}  {\bf Y}_0 \right)$.
In other words, the other parts are negligible as $t\to t_{\max}$, compared with the above terms\footnote{
If our interests are only several lower terms, the governing part from the remaining term can be eliminated.
}.
As a summary, the asymptotically governing system for ${\bf Y}_1$ as $t\to t_{\max}-0$ \KMk{we consider is}
\begin{align}
\label{2nd-0}
\frac{d}{dt}{\bf Y}_1 &= \theta(t)^{-1}\left\{ -\KMf{ \frac{1}{k}\Lambda_\alpha } + D f_{\alpha, k}({\bf Y}_0) \right\}{\bf Y}_1 + \KMc{\theta(t)^{\KMf{ \frac{1}{k}\Lambda_\alpha }}}  f_{\rm res}( \KMc{\theta(t)^{-\KMf{ \frac{1}{k}\Lambda_\alpha }}} {\bf Y}_0),
\end{align}
which is inhomogeneous and {\em linear}.
Moreover, the coefficient of the homogeneous part is the product of \KMe{$\theta(t)^{-1}$} and a {\em constant} matrix.
\KMe{We can therefore apply the fundamental matrix of linear homogeneous systems with constant coefficient matrix to obtaining the general form of solutions.
In concrete problems, we can transform this constant matrix to a canonical form through a constant matrix and reduce the problem to $n$ inhomogeneous differential equations of the first order.}
The constant matrix $-\KMf{ \frac{1}{k}\Lambda_\alpha } + D f_{\alpha, k}({\bf Y}_0)$ plays a key role in our asymptotic analysis.

\begin{dfn}[Blow-up power eigenvalues]\rm
\label{dfn-blow-up-power-ev}
Suppose that a nonzero root ${\bf Y}_0$ of the balance law (\ref{0-balance}) is given.
We call the constant matrix
\begin{equation}
\label{blow-up-power-determining-matrix}
A = -\KMf{ \frac{1}{k}\Lambda_\alpha } + D f_{\alpha, k}({\bf Y}_0)
\end{equation}
the {\em \KMf{blow-up} power-determining matrix} for the blow-up solution ${\bf y}(t)$, and call the eigenvalues $\{\lambda_i\}_{i=1}^n \equiv {\rm Spec}(A)$ the {\em blow-up power eigenvalues}, where eigenvalues with nontrivial multiplicity are distinguished in this expression, except specifically noted.
\end{dfn}
Using the matrix $A$, (\ref{asym-eq}) is simply rewritten as follows:
\begin{equation*}
\frac{d}{dt} \tilde {\bf Y} = \theta(t)^{-1} \left[ A\tilde {\bf Y} + R_{\alpha, k}({\bf Y}) \right]
	+  \theta(t)^{\KMf{ \frac{1}{k}\Lambda_\alpha }}  f_{{\rm res}}\left( \theta(t)^{-\KMf{ \frac{1}{k}\Lambda_\alpha }} {\bf Y} \right).
\end{equation*}
Similarly, the linear system (\ref{2nd-0}) is rewritten as follows:
\begin{align}
\label{2nd}
\frac{d}{dt}{\bf Y}_1 &= \theta(t)^{-1}A {\bf Y}_1 + {\bf g}_1(t),\quad 
{\bf g}_1(t) = \KMc{\theta(t)^{\KMf{ \frac{1}{k}\Lambda_\alpha }}}  f_{\rm res}( \KMc{\theta(t)^{-\KMf{ \frac{1}{k}\Lambda_\alpha }}} {\bf Y}_0).
\end{align}

\subsubsection{\KMe{Multi-order asymptotic expansions}}


Let $\Psi(t) \equiv \theta(t)^{-A}$ be the fundamental matrix of the homogeneous system
\begin{align*}
\frac{d}{dt}{\bf Y}_1 &= \theta(t)^{-1}A {\bf Y}_1.
\end{align*}
\KMj{Also, let 
$\mathbb{R}^n = \mathbb{E}_A^u \oplus  \mathbb{E}_A^c \oplus \mathbb{E}_A^s$
be the decomposition of $\mathbb{R}^n$ into the invariant subspaces associated with ${\rm Spec}(A)$ mentioned in (\ref{dec-Rn}).}
The following proposition provides a simple form of ${\bf Y}_1(t)$ satisfying the asymptotic relation (\ref{Y-asym}) with $j=1$ which is convergent under a mild assumption to $A$ and a suitable choice of initial points.

\begin{prop}
\label{prop-conv-2nd}
Fix $t_{\max}\KMf{>t_0}$, assuming finite, and let ${\bf Y}_0$ be a nonzero root of the balance law (\ref{0-balance}) for (\ref{blow-up-basic}).
Assume that the associated  blow-up power-determining matrix $A$ is hyperbolic:
\begin{equation*}
{\rm Spec}(A) \cap \sqrt{-1}\mathbb{R} = \emptyset.
\end{equation*}
Then the inhomogeneous system (\ref{2nd}) admits a \KMk{global-in-time} solution (\ref{formula-2nd}) \KMk{with}
\begin{equation}
\label{convergent-2nd}
\lim_{t\to t_{\max}-0}{\bf Y}_1(t) = {\bf 0}\in \mathbb{R}^n.
\end{equation}
This solution can be written as
\begin{equation}
\label{formula-2nd}
{\bf Y}_1(t) = \left(\frac{\theta (t)}{\theta (t_0)}\right)^{-A} \left\{ P_- \left( {\bf Y}_1^0 + \int_{t_0}^t \left(\frac{\theta (\eta)}{\theta (t_0)}\right)^{A} {\bf g}_1(\eta) d\eta \right) - P_+  \int_{t}^{t_{\max}} \left(\frac{\theta (\eta)}{\theta (t_0)}\right)^{A} {\bf g}_1(\eta) d\eta \right\}
,\quad {\bf Y}_1^0 \in \mathbb{R}^n
\end{equation}
\KMk{for $t\in (t_0, t_{\max})$} 
with the initial time $t=t_0$, whenever $\|P_- {\bf Y}_1^0\|$ is sufficiently small.
In particular, (\ref{2nd}) locally generates a $C^r$ $m_A$-parameter family of solutions ${\bf Y}_1(t)$ satisfying (\ref{convergent-2nd}), where
\begin{equation}
\label{mA}
m_A = \sharp \{\lambda\in {\rm Spec}(A) \mid {\rm Re}\lambda < 0\} \equiv \dim \mathbb{E}_A^s.
\end{equation}
\end{prop}

The proof is left to Appendix \ref{section-app-proof-conv-2nd}.
The key point of the expression \KMk{(\ref{formula-2nd})} is that free parameters describing the solution ${\bf Y}_1(t)$ are chosen only from $\mathbb{E}_A^s$.
We see in examples in Part II \cite{asym2} that integral constants stemming from $\mathbb{E}_A^u$ vanish in \KMk{all} cases, from asymptotic quasi-homogeneity of $f$ and hyperbolicity of $A$.

\par
\bigskip
Multi-order asymptotic expansion of ${\bf Y}(t)$ can be derived inductively.
In other words, the higher-order terms ${\bf Y}_j(t)$, $j\geq 2$, can be also determined in the similar way \KMf{to ${\bf Y}_1(t)$}, while one attention mentioned below is needed.
Now assume that, for given $j \geq 2$, the preceding terms $\{{\bf Y}_l(t)\}_{l=0}^{j-1}$ are calculated.
Neglecting possible higher-order terms, the governing system for ${\bf Y}_j(t)$ is given as follows:
\begin{align}
\label{jth}
\frac{d}{dt} {\bf Y}_j &= \theta(t)^{-1} A{\bf Y}_j + {\bf g}_j(t),\\
\notag
{\bf g}_j(t) &= \theta(t)^{-1} \left\{ R_{\alpha, k}({\bf S}_{j-1}({\bf Y})(t)) - R_{\alpha, k}({\bf S}_{j-2}({\bf Y})(t)) \right\}\\
\label{gj}
	&\quad  + \theta(t)^{\KMf{ \frac{1}{k}\Lambda_\alpha }}  \left\{ f_{{\rm res}}({\bf S}^{\theta}_{j-1}({\bf Y})(t)) - f_{{\rm res}}({\bf S}^{\theta}_{j-2}({\bf Y})(t)) \right\},
\end{align}
where
\begin{align}
\label{Sn}
{\bf S}_j({\bf Y})(t) &\equiv \sum_{\tilde j=0}^j {\bf Y}_{\tilde j}(t),\quad 
{\bf S}^{\theta}_j ({\bf Y})(t) \equiv \theta(t)^{-\KMf{ \frac{1}{k}\Lambda_\alpha }}  {\bf S}_j({\bf Y})(t),\quad j =0,1,\ldots,\\
\notag
{\bf S}_j ({\bf Y})(t) &\equiv {\bf S}^{\theta}_j ({\bf Y})(t)\equiv 0,\quad j <0.
\end{align}
The system (\ref{jth}) \KMi{is} inhomogeneous and linear, like (\ref{2nd}).
In the similar way to solving (\ref{2nd}), we obtain a general form of solutions ${\bf Y}_j(t)$.
However, we have to pay attention to asymptotic relations (\ref{Y-asym}).
Although the validity of (\ref{Y-asym}) depends on $f$ and $\{{\bf Y}_l(t)\}_{l=0}^{j-1}$, one can construct a solution of (\ref{jth}) such that free parameters describing ${\bf Y}_j(t)$ are not newly generated except those included in ${\bf Y}_1(t)$, which is uniquely determined.
As shown in the proof of the following proposition, such a choice of solutions indicates the {\em elimination of exponential terms generated by the homogeneous part of (\ref{jth})}, namely $d{\bf Y}_j / dt = \theta(t)^{-1}A{\bf Y}_j$, in solutions ${\bf Y}_j(t)$, which would yield \KMf{a necessary condition of} the asymptotic relation (\ref{Y-asym}).


\begin{prop}
\label{prop-conv-(j+1)th}
Fix $t_{\max}\KMf{> t_0}$, assuming finite, and let ${\bf Y}_0$ be a nonzero root of the balance law (\ref{0-balance}) for (\ref{blow-up-basic}).
Assume that the associated  blow-up power-determining matrix $A$ is hyperbolic and that, for given $j \geq 2$, the preceding terms $\{{\bf Y}_l(t)\}_{l=0}^{j-1}$ are calculated.
Then the integral 
\begin{equation}
\label{formula-(j+1)}
{\bf Y}_j(t) =  \left(\frac{\theta (t)}{\theta (t_0)}\right)^{-A} \left\{ P_- \left( {\bf Y}_j^0 + \int_{t_0}^t \left(\frac{\theta (\eta)}{\theta (t_0)}\right)^{A} {\bf g}_j(\eta) d\eta \right) - P_+  \int_{t}^{t_{\max}} \left(\frac{\theta (\eta)}{\theta (t_0)}\right)^{A} {\bf g}_j(\eta) d\eta \right\}
\end{equation}
\KMk{for $t\in (t_0, t_{\max})$} 
with the initial time $t=t_0$ solves the inhomogeneous system (\ref{jth}).
The constant vector ${\bf Y}_j^0 \in \mathbb{R}^n$ \KMl{can be} chosen to satisfy
\begin{equation}
\label{const-asym-j}
P_- {\bf Y}_j^0 =  P_- \KMl{P{\bf H}_j}(t_0),
\end{equation}
where \KMl{the matrix $A$ has the Jordan structure through a nonsingular matrix $P$: $P^{-1}AP = \Lambda$, and ${\bf H}_j = {\bf H}_j(t)$ denotes the primitive function of $\left( \theta (t) / \theta (t_0) \right)^{\Lambda}P^{-1} P_- {\bf g}_j(t)$} satisfying
\begin{equation}
\label{integral-jth-H}
\KMl{\lim_{t\to t_{\max}-0} ({\bf H}_j(t))_i = 0 \quad \text{whenever}\quad \deg_\theta \left( \left( \left( \frac{\theta (t)}{\theta (t_0)} \right)^{\Lambda}P^{-1} P_- {\bf g}_j(t) \right)_i \right) > -1}.
\end{equation}
The initial point 
\begin{equation*}
{\bf Y}_j(t_0) =  P_- {\bf Y}_j^0 - P_+  \int_{t_0}^{t_{\max}} \left(\frac{\theta (\eta)}{\theta (t_0)}\right)^{A} {\bf g}_j(\eta) d\eta
\end{equation*}
at $t=t_0$ is \KMl{then} uniquely determined once free parameters in ${\bf Y}_1(t)$ constructing ${\bf g}_j(t)$ are fixed.
\end{prop}

The proof is left to Appendix \ref{section-app-proof-conv-(j+1)th}. 
\KMl{The choice of ${\bf Y}_j^0$ in the proposition plays a key role in justification of multi-order asymptotic expansions: Proposition \ref{prop-order-increase} and Theorem \ref{thm-asym-exp} stated below.}
Under this arrangement, we see in (\ref{formula-(j+1)}) that general solutions of the homogeneous part make {\em no} contributions to ${\bf Y}_j$ with $j\geq 2$.
In such a case, it turns out that our choice in (\ref{2nd}) that the governing term in the {\em left-hand side} of (\ref{asym-eq}) being $d{\bf Y}_1/dt$ is valid.

\begin{rem}
\label{rem-order-Yj}
For $j\geq 2$, the leading terms of the function ${\bf g}_j(t)$ are essentially ${\bf Y}_{j-1}(t)$ and $Df_{{\rm res}}({\bf S}^{\theta}_{j-2}({\bf Y})(\KMe{t})){\bf Y}_{j-1}(t)$ \KMf{by the Taylor's theorem}.
The integral in (\ref{formula-(j+1)}) would therefore increase the order of $\theta(t)$ in ${\bf Y}_j(t)$, which would \KMe{provide} the asymptotic relation (\ref{Y-asym}).
As seen in Section \ref{section-justification}, this expectation is justified \KMk{under mild assumptions}.
\end{rem}

The above arguments are summarized as the following theorem.
\KMi{Obviously, the formal sum of \KMk{systems} (\ref{0-balance}), (\ref{2nd}) and (\ref{jth}) for $j\geq 2$ recovers (\ref{blow-up-basic}).}
\begin{thm}[Formal asymptotic expansion of blow-up solutions]
\label{thm}
Let ${\bf y}(t)$ be a blow-up solution of (\ref{ODE-original}) satisfying Assumption \ref{ass-fundamental}.
Assume that a nonzero root ${\bf Y}_0$ of the balance law (\ref{0-balance}) is given and that the associated blow-up power-determining matrix $A$ is hyperbolic.
Then ${\bf y}(t)$ with the initial time $t=t_0$ is formally expanded as follows\footnote{
\KMe{The terminology {\em formal} in this sentence means that neither the convergence of the infinite series (\ref{formula-asym}) nor the asymptotic relation (\ref{Y-asym}) are considered in detail.} 
}:
\begin{align}
\notag
{\bf y}(t) &= \left(\frac{\theta (t)}{\theta (t_0)}\right)^{-\KMf{ \frac{1}{k}\Lambda_\alpha }} {\bf Y}_0\\
\label{formula-asym}
	& + \left(\frac{\theta (t)}{\theta (t_0)}\right)^{-(\KMf{ \frac{1}{k}\Lambda_\alpha } + A)}  \sum_{j=1}^\infty \left\{ P_- \left( {\bf Y}_j^0 + \int_{t_0}^t \left(\frac{\theta (\eta)}{\theta (t_0)}\right)^{A} {\bf g}_j(\eta) d\eta \right) - P_+  \int_{t}^{t_{\max}} \left(\frac{\theta (\eta)}{\theta (t_0)}\right)^{A} {\bf g}_j(\eta) d\eta \right\},
\end{align}
where ${\bf g}_1(t)$ is given in (\ref{2nd}), while ${\bf g}_j(t)$ for $j\geq 2$ \KMi{are} given in (\ref{gj}).
The constant vector ${\bf Y}_1^0 \in \mathbb{R}^n$ is arbitrarily chosen so that $\|P_- {\bf Y}_1^0\|$ is sufficiently small, while ${\bf Y}_j^0 \in \mathbb{R}^n$ with $j\geq 2$ are chosen so that (\ref{const-asym-j}) holds.
In particular, (\ref{2nd}) and (\ref{jth}) for $j\geq 2$ generate an $m_A$-parameter family of formal expansions of ${\bf y}(t)$ with fixed $t_{\max}$.
\end{thm}

This is a systematic procedure of a candidate of asymptotic expansion of ${\bf y}(t)$ in an arbitrary order.
On the other hand, the asymptotic relation
\begin{equation*}
{\bf y}(t) - \theta(t)^{-\KMf{ \frac{1}{k}\Lambda_\alpha }} \sum_{\KMc{j}=0}^N {\bf Y}_j(t) \ll \theta(t)^{-\KMf{ \frac{1}{k}\Lambda_\alpha }} {\bf Y}_N(t)\quad (t\to t_{\max}) 
\end{equation*}
has to be independently investigated to claim that (\ref{formula-asym}) is indeed an asymptotic expansion of ${\bf y}(t)$.
In fact, Proposition \ref{prop-conv-(j+1)th} does not refer to the decay order of ${\bf Y}_j(t)$ so that asymptotic relation ${\bf Y}_j(t) \ll {\bf Y}_{j-1}(t)$ as $t\to t_{\max}-0$ is satisfied for all $j \in \KMl{\mathbb{Z}_{\geq 1}}$.
\KMi{Nevertheless, we see later that, \KMk{under mild assumptions}, the above solution satisfies (\ref{Y-asym}) in an appropriate sense.}
\par
\bigskip
\KMi{On the other hand, there is a nontrivial problem whether the (formal) series ${\bf Y}(t)$ is analytically meaningful, that is, convergent.}
The following theorem shows that the full system (\ref{blow-up-basic}) possesses a global-in-time solution converging to a root ${\bf Y}_0$ of the balance law (\ref{0-balance}) and locally generates an $m_A$-parameter family of such a solution, indicating that the consideration of asymptotic expansion of the form (\ref{formula-asym}) {\em as an $m_A$-parameter family} is consistent.


\begin{thm}[Convergence and parameter dependence of ${\bf Y}(t)$]
\label{thm-smoothness-Y}
For a blow-up solution associating ${\bf Y}(t)$, fix $t_{\max}$, assuming finite, and let ${\bf Y}_0$ be a nonzero root of the balance law (\ref{0-balance}) for (\ref{blow-up-basic}).
Assume that the associated  blow-up power-determining matrix $A$ is \KMf{hyperbolic.}
Then the system (\ref{blow-up-basic}) has a solution ${\bf Y}(t)$ converging to ${\bf Y}_0$ as $t\to t_{\max} - 0$ and generates an $m_A$-parameter family of solutions 
\begin{equation*}
{\bf Y}(t) = {\bf Y}(t; t_{\max}, C_1, \ldots, C_{m_A})
\end{equation*}
satisfying (\ref{Y-asym}), where $m_A$ is given in (\ref{mA}),
and free parameters $C_1, \ldots, C_{m_A}$ are chosen in a small neighborhood of ${\bf 0}\in \mathbb{R}^{m_A}$.
Moreover, the parameter family ${\bf Y}(t; t_{\max}, C_1, \ldots, C_{m_A})$ is $C^r$ with respect to $t$ with $t<t_{\max}$ and $C_1,\ldots, C_{m_A}$.
\end{thm}

The right-\KMc{sided} compactification of time variable introduced in \cite{WXJ2021} is applied to the proof, which is left to Appendix \ref{section-app-proof-smoothness-Y}.
\begin{rem}[Validity of the balance law (\ref{0-balance})]
\label{rem-justification-balance-0}
Using the convergence stated in Theorem \ref{thm-smoothness-Y}, our choice of terms in (\ref{blow-up-basic-0}) for deriving the balance law (\ref{0-balance}) is confirmed to be valid.
That is, all possible dominant terms in (\ref{blow-up-basic-0}) as $t\to t_{\max}$ appears in (\ref{0-balance}). 
Indeed, the proof of Theorem \ref{thm-smoothness-Y}, shown in Appendix \ref{section-app-proof-smoothness-Y}, implies that $\theta(t) \frac{d{\bf Y}}{dt} = O(\theta(t)^\mu)$ for some $\mu > 0$, whenever $A$ is hyperbolic.
\end{rem}
\par
We can also prove that, for each $N\in \mathbb{N}$, the remainder \KMe{${\bf Y}_N^c(t) \equiv {\bf Y}(t) -S_{N}({\bf Y})(t)$} also converges as $t\to t_{\max}$ whenever free parameters are small.
The precise statement is below, which is proved in Appendix \ref{section-app-proof-convergence-Y}.

\begin{prop}[Convergence of the remainder of ${\bf Y}(t)$]
\label{prop-convergence-Y}
Under the same assumptions as Theorem \ref{thm-smoothness-Y}, for each $N\in \KMi{\mathbb{N}}$,
assume that vector-valued functions \KMe{$\{{\bf Y}_j(t)\}_{j=0}^{N}$} are constructed as ${\bf Y}_0$ when $j=0$, (\ref{formula-2nd}) when $j=1$ and (\ref{formula-(j+1)}) for $j\geq 2$\KMi{.}
\par
Then, for each $N\in \KMi{\mathbb{N}}$, the remainder \KMe{${\bf Y}_N^c(t) \equiv {\bf Y}(t) -S_{N}({\bf Y})(t)$} of the formal asymptotic expansion of ${\bf Y}(t)$ converges to ${\bf 0}\in \mathbb{R}^n$ as $t\to t_{\max} - 0$ whenever free parameters $C_1,\ldots, C_{m_A}$ are small.
\end{prop}

\subsection{Complete determination of asymptotic expansion: a case study}
In the above setting, the function ${\bf g}_j$ contained in the integrand consists of nonlinearities depending on ${\bf S}_{j-1}({\bf Y})(t)$.
Powers of asymptotic series are then determined by blow-up power eigenvalues $\{\lambda_i\}$, the quasi-homogeneous part $f_{\alpha, k}$ and \KMi{the lower term} $f_{\rm res}$.
In case that $f$ has a special form, the complete form of asymptotic expansions can be determined.
For example, \KMe{assume that the vector field $f$ is {\em quasi-homogeneous} polynomial. Then} we know that $f_{\rm res} \equiv 0$ and the equations for all ${\bf Y}_j$ become much simpler.
When we further assume that $A$ is diagonalizable \KMc{and that ${\rm Spec}(A)\subset \mathbb{R}$}, (\ref{2nd}) becomes linear and {\em homogeneous}, and hence the solution ${\bf Y}_1$ is written as the linear combination of $\{ \theta(t)^{-\lambda_i}\}_{i=1}^n$.
In the next step for solving ${\bf Y}_2$, the powers of $\theta(t)$ \KMf{appeared in (\ref{jth}) with $j=2$} are determined as the following form:
\begin{equation*}
\KMi{\left( \sum_{\substack{\beta\in \mathbb{Z}_{\geq 0}^n\\ |\beta| = l }} \sum_{i=1}^n \beta_i \lambda_i \right)} - 1,\quad l =1,2,\ldots,
\end{equation*}
\KMf{because} $f_{\rm res} \equiv 0$.
\KMf{Through the integration,} we see that the powers of $\theta(t)$ in ${\bf Y}_2$ have the following form:
\begin{equation}
\label{power-QH}
\sum_{\substack{\beta\in \mathbb{Z}_{\geq 0}^n\\ |\beta| = l }} \sum_{i=1}^n \beta_i \lambda_i,\quad l=1,2,\ldots.
\end{equation}
The same conclusion holds for ${\bf Y}_j$ with $j\geq 3$.
We therefore expect that all possible powers of \KMe{$\theta(t)$} appeared in ${\bf Y}(t)$ are restricted to the form (\ref{power-QH}), and hence the true powers of $\theta(t)$ in $y_i(t)$ are the sum of the above exponents and the principal exponent $-\alpha_i/k$.
The similar argument yields the complete form of (formal) asymptotic expansions\KMe{, even though} $A$ admits nontrivial Jordan blocks.

\begin{thm}
\label{thm-asym-QH}
\KMd{Consider (\ref{ODE-original}) with a polynomial vector field $f$ which is quasi-homogeneous.}
Let ${\bf y}(t) = (y_1(t), \ldots, y_n(t))^T$ be a blow-up solution for $f$ satisfying Assumption \ref{ass-fundamental}.
Suppose that a nonzero root ${\bf Y}_0$ of the balance law is associated and that the induced blow-up power-determining matrix $A$ satisfies ${\rm Spec}(A)\subset \mathbb{R}$.
Then ${\bf y}(t)$ has the form
\begin{equation}
\label{asym-form}
y_i(t) = \KMd{\theta(t)}^{-\alpha_i / k} \sum_{j=0}^\infty \left\{ \sum_{M=0}^{M_{j,k,\alpha, A}}\sum_{\substack{\beta\in \mathbb{Z}_{\geq 0}^n\\ |\beta| = j}} Y_{\beta,M,i} \KMd{\theta(t)}^{-\beta \cdot_{A} \lambda} (\ln \KMd{\theta(t)} )^M \right\},
\end{equation}
where $Y_{\beta,M} = (Y_{\beta,M,1}, \ldots, Y_{\beta,M,n})^T\in \mathbb{R}^n$ and
\begin{equation*}
\beta \cdot_A \lambda \equiv \sum_{l=1}^n \beta_l (\lambda_l)_-,\quad (\mu)_- \equiv \min \{0, \mu\}
\end{equation*}
for $\beta\in \mathbb{Z}_{\geq 0}^n$ and $\lambda = (\lambda_1,\ldots, \lambda_n)^T$ with $\{\lambda_i\}_{i=1}^n = {\rm Spec}(A)$ counting the multiplicity.
The integer $M_{j,k,\alpha, A}$ determined by $j, k, \alpha$, \KMk{geometric multiplicity} of elements in ${\rm Spec}(A)$ \KMk{and size of associated Jordan blocks}.
\end{thm}

The proof is left in Appendix \ref{section-appendix-proof}.
The similar conclusion \KMi{will} be obtained when $f$ is \KMe{polynomial} which is not necessarily quasi-homogeneous.
In this case, exponents of $\theta(t)$ and $\ln \theta(t)$ (\ref{asym-form}) depend on \KMe{$f_{\rm res}$}.
\begin{rem}
\label{rem-calc-cpx}
When a pair of complex conjugate eigenvalues is contained in ${\rm Spec}(A)$, the function ${\bf Y}(t)$ includes terms of the following form:
\begin{equation*}
\theta(t)^{\eta} \cos^{m_1}(\lambda_{im}\ln \theta(t))\sin^{m_2}(\lambda_{im}\ln \theta(t)) (\ln \theta(t))^{m_3}
\end{equation*}
for some $\eta \in \mathbb{R}$, $\lambda_{im}\in \mathbb{R}\setminus \{0\}$ and $m_1, m_2, m_3\in \mathbb{Z}_{\geq 0}$.
If such a term is included in ${\bf Y}_N(t)$ for some $N\geq 1$, the next term ${\bf Y}_{N+1}(t)$ needs lengthy calculations because the integral of the above term is required.
\end{rem}

\subsection{Justification of asymptotic expansions}
\label{section-justification}
We have observed that, in the construction of multi-order asymptotic \KMi{expansions} of blow-ups, free parameters  corresponding to blow-up power eigenvalues with negative real parts \KMi{appear}. 
Next we consider if the formal asymptotic expansion (\ref{formula-asym}) of ${\bf y}(t)$ is certainly an asymptotic expansion in the original sense.
It should be noted that, in our construction (\ref{formula-asym}), the asymptotic relation ${\bf Y}_{l}(t)\ll {\bf Y}_{l-1}(t)$ in (\ref{Y-asym}) is applied as the {\em necessary condition} to asymptotic expansions, which is not always sufficient to prove the requirement of asymptotic expansions
\begin{equation}
\label{asym-true-cond}
\frac{\KMe{\theta(t)^{\alpha_i/k} y_i(t)} - \sum_{j=0}^N Y_{j,i}(t)}{Y_{N,i}(t)} = \frac{\sum_{j=N+1}^\infty Y_{j,i}(t)}{Y_{N,i}(t)}  = o(1)\quad \text{ as }t\to t_{\max}
\end{equation}
for all $N\in \KMi{\mathbb{N}}$ and each $i = 1,\ldots, n$, \KMk{while} the residual series $\sum_{j=N+1}^\infty Y_{j,i}(t)$ converges to $0$ as $t\to t_{\max}$ from Proposition \ref{prop-convergence-Y} \KMi{and the uniqueness of solutions}.
To verify the asymptotic relation (\ref{asym-true-cond}) for (\ref{formula-asym}), 
it requires concrete forms of $f$ as well as the blow-up power-determining matrix $A$ and a nonsingular matrix $P$ inducing the Jordan canonical form of $A$.
Indeed, there can be a case that $Y_{N,i}(t)$ goes to zero much faster than $Y_{N+1,i}(t)$ as $t\to t_{\max}$ for some $N$, due to the situation that coefficients of low powers of $\theta(t)$ accidentally become zero\footnote{
\KMc{This situation includes the case that $Y_{N,i}(t) \equiv 0$.}
}.
Nevertheless, the lowest orders of $\theta(t)$ in the sequence $\{{\bf Y}_j(t)\}_{j=0}^\infty$ are \KMd{expected} to increase monotonously under a mild assumption, and hence the asymptotic series in the sense of (\ref{asym-true-cond}) can be constructed by {\em arranging several terms} so that an asymptotic series in the original sense is obtained.
In what follows, we prove that this presumption is true\KMj{, at least, under several restrictions of $f_{\rm res}$ and its derivatives.}
\KMj{The degree function $\deg_\theta$ introduced in Section \ref{section-degree} plays a key role in estimating the degree of ${\bf Y}_j(t)$ for each $j\in \mathbb{Z}_{\geq 0}$ a priori, which yields the justification of our asymptotic expansions.}

\begin{prop}
\label{prop-order-increase}
Let (\ref{formula-asym}) be the formal asymptotic expansion of ${\bf y}(t)$ \KMl{obtained in Theorem \ref{thm-smoothness-Y}} for \KMh{an asymptotically quasi-homogeneous} vector field $f$ of type $\alpha = (\alpha_1,\ldots, \alpha_n)$ and order $k+1$, and ${\bf Y}_0$ be the corresponding non-zero root of the balance law.
For $i=1,\ldots, n$, define $\gamma_i \in \mathbb{R}$ by 
\begin{align}
\label{order-fres}
\KMi{
\gamma_i := \inf_{\tilde {\bf x}\in \mathbb{R}^n} \KMj{\deg_\theta} \left( f_{i; {\rm res}}\left( \theta(\cdot)^{-\frac{1}{k}\Lambda_\alpha}\tilde {\bf x} \right) \right).
}
\end{align}
Assume that 
\begin{itemize}
\item for $i=1,\ldots, n$,
\begin{equation}
\label{order-fres-diff}
\KMi{\gamma_i \leq \inf_{\tilde {\bf x}\in \mathbb{R}^n}\inf_{\tilde {\bf v}\in \mathbb{R}^n} \KMj{\deg_\theta} \left(  \left[ Df_{{\rm res}}\left( \theta(\cdot )^{-\frac{1}{k}\Lambda_\alpha}\tilde {\bf x} \right)\theta(\cdot )^{-\frac{1}{k}\Lambda_\alpha} \tilde {\bf v} \right]_i \right)},
\end{equation}
\item
\KMj{the following inequality holds:
\begin{equation*}
\frac{\alpha_i}{k} + \gamma_i  > -1\quad \text{ for all }\quad i = 1,\ldots, n.
\end{equation*}
}
\end{itemize}
Then we have
\begin{equation}
\label{estimate-order-asym}
\KMj{\deg_\theta}({\bf Y}_N) \geq N\delta\quad \text{ for all }\quad N\in \mathbb{Z}_{\geq 0}\KMj{,}
\end{equation}
\KMj{where
\begin{equation}
\label{suff-asym}
\delta := \min \left\{ \min_{l=1,\ldots, n}\left\{\frac{\alpha_l}{k} + \gamma_l \right\} + 1, \, \min_{ \substack{ i=1,\ldots, n\\ {\rm Re}\, \lambda_i < 0}} \left( -{\rm Re}\, \lambda_i \right)\right\} > 0.
\end{equation}
}
\end{prop}

The proof is left to Appendix \ref{section-appendix-proof}.
\begin{rem}
The asymptotic quasi-homogeneity of $f$ implies
\begin{equation*}
\KMi{ \gamma_i \geq -\frac{k+\alpha_i}{k} }
\end{equation*}
and hence the quantity \KMi{$\delta$ in (\ref{suff-asym}) is always nonnegative}.
\KMj{In particular, $\delta$ becomes positive under our assumption.}
\KMi{When} $f$ is polynomial, we have the following sharper estimate:
\begin{equation*}
\gamma_i \geq -\frac{k+\alpha_i-1}{k},
\end{equation*}
indicating $\delta \geq 1/k$ \KMc{if ${\rm Spec}(A)\subset \{\lambda\in \mathbb{C} \mid {\rm Re}\,\lambda \geq 0\}$}.
Note that the statement itself also holds for the case that each $\alpha_i$ is {\em nonnegative real numbers (namely, not always integers)} except $\alpha_i \equiv 0$ for all $i$, by generalizing the quasi-homogeneity of $f$ in Definition \ref{dfn-AQH} to the $n$-tuple $\alpha = (\alpha_1,\ldots, \alpha_n)\in (\mathbb{R}_{\geq 0})^n \setminus \{{\bf 0}\}$ as the type of $f$.
\end{rem}
The \KMj{proposition} shows that the magnitude of $\theta(t)$ in the series $\{{\bf Y}_j(t)\}_{j=0}^\infty$ {\em eventually increases} in the sense of (\ref{estimate-order-asym}).
Therefore, arranging $\{{\bf Y}_j(t)\}_{j=0}^\infty$ componentwise, we can show that, for {\em any finite} sequence $\{{\bf Y}_j(t)\}_{j=0}^N$, the arranged sequence includes the finite order asymptotic expansion of ${\bf Y}(t)$.
The precise statement is shown below.

\begin{thm}[Justification of asymptotic expansion]
\label{thm-asym-exp}
Consider \KMi{an asymptotically quasi-homogeneous vector field $f$ of type $\alpha$ and order $k+1$.}
Let (\ref{formula-asym}) be the formal asymptotic expansion of ${\bf y}(t)$ \KMl{obtained in Theorem \ref{thm-smoothness-Y}}.
Suppose that all assumptions in Proposition \ref{prop-order-increase} are satisfied.
Then, for any $\KMl{N}\in \mathbb{N}$, there exist a natural number $N_1 = N_1(\KMl{N})$ satisfying
\begin{equation*}
\lim_{\KMl{N}\to \infty}N_1(\KMl{N}) = \infty
\end{equation*}
and an arranged sequence $\{\bar {\bf Y}_j(t)\}_{j=1}^{\KMl{N}}$ of the finite sequence $\{{\bf Y}_j\}_{j=1}^{\KMl{N}}$ obtained by the permutation of indices $\{1,\ldots, \KMl{N}\}$ such that 
the finite sum
\begin{equation*}
\theta(t)^{-\KMf{ \frac{1}{k}\Lambda_\alpha }}  \sum_{j=1}^{N_1} \bar {\bf Y}_{j}(t)
\end{equation*}
is the $N_1$-th order asymptotic expansion of ${\bf y}(t)$ as $t\to t_{\max}$ in the sense that
(\ref{asym-true-cond}) holds with $N = N_1$\KMh{.}
\end{thm}

\begin{proof}
Let $N_1 := \min\{\lceil N\delta \rceil, N\}$, where $\delta > 0$ is the number given in (\ref{suff-asym}).
Obviously $N_1\to \infty$ holds as $N\to \infty$.
Proposition \ref{prop-order-increase} indicates that all terms ${\bf Y}_j$ with $\KMj{\deg_\theta}({\bf Y}_j) \leq N_1$ must be included in the sequence $\{{\bf Y}_j\}_{j=1}^N$.
The sequence $\{{\bf Y}_j\}_{j=1}^N$ can be arranged {\em componentwise} to $\{\bar {\bf Y}_j\}_{j=1}^N$ so that $\KMj{\deg_\theta}(\bar {\bf Y}_j)$ is monotonously increasing in $j$.
It follows by construction that
\begin{equation*}
\lim_{t\to t_{\max}}\frac{\bar Y_{j+1, i}(t)}{ \bar Y_{j,i}(t)} = 0,\quad j=0,\cdots, N_1-1,
\end{equation*}
where the case $j=0$ is excluded when $Y_{j,i}(t) = \bar Y_{j,i}(t) = 0$.
Proposition \ref{prop-order-increase} also indicates that $\KMj{\deg_\theta}({\bf Y}_j) \geq N_1$ for $j \geq N+1$.
Moreover, Proposition \ref{prop-convergence-Y} shows $\sum_{j=N+1}^\infty {\bf Y}_j(t) \to 0$ as $t\to t_{\max}$.
In particular,
\begin{equation*}
\lim_{t\to t_{\max}}\frac{ \sum_{j=N+1}^\infty Y_{j,i}(t) }{ \bar Y_{N_1,i}(t)} = 0
\end{equation*}
holds and consequently we have
\begin{equation*}
\lim_{t\to t_{\max}}\frac{ \sum_{j=N_1+1}^N \bar Y_{j,i}(t) + \sum_{j=N+1}^\infty Y_{j,i}(t) }{\bar Y_{N_1,i}(t)} = 0
\end{equation*}
holds for each $i=1,\cdots, n$.
\end{proof}

Consequently, our formal asymptotic expansion (\ref{formula-asym}) provides the true asymptotic expansion of the blow-up solution ${\bf y}(t)$ in the sense of Theorem \ref{thm-asym-exp} \KMj{under mild assumptions to the asymptotic degree of the residual term $f_{\rm res}$}.

%% file: examples_1.tex
\label{section-examples}

Examples of asymptotic expansions of blow-up solutions are \KMg{collected}.
In some examples, correspondence of algebraic information for describing asymptotic expansions to dynamics at infinity is also revealed.

\subsection{One-dimensional ODEs}
\label{section-ex-1dim}
First we demonstrate our methodology to one-dimensional ODEs to see effectiveness of the methodology and interpretations of results. 

\subsubsection{A simple example}
The first example is 
\begin{equation}
\label{ex1-1dim-1}
y' = -y + y^3\KMk{,\quad {}' = \frac{d}{dt}}.
\end{equation}
If the initial point $y(0) > 0$ is sufficiently large, the corresponding solution would blow up in a finite time.
To describe a blow-up solution precisely, we apply the (homogeneous) parabolic compactification and the time-scale desingularization to (\ref{ex1-1dim-1}).
First note that the ODE (\ref{ex1-1dim-1}) is asymptotically homogeneous (namely $\alpha = (1)$) of order $k+1 = 3$, in particular $k=2$.
\par
\KMj{Our concern here is the asymptotic behavior of blow-up solutions of the following form:}
\begin{equation*}
y(t) = \theta(t)^{-1/2}Y(t), 
\end{equation*}
which yields the following equation solving $Y(t)$:
\begin{equation}
\label{system-asymptotic-1dim}
Y' = -Y + \theta(t)^{-1}\left\{ - \frac{1}{2}Y + Y^3\right\}.
\end{equation}
Note that the existence of such blow-up solutions is discussed in \cite{asym2}.
Under the asymptotic expansion of the positive blow-up solution:
\begin{equation*}
Y(t) = \sum_{n=0}^\infty Y_n(t)\quad \text{ with }\quad \lim_{t\to t_{\max}}Y(t)= Y_0 > 0,
\end{equation*}
the balance law requires $Y_0 = 1/\sqrt{2}$, which is the coefficient of the principal term of $y(t)$.
The blow-up power-determining matrix at $Y_0$, coinciding with the blow-up power eigenvalue, is
\begin{equation*}
\left\{ - \frac{1}{2} + 3Y^2\right\}_{Y = Y_0} = 1.
\end{equation*}
This eigenvalue has no contributions to $Y_n(t)$ with $n\geq 1$.
Next we shall calculate the second term $Y_1(t)$.
Using the Taylor expansion at $Y_0$, the system (\ref{system-asymptotic-1dim}) is written by
\begin{equation*}
\sum_{n=1}^\infty Y_n' = -\sum_{n=0}^\infty Y_n + \theta(t)^{-1}\sum_{m=1}^\infty \frac{1}{m!} \left\{\frac{d^m}{dY^m}\left( -\frac{1}{2}Y+Y^3\right)\right\}_{Y=Y_0}\left(\sum_{n=1}^\infty Y_n\right)^m ,
\end{equation*}
where the balance law is applied to cancel the principal \KMg{terms}.
The governing equation for $Y_1$ near $t=t_{\max}$ is then
\begin{equation*}
Y_1' = -Y_0 + \theta(t)^{-1}Y_1 = -\frac{1}{\sqrt{2}} + \theta(t)^{-1}Y_1,
\end{equation*}
which is solved \KMd{to obtain {\em a bounded} solution by (\ref{formula-2nd}):}
\begin{equation*}
\KMd{Y_1(t) = \theta(t)^{-1}\left[ -\int_{t}^{t_{\max}} \theta(s) \left(-\frac{1}{\sqrt{2}}\right)ds\right] = \frac{1}{2\sqrt{2}}\theta(t).
}
\end{equation*}
\KMd{This is consistent with the asymptotic relation $Y_1(t) \ll Y_0$ as $t\to t_{\max}$.}
The third term $Y_2(t)$ is also calculated.
The governing equation for $Y_2$ obtained by the similar way to $Y_1$ is
\begin{align}
\notag
Y_2' &= -Y_1 + \theta(t)^{-1} \left\{ Y_2 + \frac{3\sqrt{2}}{2}Y_1^2 + Y_1^3 \right\} \\
\label{system-asymptotic-1dim-3rd}
	&= -\frac{1}{2\sqrt{2}}\theta(t)  + \theta(t)^{-1}Y_2 + \frac{3\sqrt{2}}{16} \theta(t) +  \frac{1}{16\sqrt{2}} \theta(t)^2,
\end{align}
where we have used the fact that
\begin{equation*}
\frac{d^m}{dY^m}\left( -\frac{1}{2}Y+Y^3\right) \equiv 0 \quad \text{ for }\quad m\geq 4,
\end{equation*}
and all terms involving $Y_1'$ are cancelled.
Standard theory of ODEs yields that the general solution of the homogeneous equation $Y_2' = \theta(t)^{-1}Y_2$ is $c_2\theta(t)^{-1}$ with a constant $c_2$, which is already appeared in calculations of $Y_1$.
The bounded solution of (\ref{system-asymptotic-1dim-3rd}) is then obtained by (\ref{formula-(j+1)}) with $j=2$ as follows:
\begin{align}
\label{ex1-Y2}
Y_2(t) &= \theta(t)^{-1}\left[- \int_t^{t_{\max}} \theta(s)\left\{  -\frac{\sqrt{2}}{16}\theta(s) + \KMf{\frac{1}{16\sqrt{2}}} \theta(s)^2 \right\} ds \right]
	=  \frac{\sqrt{2}}{48} \theta(t)^2 - \KMf{\frac{1}{64\sqrt{2}}} \theta(t)^3,
\end{align}
which is consistent with the asymptotic assumption $Y_2(t) \ll Y_1(t)$ as $t\to t_{\max}$.
Remark that the above form does not directly determine the corresponding term in $Y(t)$.
In particular, coefficients of $\KMd{\theta(t)^m}$ with $\KMd{m}\geq 3$ can change depending on $Y_n(t)$, $n\geq 3$.
On the other hand, the similar calculations yield that there is no terms of $\theta(t)^2$ in $Y_n(t)$, $n\geq 3$, by the asymptotic assumption.
\KMd{
Indeed, the constant $\delta$ defined in (\ref{suff-asym}) is estimated as follows.
First, $\alpha = 1$ and $k=2$, and hence $\gamma_1 = -1/2$ in (\ref{order-fres}).
Because there are no negative blow-up power eigenvalues, $\delta$ is evaluated as
\begin{equation*}
\delta = \frac{1}{2} + 1 - \frac{1}{2} = 1.
\end{equation*}
\KMg{Proposition} \ref{prop-order-increase} indicates that ${\rm ord}_\theta(Y_3) \geq 3$ and hence there is no term $\theta(t)^m$ with $m< 3$ in $Y_n(t)$, $n\geq 3$.
In other words,} the coefficient of $\theta(t)^2$ is determined as $\sqrt{2}/48$.
Summarizing the above arguments, we have the following result.

\begin{thm}
\label{thm-ex1-1dim-1}
The system (\ref{ex1-1dim-1}) admits a blow-up solution with the following third order asymptotic expansion as $t\to t_{\max}$:
\begin{equation*}
y(t) \sim \frac{1}{\sqrt{2}}\theta(t)^{-1/2} + \frac{1}{2\sqrt{2}}\theta(t)^{1/2} +\frac{\sqrt{2}}{48}\theta(t)^{3/2}\quad \text{ as }\quad t\to t_{\max}.
\end{equation*}
\end{thm}

The higher-order asymptotic expansion of $y(t)$ is also obtained by calculating $Y_n(t)$, $n\geq 3$, in the similar way to the above arguments.
In the above expression, the parameter dependence of solutions appears in $t_{\max}$, which admits a various choice of initial points inducing blow-up solutions.

\begin{rem}[Analyticity of $Y(t)$ at $t=t_{\max}$]\rm
\label{rem-ex1-analytic-Y}
The solution $Y(t)$ of (\ref{system-asymptotic-1dim}) is actually real-analytic at $t=t_{\max}$.
Indeed, (\ref{ex1-1dim-1}) is formally transformed into
\begin{equation*}
2y^{-3} y' = 2(-y^{-2}+1).
\end{equation*}
Using the identity
\begin{equation*}
(-1+y^{-2})\KMf{'} = (y^{-2})' = -2y^{-3}\KMg{y'},
\end{equation*}
the above equation becomes
\begin{equation*}
(1-y^{-2})' = 2(1-y^{-2}).
\end{equation*}
The standard method of variable separation is thus applied to obtain
\begin{equation*}
y(t) = \left(1 - \left(1-\frac{1}{y_0^2}\right)e^{2t}\right)^{-1/2}
\end{equation*}
for blow-up solutions with $y(0) = y_0 > 1$.
Letting 
\begin{equation*}
1-\frac{1}{y_0^2} = e^{-2t_{\max}},
\end{equation*}
the above solution is rewritten by
\begin{equation}
\label{ex1-sol-exact}
y(t) = \left(1 - e^{-2\theta(t)}\right)^{-1/2}.
\end{equation}
The above function is also written by $y(t) = \left(2\theta(t)h(t)\right)^{-1/2}$, where
\begin{equation*}
h(t) \equiv \frac{1-e^{-2\theta(t)}}{2\theta(t)}
\end{equation*}
is real-analytic at $t=t_{\max}$ with $\lim_{t\to t_{\max}} h(t) = 1$. 
It is concluded that $y(t) = \theta(t)^{-1/2}Y(t)$, where $Y(t) = \frac{1}{\sqrt{2}}h(t)^{-1/2}$ is real analytic at $t = t_{\max}$.
\end{rem}

\KMg{
\begin{rem}
The function $h(t)$ appeared in Remark \ref{rem-ex1-analytic-Y} is expressed as the generating function of the Bernoulli numbers \KMl{(e.g. \cite{AIK2014})}.
Indeed, the generating function is
\begin{equation*}
B(t) = \frac{t}{1-e^{-t}} = \sum_{l\geq 0}\frac{B_l}{l!}t^l,
\end{equation*}
where $\{B_l\}_{l\geq 0}$ are the Bernoulli numbers.
Using the function $B(t)$, the exact solution $y(t)$ in (\ref{ex1-sol-exact}) is rewritten as
\begin{align*}
y(t) &= \frac{1}{\sqrt{2}}\theta(t)^{-1/2} B(2\theta(t))^{1/2} =  \frac{1}{\sqrt{2}}\theta(t)^{-1/2}\{1+ (B(2\theta(t))-1)\}^{1/2} \\
	&= \frac{1}{\sqrt{2}}\theta(t)^{-1/2} \sum_{l\geq 0}\begin{pmatrix}
1/2 \\ l
\end{pmatrix} (B(2\theta(t)) - 1)^l.
\end{align*}
From the binomial expansion and 
\begin{equation*}
(B(2\theta(t))-1)^m = \left(\sum_{l\geq 1} \frac{B_l}{l!} (2\theta(t))^l \right)^m,
\end{equation*}
we have the expression
\begin{equation*}
y(t) = \frac{1}{\sqrt{2}} \sum_{n\geq 0}a_n \theta(t)^{-\frac{1}{2} +n},
\end{equation*}
where $a_0 = 1$ and
\begin{align*}
a_n &= \sum_{l=1}^n \sum_{\substack{I = (i_1,\ldots, i_n) \\ |I| = l, w(I)=n}} \begin{pmatrix}
1/2 \\ l
\end{pmatrix}\frac{1}{l!} \frac{2^{i_1 + 2i_2 + \cdots + ni_n}}{(1!)^{i_1}(2!)^{i_2}\cdots (n!)^{i_n}} \frac{l!}{i_1! \cdots i_n!}B_1^{i_1}\cdots B_n^{i_n}\\
	&= \sum_{l=1}^n \sum_{\substack{I = (i_1,\ldots, i_n) \\ |I| = l, w(I)=n}} \begin{pmatrix}
1/2 \\ l
\end{pmatrix}\frac{1}{l!} \frac{2^n}{(1!)^{i_1}(2!)^{i_2}\cdots (n!)^{i_n}} \frac{l!}{i_1! \cdots i_n!}B_1^{i_1}\cdots B_n^{i_n},
\end{align*}
where the multi-index $I=(i_1,\ldots, i_n)$ in the above sum runs over $\mathbb{Z}_{\geq 0}^n$ satisfying
\begin{equation*}
|I| \equiv i_1+\cdots + i_n = l,\quad w(I) \equiv i_1 + 2i_2 + \cdots + ni_n = n.
\end{equation*}
The first three coefficients are 
\begin{align*}
a_0 &= 1,\quad a_1 = B_1 = \frac{1}{2},\\
a_2 &= \frac{1}{2}\frac{2^2}{2!}B_2 + \frac{1}{2}\left(-\frac{1}{2}\right)\frac{2^2}{2!} B_1^2 = B_2 - \frac{1}{2}B_1^2 = \frac{1}{24},
\end{align*}
which yield the coincidence of coefficients with $y(t)$ which we have obtained in Theorem \ref{thm-ex1-1dim-1}.
The above observation shows the validity of our expansion methodology.
On the other hand, the coefficient $a_3$ is calculated as
\begin{align*}
a_3 &= \begin{pmatrix}
1/2 \\ 1
\end{pmatrix}\frac{2^3}{(3!)^3}\frac{1!}{1!}B_3^3 + \begin{pmatrix}
1/2 \\ 2
\end{pmatrix}\frac{2^3}{(1!)^1 (2!)^1}\frac{2!}{1!1!}B_1^1 B_2^1 +  \begin{pmatrix}
1/2 \\ 3
\end{pmatrix}\frac{2^3}{(1!)^3} \frac{3!}{3!}B_1^3\\
	&= 0 - B_1B_2 + \frac{1}{2}B_1^3 = -\frac{1}{48},
\end{align*}
which indicates that we require calculating $Y_3(t)$ to obtain the correct coefficient of $\theta(t)^3$, which is different from that obtained in (\ref{ex1-Y2}).
\end{rem}
}

\subsubsection{Ishiwata-Yazaki's example}
The next example concerns with blow-up solutions of the following system:
\begin{equation}
\label{IY}
u' = a u^{\frac{a+1}{a}}v,\quad v' = a v^{\frac{a+1}{a}}u,
\end{equation}
where \KMg{$a\in (0,1)$} is a parameter.
In (\ref{IY}), the following results are obtained in preceding works, which are originally obtained in \cite{IY2003} and revisited in \cite{Mat2019}.
\begin{rem}[cf. \cite{IY2003, Mat2019}]
\label{rem-IY}
\KMg{Consider} initial points $u(0), v(0) > 0$.
If $u(0) \not = v(0)$, then the solution $(u(t), v(t))$ blows up at $t=t_{\max} < \infty$ with the blow-up rate $O(\theta(t)^{-a})$.
On the other hand, if $u(0) = v(0)$, the solution $(u(t), v(t))$ blows up at $t=t_{\max} < \infty$ with the blow-up rate $O(\theta(t)^{-a/(a+1)})$.
\end{rem}
The \KMk{above} blow-up mechanism based on dynamics at infinity is discussed in \cite{Mat2019}.
The main interest here is the multi-order asymptotic expansion of blow-up solutions for (\ref{IY}) mentioned in Remark \ref{rem-IY}\KMg{.} 

\par
First note that the system (\ref{IY}) has the first integral
\begin{equation*}
I = I(u, v) := v^{1-\frac{1}{a}} - u^{1-\frac{1}{a}}.
\end{equation*}
In other words, the time-differential of the functional $I$ along \KMg{solutions} of (\ref{IY}) is always zero, equivalently the level set $\{I(t) = C\}$ is invariant for (\ref{IY}).
Indeed,
\begin{align*}
\frac{d}{dt}I(t) &= \left(1-\frac{1}{a}\right)\left\{ v^{-1/a}v' - u^{-1/a}u'\right\}\\
	&=  (a-1)\left\{ v^{-1/a}v^{\frac{a+1}{a}}u - u^{-1/a} u^{\frac{a+1}{a}}v \right\}\\
	&=  (a-1)\left\{ vu - uv \right\} = 0.
\end{align*}
Using this functional, 
\begin{equation*}
v = \left(I + u^\frac{a-1}{a}\right)^{\frac{a}{a-1}}
\end{equation*}
and the system (\ref{IY}) is then reduced to a one-dimensional ODE
\begin{equation}
\label{IY-1dim}
u' = a u^{\frac{a+1}{a}}\left( u^{1 - \frac{1}{a}} + I \right)^{\frac{a}{a-1}}.
\end{equation}
Blow-up solutions of the rate $O(\theta(t)^{-a})$ correspond to $I \not = 0$, while those of the rate $O(\theta(t)^{-a/(a+1)})$ correspond to $I=0$.
We pay attention to the case $u(0) > v(0)$ when $I\not = 0$, in which case $I>0$ holds\footnote{
Our assumption $0< a<1$ is essential for this correspondence.
}.
\par
First consider the case $I>0$, where the vector field (\ref{IY-1dim}) is asymptotically homogeneous of the order $1+a^{-1}$.
Using the asymptotic expansion
\begin{equation*}
u(t) = \theta(t)^{-a}U(t) = \theta(t)^{-a}\left(\sum_{n=0}^\infty U_n(t)\right),\quad \lim_{t\to t_{\max}} U(t) = U_0,\\
\end{equation*}
the system becomes
\begin{align}
\notag
U' &= a\theta(t)^{-1} \left\{ -U + U^{\frac{a+1}{a}}\left(  \theta(t)^{-(a-1)}U^{1 - \frac{1}{a}} + I \right)^{\frac{a}{a-1}}\right\}\\
\label{IY-system-U}
	&= a\theta(t)^{-1} \left\{ -U + \KMf{ I^{\frac{a}{a-1}} } U^{\frac{a+1}{a}} \left( I^{-1} \theta(t)^{-(a-1)}U^{1 - \frac{1}{a}} + 1 \right)^{\frac{a}{a-1}}\right\}.
\end{align}
\KMf{The principal component in the above vector field is extracted as follows.}
Let
\begin{equation*}
f(t; U) \KMf{:=} \left(I^{-1} \theta(t)^{1-a}U^{1 - \frac{1}{a}} + 1 \right)^{\frac{a}{a-1}}
\end{equation*}
following (\ref{IY-1dim}).
\KMf{Because} $\theta(t)^{1-a}U^{1 - \frac{1}{a}} \ll I$ \KMg{as $t\to t_{\max}$ holds} by $a<1$, $I>0$ fixed and $U\to U_0$ \KMg{as $t\to t_{\max}$, then} the nonlinear term \KMf{$f(t; U)$ itself} converges to $1$ \KMg{as $t\to t_{\max}$}.
\KMf{We therefore know that the principal component of the vector field (\ref{IY-system-U}) is $a\theta(t)^{-1} \left\{ -U +  I^{\frac{a}{a-1}} U^{\frac{a+1}{a}}\right\}$, provided $I>0$.}
Using the binomial series, 
\begin{align}
\left( I^{-1} \theta(t)^{1-a}U^{1 - \frac{1}{a}} + 1 \right)^{\frac{a}{a-1}} 
\label{IY-vf-series}
	&= \sum_{k=0}^\infty \begin{pmatrix}
\frac{a}{a-1} \\ k
\end{pmatrix}\left( I^{-1}\theta(t)^{1-a}U^{\frac{a-1}{a}} \right)^k,
\end{align}
where
\begin{equation*}
\begin{pmatrix}
\frac{a}{a-1} \\ k
\end{pmatrix} = \frac{\left(\frac{a}{a-1}\right)_k}{k!},\quad 
\left(\frac{a}{a-1}\right)_k = \frac{a}{a-1} \left(\frac{a}{a-1}-1\right)\left(\frac{a}{a-1}-2\right) \cdots \left(\frac{a}{a-1}-k+1\right).
\end{equation*}
The latter is the well-known Pochhammer symbol.
The balance law then yields
\begin{equation*}
\KMj{-U_0 +  I^{\frac{a}{a-1}} U_0^{\frac{a+1}{a}} = 0}\quad \Rightarrow \quad U_0 = I^{-a^2/(a-1)}.
\end{equation*}
The \KMf{corresponding} blow-up power-determining matrix is
\begin{align*}
\frac{d}{dU}\left(a \left\{ -U + U^{\frac{a+1}{a}} I^{\frac{a}{a-1}} \right\} \right)_{U=U_0} &=a \left\{ -1 + \frac{a+1}{a}U_0^{\frac{1}{a}} I^{\frac{a}{a-1}} \right\} = 1\KMj{.}
\end{align*}
With the binomial series (\ref{IY-vf-series}), the governing equation for the second term $U_1(t)$ is
\begin{align*}
U_1' &= \theta(t)^{-1}U_1  + a\theta(t)^{-1} U_0^{\frac{a+1}{a}}I^{\frac{a}{a-1}} \sum_{k=1}^\infty \begin{pmatrix}
\frac{a}{a-1} \\ k
\end{pmatrix}\left( I^{-1}\theta(t)^{1-a}U_0^{\frac{a-1}{a}} \right)^k\\
	&= \theta(t)^{-1}U_1  + a\theta(t)^{-1} I^{-\frac{a^2}{a-1}} \sum_{k=1}^\infty \begin{pmatrix}
\frac{a}{a-1} \\ k
\end{pmatrix}\left( I^{-1}\theta(t)^{1-a}U_0^{\frac{a-1}{a}} \right)^k.
\end{align*}
The bounded solution is
\begin{align*}
U_1(t) &= \theta(t)^{-1} \left[ -a \int_t^{t_{\max}} I^{-\frac{a^2}{a-1}}\sum_{k=1}^\infty \begin{pmatrix}
\frac{a}{a-1} \\ k
\end{pmatrix}
\left( I^{-1}\theta(s)^{1-a}U_0^{1 - \frac{1}{a}} \right)^k ds \right]\\
	&= \theta(t)^{-1} \left[ \KMg{- a} I^{-\frac{a^2}{a-1}}\sum_{k=1}^\infty \begin{pmatrix}
\frac{a}{a-1} \\ k
\end{pmatrix}\left( I^{-1}U_0^{1 - \frac{1}{a}} \right)^k \frac{\theta(t)^{k(1-a)+1}}{k(1-a)+1} \right].
\end{align*}
\KMg{Note that all series appeared in the above equalities have positive convergence radii.}
Therefore
\begin{equation*}
U_1(t) = I^{\frac{-2a^2+1}{a-1}}
\frac{a^2}{(1-a)(2-a)} \theta(t)^{1-a}
 - a I^{-\frac{a^2}{a-1}} \sum_{k=2}^\infty \begin{pmatrix}
\frac{a}{a-1} \\ k
\end{pmatrix}\left( I^{-1}U_0^{1 - \frac{1}{a}} \right)^k \frac{\theta(t)^{\KMd{k(1-a)}}}{k(1-a)+1}.
\end{equation*}
\KMd{
As in the case of (\ref{ex1-1dim-1}), coefficients of $\theta(t)^{k(1-a)}$ also depend on $U_n(t)$ with $n\geq 2$, while the coefficient of $\theta(t)^{1-a}$ is determined here from the asymptotic assumption.
Indeed, the constant $\delta$ defined in (\ref{suff-asym}) is estimated as follows.
First, $\alpha = 1$ and $k=1/a$.
Next, we shall extract the residual term $f_{\rm res}$ so that $\gamma_1$ in (\ref{order-fres}) is evaluated.
Note that $k$ is not always an integer, but the concept of quasi-homogeneity is generalized to any positive numbers.
In (\ref{IY-1dim}), the vector field is written as 
\begin{equation*}
f(u) \equiv au^{\frac{a+1}{a}}\left(u^{1-\frac{1}{a}} + I\right)^{\frac{a}{a-1}} = aI^{\frac{a}{a-1}} u^{\frac{a+1}{a}} + a\left( u^{\frac{a+1}{a}}\left(u^{1-\frac{1}{a}} + I\right)^{\frac{a}{a-1}} - I^{\frac{a}{a-1}} u^{\frac{a+1}{a}} \right),
\end{equation*}
and hence the residual term $f_{\rm res}(u)$ is
\begin{equation*}
f_{\rm res}(u) = a\left( u^{\frac{a+1}{a}}\left(u^{1-\frac{1}{a}} + I\right)^{\frac{a}{a-1}} - I^{\frac{a}{a-1}} u^{\frac{a+1}{a}} \right).
\end{equation*}
Therefore, using $\alpha / k = a$, we have 
\begin{equation*}
f_{\rm res}(\theta(t)^{-a}u) = a\left( \theta(t)^{-(a+1)}u^{\frac{a+1}{a}}\left(\theta(t)^{-a+1} u^{1-\frac{1}{a}} + I\right)^{\frac{a}{a-1}} - I^{\frac{a}{a-1}} \theta(t)^{-(a+1)} u^{\frac{a+1}{a}} \right) = O(\theta(t)^{-2a})
\end{equation*}
as $t\to t_{\max}$, and hence $\gamma_1 = -2a$.
Because there are no negative blow-up power eigenvalues, $\delta$ is evaluated as
\begin{equation*}
\delta = \frac{\alpha}{k} + 1 + \gamma_1 = a + 1 - 2a = 1-a.
\end{equation*}
\KMg{Proposition} \ref{prop-order-increase} (with the assertion right after the \KMl{proposition}) indicates that ${\rm ord}_\theta(\KMf{U_m}) \geq m(1-a)$ for $m\geq 2$ and hence there is no term $\theta(t)^{1-a}$ in $\KMf{U_n}(t)$, $n\geq 2$.
}
As a summary, we have the second order asymptotic expansion 
\begin{equation*}
U(t) \sim I^{-a^2/(a-1)} 
+ I^{\frac{-2a^2+1}{a-1}} \frac{a^2}{(1-a)(2-a)} \theta(t)^{1-a}
\end{equation*}
as $t\to t_{\max}$.
Back to the original coordinate, we have the \KMg{second order} asymptotic expansion of blow-up solution
\begin{align*}
u(t) &\sim I^{-a^2/(a-1)} \theta(t)^{-a} + I^{\frac{-2a^2+1}{a-1}}
\frac{a^2}{(1-a)(2-a)} \theta(t)^{1-2a},\\
v(t) &\equiv \left(u(t)^{(a-1)/a} + I\right)^{a/(a-1)} = I^{a/(a-1)}\left(I^{-1}u(t)^{(a-1)/a} + 1\right)^{a/(a-1)}\\
	&\sim I^{\frac{a}{a-1}} - I^{\frac{1}{a-1}-a} \frac{a}{1-a} \theta(t)^{1-a}
\end{align*}
as $t\to t_{\max}$.
The higher-order asymptotic expansion is derived in the similar way.
\par
Our asymptotic expansion has {\em two} parameters: $t_{\max}$ and $I \equiv v(0)^{1-\frac{1}{a}} - u(0)^{1-\frac{1}{a}}$, \KMj{whose dynamical interpretation is mentioned in Part II \cite{asym2}}.
\par
\bigskip
Next consider the case $I=0$:
\begin{equation}
\label{IY-0}
u' = a u^{2+\frac{1}{a}},
\end{equation}
which is homogeneous of the order $2+a^{-1}$.
Note that the order is different from that for $I>0$.
Using the asymptotic expansion
\begin{equation*}
u(t) = \theta(t)^{-a/(a+1)}U(t) = \theta(t)^{-a/(a+1)}\left(\sum_{n=0}^\infty U_n(t)\right),\quad \lim_{t\to t_{\max}} U(t) = U_0,\\
\end{equation*}
the system becomes
\begin{equation}
\label{IY-I0}
U' = a\theta(t)^{-1} \left(-\frac{1}{a+1} U + U^{2+\frac{1}{a}}\right).
\end{equation}
The balance law yields
\begin{equation*}
U_0 = \left( \frac{1}{a+1}\right)^{\frac{a}{a+1}}.
\end{equation*}
Letting
\begin{equation*}
f_0(U):= a \left(-\frac{1}{a+1} U + U^{2+\frac{1}{a}}\right),
\end{equation*}
the equation (\ref{IY-I0}) is rewritten by 
\begin{equation*}
U' = \theta(t)^{-1} \left\{(U-U_0) + \sum_{m=2}^\infty \frac{1}{m!}\left(\frac{d^m}{dU^m}U^{2+\frac{1}{a}}\right)_{U=U_0}(U-U_0)^m\right\}
\end{equation*}
by using the Taylor expansion of $f_0$ at $U_0$.
The governing equation of the second term $U_1$ is thus
\begin{equation*}
U_1' = \theta(t)^{-1}U_1
\end{equation*}
whose general solution is $U_1(t) = c_1\theta(t)^{-1}$. 
In the above equation we have used the fact that $\frac{df_0}{dU}(U_0) = 1$\KMj{.}
The asymptotic assumption requires $c_1 = 0$, and hence $U_1(t) \equiv 0$.
The same arguments yield $U_n(t) \equiv 0$ for all $n\geq 1$.
Consequently, the asymptotic expansion of the blow-up solution is
\begin{equation}
\label{IY-sol-I0}
u(t) \sim \left( \frac{1}{a+1}\right)^{\frac{a}{a+1}}\theta(t)^{-a/(a+1)},
\end{equation}
which is in fact the exact solution of (\ref{IY-0}).
As a summary, we obtain the following result for asymptotic expansions of blow-up solutions.
\begin{thm}
The system (\ref{IY}) with fixed $a\in (0,1)$ admits blow-up solutions with the following second-order asymptotic expansions as $t\to t_{\max}$:
\begin{align*}
u(t) &\sim I^{-a^2/(a-1)} \theta(t)^{-a} + I^{\frac{-2a^2+1}{a-1}}
\frac{a^2}{(1-a)(2-a)} \theta(t)^{1-2a},\quad
v(t) \sim I^{\frac{a}{a-1}} - I^{\frac{1}{a-1}-a} \frac{a}{1-a} \theta(t)^{1-a}
\end{align*}
with $I> 0$, where $I$ is a free parameter.
When $I=0$, (\ref{IY}) admits a blow-up solution
\begin{equation*}
u(t) = v(t) = \left( \frac{1}{a+1}\right)^{\frac{a}{a+1}}\theta(t)^{-a/(a+1)}.
\end{equation*}
\end{thm}

Note that the solution through the above argument coincides with that obtained by the method of separation of variables in (\ref{IY-0}).

\subsection{Two-phase flow model}
\label{section-ex-2phase}

We move to multi-dimensional examples.
In this example\KMg{,} the following system is concerned (see e.g. \cite{KSS2003, Mat2018} for the details of the system):
\begin{equation}
\label{two-fluid-1}
\begin{cases}
\beta' = vB_1(\beta) - c\beta - c_1, & \\
v' =  v^2 B_2(\beta) - cv - c_2, & 
\end{cases}\quad {}'=\frac{d}{dt},
\end{equation}
where
\begin{equation*}
B_1(\beta) = \frac{(\beta-\rho_1)(\beta-\rho_2)}{\beta},\quad B_2(\beta) = \frac{\beta^2- \rho_1\rho_2}{2\beta^2}
\end{equation*}
with $\rho_2 > \rho_1 > 0$, 
\begin{equation*}
c = \frac{v_R B_1(\beta_R) - v_L B_1(\beta_L)}{\beta_R - \beta_L}
\end{equation*}
and $(c_1,c_2) = (c_{1L}, c_{2L})$ or $(c_{1R}, c_{2R})$, where 
\begin{equation*}
\label{constants-two-phase}
\begin{cases}
c_{1L} = v_L B_1(\beta_L) - c\beta_L, & \\
c_{2L} = v_L^2 B_2(\beta_L) -cv_L, & \\
\end{cases}
\quad
\begin{cases}
c_{1R} = v_R B_1(\beta_R) - c\beta_R, & \\
c_{2R} = v_R^2 B_2(\beta_R) -cv_R. & \\
\end{cases}
\end{equation*}
Points $(\beta_L, v_L)$ and $(\beta_R, v_R)$ are given in advance.
The following property immediately holds by simple calculations.
\begin{prop}[\cite{Mat2018}]
The system (\ref{two-fluid-1}) is asymptotically quasi-homogeneous of type $(0,1)$ and order $2$.
\end{prop}

Following arguments in \cite{Mat2018}, we observe that there is a blow-up solution with the asymptotic behavior
\begin{equation}
\label{two-fluid-2}
\beta(t) \sim \rho_2,\quad v(t)\sim V_0\theta(t)^{-1}\quad\text{ as }\quad t\to t_{\max}-0,
\end{equation}
which is consistent with arguments in \cite{KSS2003}.
In particular, type-I blow-up solutions are observed.

\par
Our main concern here is to derive multi-order asymptotic expansion of the blow-up solution (\ref{two-fluid-2}) for (\ref{two-fluid-1}).
To this end, write the blow-up solution $(\beta(t), v(t))$ as follows:
\begin{align}
\notag
\beta(t) &= b(t),\quad v(t) = \theta(t)^{-1}V(t),\\
\label{form-b-V}
b(t) &= \sum_{n=0}^\infty b_n(t) \equiv b_0 + \tilde b(t),\quad b_0 = \rho_2,\quad b_n(t) \ll b_{n-1}(t)\quad (t\to t_{\max}-0),\quad n\geq 1,\\
\notag
V(t) &= \sum_{n=0}^\infty V_n(t)\equiv V_0 + \tilde V(t),\quad V_n(t) \ll V_{n-1}(t)\quad (t\to t_{\max}-0),\quad n\geq 1.
\end{align}

The balance law which \KMl{$(b_0, V_0)$} satisfies can be easily derived. 
Substituting the form (\ref{form-b-V}) into (\ref{two-fluid-1}), we have
\begin{align*}
\beta' &= b' \\
	&= \theta(t)^{-1} V B_1(b) - cb - c_1,\\
v' &= \theta(t)^{-2} V + \theta(t)^{-1}V'\\
	&= \theta(t)^{-2} V^2 B_2(b) - c\theta(t)^{-1}V - c_2.
\end{align*}
Dividing the first equation by $\theta(t)^{0} \equiv 1$ and the second equation by $\theta(t)^{-1}$, we have 
\begin{align}
\label{2phase-asym-main}
\frac{d}{dt}\begin{pmatrix}
b \\
V
\end{pmatrix} = \theta(t)^{-1} \begin{pmatrix}
V B_1(b) \\
-V  + V^2 B_2(b)
\end{pmatrix} - \begin{pmatrix}
cb + c_1 \\
cV  + \theta(t) c_2
\end{pmatrix}
\end{align}

The balance law is then 
\begin{equation*}
\begin{pmatrix}
V_0 B_1(b_0) \\
-V_0  + V_0^2 B_2(b_0)
\end{pmatrix} = \begin{pmatrix}
0 \\ 
0
\end{pmatrix}.
\end{equation*}
that is,
\begin{equation*}
V_0  \frac{(b_0-\rho_1)(b_0-\rho_2)}{\beta} = 0,\quad -V_0 + V_0^2 \frac{b_0^2- \rho_1\rho_2}{2b_0^2} = 0.
\end{equation*}
In the present case, $b_0 = \rho_2$ is already determined as the principal term of $b(t)$, which satisfies the first equation.
Substituting $b_0 = \rho_2$ into the second equation, we have
$V_0 = 2\rho_2 / (\rho_2- \rho_1)$, provided $V_0 \not =0$.
As a summary, the root of the balance law (under (\ref{two-fluid-2})) is uniquely determined by
\begin{equation}
\label{balance-two-phase}
(b_0, V_0) = \left(\rho_2, \frac{2\rho_2}{\rho_2- \rho_1}\right).
\end{equation}
Then the evolutionary system for the residual terms $(\tilde b, \tilde V)$ is given below:
\begin{align*}
\frac{d}{dt}\begin{pmatrix}
\tilde b \\
\tilde V
\end{pmatrix} = \theta(t)^{-1} \begin{pmatrix}
V_0\{ B_1(b_0 + \tilde b) - B_1(b_0) \} + \tilde V B_1(b_0 + \tilde b) \\
-\tilde V  + V_0^2 \{B_2(b_0 + \tilde b) - B_2(b_0)\} + (2V_0 \tilde V + \tilde V^2) B_2(b_0 + \tilde b)
\end{pmatrix} - \begin{pmatrix}
c(b_0 + \tilde b) + c_1 \\
c(V_0 + \tilde V)  + \theta(t) c_2
\end{pmatrix}
\end{align*}
Next consider the second term \KMl{$(b_1(t), V_1(t))$}.
The strategy is essentially the same as the previous example, while the vector field is rational.
Direct calculations yield
\begin{align*}
\frac{d}{d\beta}B_1(\beta) 
	&= 1 - \rho_1\rho_2\beta^{-2},\quad
\frac{d}{d\beta}B_2(\beta) = \rho_1\rho_2 \beta^{-3}.
\end{align*}
Letting
\begin{align*}
\bar f (b_0 + \tilde b,V_0 + \tilde V) &:= \KMm{\begin{pmatrix}
V_0\{ B_1(b_0 + \tilde b) - B_1(b_0) \} + \tilde V B_1(b_0 + \tilde b) \\
-\tilde V  + V_0^2 \{B_2(b_0 + \tilde b) - B_2(b_0)\} + (2V_0 \tilde V + \tilde V^2) B_2(b_0 + \tilde b)
\end{pmatrix} }\\
&\equiv -\Lambda_\alpha \begin{pmatrix}
b \\ V
\end{pmatrix}+ f_{\alpha, k}(b,V)\quad (\text{with }k=1)
\end{align*}
\KMm{with the aid of the balance law}, we have
\begin{align*}
D\KMg{\bar f}(b_0, V_0) &= \begin{pmatrix}
V_0\frac{d}{d\tilde b}B_1(\beta)|_{\beta=b_0} & B_1(b_0)\\
V_
0^2 \frac{d}{d\beta}B_2(\beta)|_{\beta=b_0} & -1 + 2V_0 B_2(b_0)\\
\end{pmatrix}
	= \begin{pmatrix}
V_0(1 - \rho_1\rho_2b_0^{-2}) & \frac{(b_0-\rho_1)(b_0-\rho_2)}{\KMf{b_0}}\\
V_0^2 \rho_1\rho_2 b_0^{-3} & -1 + 2V_0 \frac{b_0^2- \rho_1\rho_2}{2b_0^2}\\
\end{pmatrix},
\end{align*}
which is the blow-up power determining matrix associated with the blow-up solution (\ref{two-fluid-2}).
Using (\ref{balance-two-phase}), we have
\begin{align*}
A\equiv D\KMg{\bar f}(b_0, V_0) &= \begin{pmatrix}
 \frac{2\rho_2}{\rho_2- \rho_1}(1 - \rho_1\rho_2^{-1}) & 0\\
( \frac{2\rho_2}{\rho_2- \rho_1})^2 \rho_1\rho_2 b_0^{-3} & -1 + 2 \frac{2\rho_2}{\rho_2- \rho_1} \frac{\rho_2- \rho_1}{2\rho_2}\\
\end{pmatrix}\\
	&=\begin{pmatrix}
2 & 0\\
\frac{4\rho_1}{(\rho_2- \rho_1)^2} & 1\\
\end{pmatrix}.
\end{align*}
We immediately know that eigenvalues are $2, 1$, which are independent of $\rho_1, \rho_2$ and make {\em no} contributions to asymptotic behavior of the blow-up solution.
Associated eigenvectors are
\begin{align*}
\begin{pmatrix}
2 & 0\\
\frac{4\rho_1}{(\rho_2- \rho_1)^2} & 1\\
\end{pmatrix}\begin{pmatrix}
x_1 \\ x_2
\end{pmatrix} = 2\begin{pmatrix}
x_1 \\ x_2
\end{pmatrix} \quad &\Rightarrow \quad \begin{pmatrix}
x_1 \\ x_2
\end{pmatrix} = \begin{pmatrix}
1 \\ \frac{4\rho_1}{(\rho_2- \rho_1)^2}
\end{pmatrix},\\
\begin{pmatrix}
2 & 0\\
\frac{4\rho_1}{(\rho_2- \rho_1)^2} & 1\\
\end{pmatrix}\begin{pmatrix}
x_1 \\ x_2
\end{pmatrix} = \begin{pmatrix}
x_1 \\ x_2
\end{pmatrix} \quad & \Rightarrow \quad \begin{pmatrix}
x_1 \\ x_2
\end{pmatrix} = \begin{pmatrix}
0 \\ 1
\end{pmatrix}.
\end{align*}
Define
\begin{equation*}
P := \begin{pmatrix}
1 & 0 \\
\frac{4\rho_1}{(\rho_2- \rho_1)^2} & 1
\end{pmatrix}\quad \Leftrightarrow\quad P^{-1} = \begin{pmatrix}
1 & 0 \\
-\frac{4\rho_1}{(\rho_2- \rho_1)^2} & 1
\end{pmatrix}
\end{equation*}
and multiply $P^{-1}$ to (\ref{2phase-asym-main}) from the left:
\begin{align}
\label{2phase-asym-main-diag}
\frac{d}{dt}P^{-1}\begin{pmatrix}
b \\
V
\end{pmatrix} = \theta(t)^{-1} \left[ \begin{pmatrix}
\KMf{2} & 0 \\
0 & \KMf{1}
\end{pmatrix}
P^{-1}\begin{pmatrix}
b \\
V
\end{pmatrix} + P^{-1}\KMl{\bar f_R}(b, V) \right] - c P^{-1}\begin{pmatrix}
b \\
V
\end{pmatrix} - P^{-1}\begin{pmatrix}
c_1 \\
\theta(t) c_2
\end{pmatrix},
\end{align}
where
\begin{equation}
\label{g_res}
\KMl{\bar f_{\rm R}}(b, V) = \KMg{\bar f}(b, V) - D\KMg{\bar f}(b_0, V_0)\begin{pmatrix}
\tilde b \\ \tilde V
\end{pmatrix}
\end{equation}
is the second order residual term of $\KMg{\bar f}$ whose details are \KMl{shown in Remark \ref{rem-detail-barf_2nd} below}, 
and we have used 
\begin{equation*}
\begin{pmatrix}
\KMf{2} & 0 \\
0 & \KMf{1}
\end{pmatrix}
P^{-1}\begin{pmatrix}
b_0 \\
V_0
\end{pmatrix} + P^{-1}\KMl{\bar f_{\rm R}}(b_0, V_0) = \begin{pmatrix}
0 \\
0
\end{pmatrix}.
\end{equation*}

\begin{rem}
\label{rem-detail-barf_2nd}
\KMl{The detailed expression of $\bar f_R(b,V)$ around the root $(b, V) = (b_0, V_0)$ is derived by the Taylor's expansion at $(b, V) = (b_0, V_0)$:}
\begin{align*}
V B_1(b) &= \KMl{\left(V_0 \tilde b + \tilde V \tilde b^2 \right) \frac{(2b_0 V_0^{-1} + \tilde b)}{b_0 + \tilde b} \quad (\text{from (\ref{balance-two-phase})}) } \\
	&\KMl{ = V_0 \tilde b \left[  2V_0^{-1} + \frac{1-2V_0^{-1}}{b_0} \tilde b -\frac{(1-2V_0^{-1})}{b_0 (b_0  + \tilde b)}\tilde b^2 \right] + \tilde V \tilde b \left[ 2V_0^{-1} + \frac{1  -2V_0^{-1}}{b_0 (b_0+\tilde b)} \tilde b \right] }\\
&\KMl{ \equiv 2\tilde b + \bar f_{R, 1}(b,V)},
\end{align*}
\begin{align*}
-V  + V^2 B_2(b) 
	&= - (V_0 + \tilde V) + \frac{1}{2} V_0^2 \left[ 1 - \frac{\rho_1 \rho_2}{b_0^2} + 2\frac{\rho_1 \rho_2}{b_0^3}\tilde b + \rho_1\rho_2 \left\{ \frac{ - 3b_0\tilde b^2  - 2 \tilde b^3}{b_0^3(b_0 + \tilde b)^2}  \right\} \right] \\
	&\quad + V_0\tilde V \left[ 1 - \frac{\rho_1 \rho_2}{b_0^2} + \rho_1\rho_2 \left\{ \frac{2b_0\tilde b + \tilde b^2}{b_0^2(b_0 + \tilde b)^2}\right\} \right] 
 + \frac{1}{2} \tilde V^2\left\{1 - \frac{\rho_1 \rho_2}{(b_0 + \tilde b)^2} \right\}\\
%
%
	&\KMl{\equiv -(V_0 + \tilde V) + \frac{1}{2}V_0^2 \left(1 - \frac{\rho_1\rho_2}{b_0^2} + 2\frac{\rho_1\rho_2}{b_0^3} \tilde b \right) + V_0 \tilde V \left( 1 - \frac{\rho_1\rho_2}{b_0^2} \right) + \bar f_{R,2}(b,V) }\\
	&\KMl{= \frac{\rho_1 \rho_2}{b_0^3}V_0^2  \tilde b + \left[-1 + V_0 \left\{ 1 - \frac{\rho_1 \rho_2}{b_0^2}\right\} \right] \tilde V  + \bar f_{R,2}(b,V) \quad (\text{from (\ref{balance-two-phase})})}\\
	&\KMl{= \frac{4\rho_1}{(\rho_2-\rho_1)^2}  \tilde b +  \tilde V + \bar f_{R,2}(b,V)},
\end{align*}
where we have used the following identities:
\KMl{
\begin{align*}
-\frac{\rho_1 \rho_2 }{(b_0 + \tilde b)^2} 
	&= - \frac{\rho_1 \rho_2}{b_0^2} + \rho_1\rho_2 \left\{ \frac{2b_0\tilde b + \tilde b^2}{b_0^2(b_0 + \tilde b)^2}\right\} \\
	&= -\frac{\rho_1 \rho_2}{b_0^2} + 2 \frac{\rho_1 \rho_2}{b_0^3}\tilde b + \rho_1\rho_2 \left\{ \frac{ -3b_0 \tilde b^2  - 2 \tilde b^3}{b_0^3(b_0 + \tilde b)^2}  \right\}. 
\end{align*}}
%
Therefore we obtain the expansion
\begin{align*}
 \begin{pmatrix}
V B_1(b) \\
-V  + V^2 B_2(b)
\end{pmatrix} &= \begin{pmatrix}
2 & 0\\
\frac{4\rho_1}{(\rho_2- \rho_1)^2} & 1\\
\end{pmatrix}\begin{pmatrix}
\tilde b \\
\tilde V
\end{pmatrix} + \KMl{\bar f_R}(b, V)
\end{align*}
around $(b_0, V_0)^T$, where
\begin{align*}
\KMl{\bar f_R} (b, V) &= (\KMl{\bar f_{R,1}}(b, V), \KMl{\bar f_{R,2}}(b, V))^T,\\
\KMl{\bar f_{R,1}}(b, V) 
	&= \KMl{ \frac{V_0-2}{b_0} \tilde b^2 - \frac{V_0-2}{b_0 (b_0 + \tilde b)}\tilde b^3  + \frac{2}{V_0} \tilde V \tilde b + \frac{(1-2V_0^{-1})}{b_0 (b_0 + \tilde b)} \tilde V \tilde b^2, } \\
\KMl{\bar f_{R,2}}(b, V) 
	&= \KMl{\frac{1}{2}V_0^2  \rho_1\rho_2 \left\{ \frac{- 3b_0 - 2\tilde b }{b_0^3(b_0 + \tilde b)^2} \right\}\tilde b^2 +  V_0 \rho_1\rho_2 \left\{ \frac{2b_0 + \tilde b}{b_0^2(b_0 + \tilde b)^2}\right\} \tilde b \tilde V + \frac{1}{2}\tilde V^2\left\{1 - \frac{\rho_1 \rho_2}{(b_0 + \tilde b)^2} \right\}.}
\end{align*}
\end{rem}

Let
\begin{equation*}
\begin{pmatrix}
e_1 \\
W_1
\end{pmatrix} := 
P^{-1}
\begin{pmatrix}
\tilde b \\
\tilde V
\end{pmatrix}
\end{equation*}
and the governing terms are collected to form the equation for $(e_1, W_1)^T$:
\begin{align*}
\label{2phase-asym-main-diag-2nd}
\frac{d}{dt}\begin{pmatrix}
e_1 \\
W_1
\end{pmatrix} = \theta(t)^{-1} \begin{pmatrix}
2 & 0 \\
0 & 1
\end{pmatrix}
\begin{pmatrix}
e_1 \\
W_1
\end{pmatrix} - \begin{pmatrix}
1 & 0 \\
-\frac{4\rho_1}{(\rho_2- \rho_1)^2} & 1
\end{pmatrix}\begin{pmatrix}
cb_0 + c_1 \\
cV_0
\end{pmatrix},
\end{align*} 
where the remaining term $\theta(t) c_2$ in the second component is omitted \KMg{because}, in the inhomogeneous term, it is obviously negligible compared with the constant term $cV_0$ as $t\to t_{\max}$.
This system is easily solved to obtain the bounded solution as follows:
\begin{align*}
e_1' &= 2\theta(t)^{-1}e_1 - (c\rho_2 + c_1)\\
	&\Leftrightarrow \quad e_1 = \theta(t)^{-2} \left\{ (c\rho_2 + c_1)\int_t^{t_{\max}} \theta(s)^{2} ds \right\},\\
W_1' &= \theta(t)^{-1}W_1 + \frac{4\rho_1}{(\rho_2- \rho_1)^2} (c\rho_2 + c_1) - \left(\frac{2c\rho_2}{\rho_2 - \rho_1} \right)\\
	&\Leftrightarrow \quad W_1 = \theta(t)^{-1} \left[ -\left\{\frac{4\rho_1}{(\rho_2- \rho_1)^2} (c\rho_2 + c_1) - \frac{2c\rho_2}{\rho_2 - \rho_1} \right\} \int_t^{t_{\max}} \theta(s) ds \right].
\end{align*}
In particular, the \KMd{bounded} solution $(e_1(t), W_1(t))$ is
\begin{align*}
e_1(t) &= \KMd{ \frac{1}{3}(c\rho_2 + c_1) \theta(t)},\quad 
W_1(t) = -\frac{1}{2} \left\{\frac{4\rho_1}{(\rho_2- \rho_1)^2} (c\rho_2 + c_1) - \frac{2c\rho_2}{\rho_2 - \rho_1}  \right\} \theta(t).
\end{align*}
In the original coordinate, we have
\begin{align*}
\begin{pmatrix}
b_1 \\
V_1
\end{pmatrix} &=  \begin{pmatrix}
1 & 0 \\
\frac{4\rho_1}{(\rho_2- \rho_1)^2} & 1
\end{pmatrix}\begin{pmatrix}
e_1 \\
W_1
\end{pmatrix}\\
	&= \theta(t) \begin{pmatrix}
1 & 0 \\
\frac{4\rho_1}{(\rho_2- \rho_1)^2} & 1
\end{pmatrix}\begin{pmatrix}
\frac{1}{3}(c\rho_2 + c_1)  \\
-\frac{1}{2} \left\{\frac{4\rho_1}{(\rho_2- \rho_1)^2} (c\rho_2 + c_1) - \frac{2c\rho_2}{\rho_2 - \rho_1} \right\} 
\end{pmatrix},\\
b_1 &= \frac{1}{3}(c\rho_2 + c_1)\theta(t) ,\\
V_1 
	&= \left[  -\frac{2\rho_1}{3(\rho_2- \rho_1)^2} (c\rho_2 + c_1) + \left(\frac{c\rho_2}{\rho_2 - \rho_1}\right) \right] \theta(t).
\end{align*} 
\KMd{
Now the constant $\delta$ defined in (\ref{suff-asym}) is estimated as follows.
As already observed, we know that $\alpha = (0,1)$ and $k=1$.
We know that $(\gamma_1, \gamma_2) = (0, -1)$.
Because there are no negative blow-up power eigenvalues, $\delta$ is evaluated as
\begin{equation*}
\delta = \min \left\{ \frac{0}{1} +0, \frac{1}{1} - 1 \right\} + 1 = 1.
\end{equation*}
\KMg{Proposition} \ref{prop-order-increase} indicates that ${\rm ord}_\theta({\bf Y}_m) \geq m$ for $m\geq 2$ and hence there is no term $\theta(t)^m$ with $m\leq 1$ in ${\bf Y}_n(t) = (b_n(t), V_n(t))$, $n\geq 2$.
}
As a summary, we obtain the following result.
\begin{thm}
The system (\ref{two-fluid-1}) with $(c_1, c_2) = (c_{1L}, c_{2L})$ admits a blow-up solution with the following second order asymptotic expansion as $t\to t_{\max}$:
\begin{align*}
\beta(t) &\sim \rho_2 + \frac{1}{3}(c\rho_2 + c_1)\theta(t),\quad
v(t) \sim \frac{2\rho_2}{\rho_2- \rho_1}\theta(t)^{-1} + \left[  -\frac{2\rho_1}{3(\rho_2- \rho_1)^2} (c\rho_2 + c_1) + \frac{c\rho_2}{\rho_2 - \rho_1} \right].
\end{align*}
\end{thm}
We omit further calculations to obtain the third term $(b_2(t), V_2(t))$, but it should be mentioned that the neglected term $\theta(t)c_2$ in the above calculations should be added because it is nontrivial whether $\theta(t)c_2$ is negligible compared with inhomogeneous terms in the system solving $(b_2, V_2)$.
In the present problem, $\theta(t)c_2$ must be included in the $(b_2, V_2)$-system \KMg{because} $cV_1$ has the same order as $\theta(t)c_2$.
\KMl{The inhomogeneous term ${\bf g}_2$ in the equation for $(b_2(t), V_2(t))$ can be the calculated from $\bar f_R$ extracting appropriate terms (cf. Remark \ref{rem-detail-barf_2nd}) and remaining terms mentioned above.}

\begin{rem}
\KMg{Because} $b_0 = \rho_2 > 0$, the function \KMl{$\bar f_R(b, V)$} in (\ref{g_res}) is real analytic in a neighborhood of $(\tilde b, \tilde V) = (0,0)$.
\end{rem}

\subsection{Andrews' system I}
\label{section-ex-Andrews1}

Consider the following system (cf. \cite{A2002, Mat2019})\footnote{
There is a typo of the system in \cite{Mat2019} (Equation (4.7)). 
The correct form is (\ref{Andrews1}) below.
The rest of arguments in \cite{Mat2019} is developed for the correct system (\ref{Andrews1}).
}:
\begin{equation}
\label{Andrews1}
\begin{cases}
\displaystyle{
\frac{du}{dt} = \frac{1}{\sin \theta} vu^2 - \frac{2a\cos\theta}{\sin \theta}u^3
}
, & \\
\displaystyle{
\frac{dv}{dt} =  \frac{a}{\sin \theta} uv^2 + \frac{1-a}{\sin \theta} \frac{uv^3}{v + 2\cos\theta u}
}, &
\end{cases}
\end{equation}
where $a\in (0,1)$ and $\theta\in (0,\pi/2)$ are constant parameters.

\begin{rem}
\label{rem-Andrews1}
The following results are known for (\ref{Andrews1}) \KMg{with} initial points $(u(0), v(0))$ such that both components are positive (cf. \cite{A2002, Mat2019}).
\begin{enumerate}
\item When $a < 1/2$, all solutions with sufficiently large initial \KMg{points} blow up at $t=t_{\max} < \infty$ with the blow-up rate $O(\theta(t)^{-1/2})$ as $t\to t_{\max}$. 
\item When $a\in (1/2, 1)$, there is a blow-up solution admitting the asymptotic behavior $O(\{\theta(t)^{-1} \log \theta(t)^{-1}\}^{1/2})$ as $t\to t_{\max}$.
\item When $a= 1/2$\KMf{,} there is a blow-up solution admitting the asymptotic behavior $O(\{\theta(t)^{-1} (\log \theta(t)^{-1})^{1/2}\}^{1/2})$ as $t\to t_{\max}$.
\end{enumerate}
\end{rem}
\par
Our \KMk{interest here is the} blow-up behavior for $a\in (0,1/2)$.
We easily see that the system (\ref{Andrews1}) is homogeneous of the order $3$, and arguments in \cite{Mat2019} show that the first term of stationary blow-up solutions has the form
\begin{equation*}
u(t) = O(\theta(t)^{-1/2}),\quad v(t) = O(\theta(t)^{-1/2}).
\end{equation*}

Before expanding $(u(t), v(t))$ directly, introduce
\begin{equation}
\label{trans-Andrews1}
w = u\cos\theta,\quad s = \frac{t}{\sin\theta \cos\theta},
\end{equation}
in which case we have
\begin{equation*}
\frac{du}{ds} = \frac{du}{dt}\frac{dt}{ds} = \sin\theta \cos\theta \frac{du}{dt},\quad \frac{dw}{ds} = \cos \theta\frac{du}{ds} =  \sin\theta \cos^2 \theta\frac{du}{dt}.
\end{equation*}
The vector field (\ref{Andrews1}) is then transformed into
\begin{align*}
\frac{dw}{ds} &= \sin\theta \cos^2 \theta\frac{du}{dt}\\
	&= \sin\theta \cos^2 \theta \left(\frac{1}{\sin \theta} vu^2 - \frac{2a\cos\theta}{\sin \theta}u^3\right)\\
	&= \sin\theta \cos^2 \theta \left(\frac{1}{\sin \theta} v\left(\frac{w}{\cos\theta}\right)^2 - \frac{2a\cos\theta}{\sin \theta} \left(\frac{w}{\cos\theta}\right)^3 \right)\\
	&= vw^2 - 2a w^3,\\
\frac{dv}{ds} &= \sin\theta \cos \theta\frac{dv}{dt} \\
	&= \sin\theta \cos \theta \left( \frac{a}{\sin \theta} \left(\frac{w}{\cos\theta}\right)v^2 + \frac{1-a}{\sin \theta} \frac{\left(\frac{w}{\cos\theta}\right)v^3}{v + 2w}\right)\\
	&= awv^2 + (1-a) \frac{wv^3}{v + 2w}.
\end{align*}
As a summary, the system (\ref{Andrews1}) is transformed into the following rational vector field independent of $\theta$:
\begin{equation}
\frac{dw}{ds} = w^2 (v - 2a w),\quad \frac{dv}{ds} = wv^2 \frac{v + 2aw}{v+2w}.
\end{equation}
\KMg{Because} $\theta \in (0,\pi/2)$, then $\sin \theta > 0$ and $\cos \theta > 0$\KMg{,} and hence dynamics of $(u,v)$ in $t$-\KMm{timescale} and $(w, v)$ in $s$-variable are mutually smoothly equivalent. 
\par
In what follows we shall consider the blow-up solution with the following blow-up rate
\begin{equation*}
w(t) = O(\tilde \theta(s)^{-1/2}),\quad v(t) = O(\tilde \theta(s)^{-1/2}),\quad \tilde \theta(s) = s_{\max}-s = \KMc{\frac{2}{\sin(2\theta)}}\theta(t).
\end{equation*}
Expand the solution $(w(s),v(s))$ as the asymptotic series
\begin{align}
\notag
w(s) &= \tilde \theta(s)^{-1/2}W(s) \equiv \tilde \theta(s)^{-1/2}\sum_{n=0}^\infty W_n(s),\quad W_n(s) \ll W_{n-1}(s),\quad \lim_{t\to t_{\max}}W_n(s) = W_0, \\
\label{series-Andrews1}
v(s) &= \tilde \theta(s)^{-1/2}V(s) \equiv \tilde \theta(s)^{-1/2}\sum_{n=0}^\infty V_n(s),\quad V_n(s) \ll V_{n-1}(s), \quad \lim_{t\to t_{\max}}V_n(s) = V_0.
\end{align}

The balance law for the solution under the assumption $W_0, V_0 \not = 0$ is 
\begin{equation}
\label{balance-Andrews1}
\frac{1}{2} = W_0 V_0 - 2aW_0^2,\quad 
\frac{1}{2} = W_0 V_0\frac{V_0 + 2aW_0}{V_0 + 2W_0}.
\end{equation}
We further assume $V_0 + 2W_0 \not = 0$ at a moment. 
Then we have
\begin{align*}
& V_0 - 2aW_0 = V_0\frac{V_0 + 2aW_0}{V_0 + 2W_0}\\
&\Leftrightarrow\quad (V_0 + 2W_0)(V_0 - 2aW_0) = V_0(V_0 + 2aW_0)\\
&\Leftrightarrow\quad V_0^2 - 2aV_0W_0 + 2V_0W_0 - 4aW_0^2 = V_0^2 + 2aV_0W_0\\
&\Leftrightarrow\quad 2(1- 2a)V_0 = 4aW_0.
\end{align*}
From this identity, we have the constraint among $U_0$ and $V_0$ so that the balance law must be satisfied:
\begin{equation}
\label{Cu-Cv}
V_0 = \frac{2a}{1-2a}W_0.
 \end{equation}
\par
Substituting (\ref{Cu-Cv}) into, say the first component of the \KMg{balance law}, we have
\begin{align*}
\frac{1}{2} = \frac{2a-2a(1-2a)}{1-2a} W_0^2 = \frac{4a^2}{1-2a} W_0^2.
\end{align*}
We therefore have
\begin{equation*}
W_0 = \pm \sqrt{\frac{1-2a}{8a^2}},\quad V_0 = \frac{2a}{1-2a}W_0 = \pm \sqrt{\frac{1}{2(1-2a)}}.
\end{equation*}
We shall choose the positive root according to the setting mentioned in Remark \ref{rem-Andrews1}. 
Moreover, the root $V_0$ can be achieved as a real value by the assumption of $a$.
\par
\bigskip
Next we derive the equation for the residual terms $W(s) - W_0$ and $V(s)-V_0$.
First, substituting (\ref{series-Andrews1}) into (\ref{Andrews1}), we have
\begin{align*}
\sum_{n=1}^\infty \frac{dW_n}{ds} &= \tilde \theta(s)^{-1}\left[ -\frac{1}{2}\sum_{n=0}^\infty W_n +\left( \sum_{n=0}^\infty W_n\right)^2  \left\{ \sum_{n=0}^\infty V_n - 2a\sum_{n=0}^\infty W_n \right\} \right],\\
\sum_{n=1}^\infty \frac{dV_n}{ds} &= \tilde \theta(s)^{-1}\left\{-\frac{1}{2}\sum_{n=0}^\infty V_n + 
\left( \sum_{n=0}^\infty W_n \right) \left( \sum_{n=0}^\infty V_n \right)^2  \frac{\sum_{n=0}^\infty V_n + 2a \sum_{n=0}^\infty W_n}{ \sum_{n=0}^\infty V_n + 2\sum_{n=0}^\infty W_n}
 \right\}.
\end{align*}
In the next step, we derive the governing equation for $(W_1(t), V_1(t))$.
\KMg{Because} the original system is homogeneous, we have the following governing system for $(W_1(t), V_1(t))$, which is exactly the linearized system of (\ref{Andrews1}) at $(W_0, V_0)$:

\begin{align}
\notag
\frac{d}{ds}\begin{pmatrix}
W_1 \\ V_1 
\end{pmatrix} &= \tilde \theta(s)^{-1}\left\{-\frac{1}{2} \begin{pmatrix}
W_1 \\ V_1 
\end{pmatrix}\right.\\
\notag
&\quad \left. + 
 \begin{pmatrix}
2V_0W_0 - 6aW_0^2 & W_0^2 \\
V_0^2 \frac{V_0+2aW_0}{V_0+2W_0} + W_0V_0^2 \frac{2a(V_0+2W_0) - 2(V_0+2aW_0)}{(V_0+2W_0)^2} & 2V_0W_0 \frac{V_0+2aW_0}{V_0+2W_0} + W_0V_0^2 \frac{(V_0+2W_0) - (V_0+2aW_0)}{(V_0+2W_0)^2} \\ 
\end{pmatrix}
\KMf{
\begin{pmatrix}
W_1 \\ V_1 
\end{pmatrix} 
}
\right\}\\
\label{2nd-Andrews1}
	&\equiv \tilde \theta(s)^{-1} \left\{-\frac{1}{2} I_2 + C(W_0, V_0; a) 
\right\} \begin{pmatrix}
W_1 \\ V_1 
\end{pmatrix}\KMf{,}
\end{align}
where
\begin{align*}
&C(W_0, V_0; a) 
=\begin{pmatrix}
2V_0W_0 - 6aW_0^2 & W_0^2 \\
V_0^2 \left\{ \frac{V_0+2aW_0}{V_0+2W_0} - W_0 \frac{2(1-a)V_0}{(V_0+2W_0)^2} \right\} & V_0W_0 \left\{ 2 \frac{V_0+2aW_0}{V_0+2W_0} + V_0 \frac{2(1-a)W_0}{(V_0+2W_0)^2} \right\} \\ 
\end{pmatrix}.
\end{align*}

Calculation of blow-up power eigenvalues is reduced to that of the matrix $C(W_0, V_0; a)$, which still looks complicating.
Using the balance law, we derive the simpler form.

\begin{rem}
\label{rem-Andrews1-identity}
We derive several identities among $V_0$ and $W_0$ so that eigenpairs of $C(W_0, V_0; a)$ are easily calculated.
First, using the identity (\ref{Cu-Cv}), we have
\begin{equation}
\label{2}
V_0+2W_0 = \left(1+ \frac{1-2a}{a} \right) V_0 = \frac{1-a}{a} V_0
\end{equation}
and 
\begin{equation}
\label{3}
V_0+2aW_0 = (1+ 1-2a ) V_0 = 2(1-a) V_0.
\end{equation}
Using these identities, we have
\begin{equation}
\label{4}
\frac{V_0+2a W_0}{V_0+2 W_0} = \frac{2(1-a) V_0}{\frac{1-a}{a} V_0} = 2a
\end{equation}
and
\begin{align}
\notag
&\frac{2(1-a) V_0 W_0}{(V_0+2W_0)^2} = \frac{a}{1-a} \frac{2(1-a) W_0}{V_0+2W_0}\quad (\text{by }(\ref{2}))\\
\notag
&\quad = \frac{2a W_0}{V_0+2W_0} = \frac{(1-2a) V_0}{V_0+2W_0}\quad (\text{by }(\ref{Cu-Cv}))\\
\notag
&\quad = \frac{(1-2a) V_0}{\frac{1-a}{a} V_0}\quad (\text{by }(\ref{2}))\\
\label{5}
&\quad = \frac{a(1-2a)}{1-a}.
\end{align}
\end{rem}

The balance law (\ref{Cu-Cv}) and identities obtained in Remark \ref{rem-Andrews1-identity} yield
\begin{align*}
2V_0 W_0 - 6 a W_0^2 &= (6a-1) V_0 W_0,\\
V_0^2 \left\{ \frac{V_0+2aW_0}{V_0+2W_0} - \frac{2(1-a) V_0 W_0}{(V_0+2W_0)^2} \right\} &= V_0^2 \left\{ 2a - \frac{a(1-2a)}{1-a} \right\} \quad (\text{by }(\ref{4})\text{ and }(\ref{5}))\\
&= \frac{a}{1-a} V_0^2,\\
V_0 W_0 \left\{ 2 \frac{V_0+2aW_0}{V_0+2W_0} + \frac{2(1-a) V_0 W_0}{(V_0+2W_0)^2} \right\} 
&= V_0 W_0 \left\{ \KMg{4a} + \frac{a(1-2a)}{1-a} \right\} \quad (\text{by }(\ref{4})\text{ and }(\ref{5})) \\
&= \frac{a(5-6a)}{1-a} V_0 W_0.
\end{align*}
Next, from the first identity of (\ref{balance-Andrews1}), we have
\begin{align*}
\frac{1}{2} &= V_0 W_0 - 2a W_0^2\\
&= V_0 W_0 -(1-2a) V_0 W_0\quad (\text{by }(\ref{Cu-Cv}))\\ 
&= 2a V_0 W_0,
\end{align*}
which yields
\begin{equation}
\label{6}
V_0 W_0 = \frac{1}{4a}.
\end{equation}
\KMg{We use} this identity to obtain
\begin{align*}
(6a-1) V_0 W_0 &= \frac{3}{2} - \frac{1}{4a}\quad (\text{by }(\ref{6})),\\
\frac{a(5-6a)}{1-a} V_0 W_0 &= \frac{3}{2} - \frac{1}{4(1-a)}.
\end{align*}
Therefore the matrix $C(W_0, V_0; a)$ is simplified as follows:
\begin{align*}
C(W_0, V_0; a) &= 
\begin{pmatrix}
\frac{3}{2} - \frac{1}{4a}  & W_0^2 \\
\frac{a}{1-a} V_0^2 & \frac{3}{2} - \frac{1}{4(1-a)} 
\end{pmatrix}.
\end{align*}
In particular, we obtain
\begin{align*}
A - I_2 &\equiv C(W_0, V_0; a) - \frac{3}{2} I_2 =
\begin{pmatrix}
 - \frac{1}{4a}  & W_0^2 \\
\frac{a}{1-a} V_0^2 & - \frac{1}{4(1-a)}
\end{pmatrix}.
\end{align*}
Using the identity (\ref{6}) again, we know that the matrix $A - I_2$ has the eigenvalue $0$ associating the eigenvector $(W_0, V_0)$\KMj{.}
In this case, another eigenvalue is easily calculated through the trace:
\begin{equation*}
{\rm trace}(A - I_2) = -\frac{1}{4a(1-a)}.
\end{equation*}
From \begin{align*}
-\frac{1}{4a(1-a)} I_2 - \begin{pmatrix}
 - \frac{1}{4a}  & W_0^2 \\
\frac{a}{1-a} V_0^2 & - \frac{1}{4(1-a)}
\end{pmatrix}
 =
\begin{pmatrix}
 \frac{a}{1-a} V_0 W_0  & -W_0^2 \\
-\frac{a}{1-a} V_0^2 & V_0 W_0
\end{pmatrix},
\end{align*}
the associated eigenvector is
\begin{equation*}
\left(W_0,\, \frac{a}{1-a} V_0\right)^T.
\end{equation*}
As a consequence, eigenpairs of $A$ are calculated as follows:
\begin{align*}
\left\{1, \begin{pmatrix}
\sqrt{\frac{1-2a}{8a^2}}\\
\sqrt{\frac{1}{2(1-2a)}}
\end{pmatrix}\right\},\quad 
\left\{ 1 - \frac{1}{4a(1-a)}, \begin{pmatrix}
\sqrt{\frac{1-2a}{8a^2}}\\
\frac{a}{1-a}\sqrt{\frac{1}{2(1-2a)}}
\end{pmatrix}
\right\}.
\end{align*}
Introducing a matrix 
\begin{equation*}
P = \begin{pmatrix}
\sqrt{\frac{1-2a}{8a^2}} & \sqrt{\frac{1-2a}{8a^2}} \\
\sqrt{\frac{1}{2(1-2a)}} & \frac{a}{1-a}\sqrt{\frac{1}{2(1-2a)}}
\end{pmatrix},
\end{equation*}
the blow-up power-determining matrix $A$ is diagonalized.
Note that only the eigenvalue $\lambda \equiv 1 - \frac{1}{4a(1-a)}$ contributes the second term of asymptotic expansion, \KMg{because} it is always negative whenever $a\in (0, 1/2)$.
Therefore we know that the second term $(W_1(s), V_1(s))$ consists only of a multiple of $\tilde \theta(s)^{-\lambda}$, namely
\begin{align*}
W_1(s) &= C_{1w}\tilde \theta(s)^{-1 + \frac{1}{4a(1-a)}},\quad 
V_1(s) = C_{1v}\tilde \theta(s)^{-1 + \frac{1}{4a(1-a)}}
\end{align*}
with constants $C_{1w}, C_{1v}$.

\begin{rem}
The concrete form \KMj{of the matrix $P$} is not necessary to obtain $(W_1(s), V_1(s))$ itself as the solution of linear homogeneous system of ODEs.
The concrete calculation of $P$ is necessary for \KMk{solving $(W_n, V_n)$ with $n\geq 2$} because $(W_1, V_1)$ and other lower-order terms are included as inhomogeneous terms.
\end{rem}


The remaining issue is to determine coefficients $C_{1w}, C_{1v}$.
Rewriting the system (\ref{2nd-Andrews1}),
\begin{align}
\label{2nd-Andrews-detail}
\frac{d}{ds}
\begin{pmatrix}
W_1 \\
V_1
\end{pmatrix}
&=
\tilde \theta(s)^{-1} \left[  -\begin{pmatrix}
\frac{1}{2} & 0 \\
0 & \frac{1}{2}
\end{pmatrix} + 
\left(
\begin{array}{cc}
\frac{3}{2} - \frac{1}{4a}  & \frac{1-2a}{8a^2} \\
\frac{a}{1-a} \frac{1}{2(1-2a)} & \frac{3}{2} - \frac{1}{4(1-a)} 
\end{array}
\right)
\right]
\begin{pmatrix}
W_1 \\
V_1
\end{pmatrix}
\end{align}
is our interest.
In the previous arguments we have shown that 
\begin{equation*}
W_1(s) = C_{1w}\tilde \theta(s)^{-\lambda},\quad 
V_1(s) = C_{1v}\tilde \theta(s)^{-\lambda}
\end{equation*}
Substituting this form into (\ref{2nd-Andrews-detail}), we have
\begin{align}
		\label{ren}
		\begin{cases}
		\lambda C_{1w} = \left\{ 1 - \frac{1}{4a}  \right\} C_{1w} +\frac{1-2a}{8a^2}C_{1v} ,\\
		\lambda C_{1v} = \frac{a}{2(1-a)(1-2a)} C_{1w} + \left\{ 1 - \frac{1}{4(1-a)} \right\} C_{1v}.
		\end{cases}
\end{align}
Using $\lambda =1 - \frac{1}{4a(1-a)}$, we have the \KMg{constraint}
\begin{equation*}
C_{1v} =\frac{2 a^2}{(1-a)(2a-1)}C_{1w}.
\end{equation*}
Letting $C_{1w}$ be an arbitrary parameter, we have the complete form of $(W_1(s), V_1(s))$:
\begin{align*}
W_1(s) &= C_{1w}\tilde \theta(s)^{-1 + \frac{1}{4a(1-a)}},\quad 
V_1(s) = \frac{2 a^2 \cos \theta}{(1-a)(2a-1)}C_{1w} \tilde \theta(s)^{-1 + \frac{1}{4a(1-a)}}.
\end{align*}
Because $f_{\rm res}\equiv 0$ and $f$ is analytic at $(U_0, V_0) = \left(\frac{W_0}{\cos\theta}, V_0\right)$, the Taylor expansion of $f$ for deriving equations for ${\bf Y}_n \equiv (W_n, V_n)$ with $n\geq 2$ yields that all possible terms appeared in ${\bf Y}_n$ are the linear combination of powers of $\theta(t)^{-m\lambda}$ with $m\geq 1$, where $\lambda = 1-\frac{1}{4a(1-a)}$.
Summarizing the above arguments, we have the second order asymptotic expansion of blow-up solutions:
\begin{align}
\label{Andrews1-sol}
\begin{aligned}
&w(s) \sim \sqrt{\frac{1-2 a}{8 a^2}} \tilde \theta(s)^{-1/ 2} + C_{1w} \tilde \theta(s)^{-\frac{3}{2} + \frac{1}{4a(1-a)}}, \\
&v(s) \sim \sqrt{\frac{1}{2(1-2 a)}} \tilde \theta(s)^{-1/2}+ \frac{2 a^2}{(1-a)(2a-1)}C_{1w} \tilde \theta(s)^{-\frac{3}{2} + \frac{1}{4a(1-a)}}
\end{aligned}
\end{align}
as $s\to s_{\max}$.
Interestingly, the exponent of the second term
$-\frac{3}{2} + \frac{1}{4a(1-a)}$
can be both positive and negative, \KMg{indicating that the divergent behavior can be enhanced depending on $a$. 
Nevertheless, the exponent} is always larger than $-1/2$, \KMg{because} the quantity $1 - \frac{1}{4a(1-a)}$ is always negative.
Reverting to the original $t$-\KMm{timescale} via (\ref{trans-Andrews1}), we obtain the following result.

\begin{thm}
The system (\ref{Andrews1}) with $a\in (0, 1/2)$ admits a blow-up solution with the following second-order asymptotic expansion as $t\to t_{\max}$:
\begin{align}
\label{Andrews1-sol-original}
\begin{aligned}
&u(t) \sim \sqrt{\frac{(1-2 a)\tan \theta}{8 a^2}} \theta(t)^{-1/ 2} + C \theta(t)^{-\frac{3}{2} + \frac{1}{4a(1-a)}}, \\
&v(t) \sim \frac{1}{2}\sqrt{\frac{\sin(2\theta)}{1-2 a}}\theta(t)^{-1/2}+ \frac{2 a^2 \cos\theta}{(1-a)(2a-1)}C  \theta(t)^{-\frac{3}{2} + \frac{1}{4a(1-a)}}
\end{aligned}
\end{align}
with a constant $C = C(\theta)$ with $\theta\in (0,\pi/2)$.
\end{thm}



\subsection{Andrews' system II}
\label{section-ex-Andrews2}

The next example is the following system:
\begin{align}
\label{Andrews2}
	\begin{cases}
		u' = u^2 (2av - bu),
		\\
		v' = bu v^2
	\end{cases}
\end{align}
with parameters $a, b$ with $a > 0$ and $2a > b > 0$.
Our interest here is the asymptotic expansion of blow-up solutions with $u(0),v(0)>0$.
\par
The system (\ref{Andrews2}) comes from the following three dimensional \KMg{system of ODEs} expressing the motion of crystalline curvature flow for triangles (cf. \cite{A2002}):
\begin{equation}
\label{Andrews2-original}
\begin{cases}
u_1' = u_1^2(a_2 u_2 + a_3 u_3 - a_1 u_1), & \\
u_2' = u_2^2(a_3 u_3 + a_1 u_1 - a_2 u_2), & \\
u_3' = u_3^2(a_1 u_1 + a_2 u_2 - a_3 u_3), & \\
\end{cases}
\end{equation}
where $a_1,a_2,a_3$ are positive constants.
In \cite{Mat2019}, the parameter dependence of blow-up behavior is investigated under the additional assumption $a_2 = a_3 \equiv a > 0$.
In this case, it is proved in \cite{Mat2019} that $\{u_2 = u_3\}$ is invariant under (\ref{Andrews2-original}) and that, if $2a > a_1 > 0$, the solutions with initial points satisfying $u_2(0) = u_3(0) > 0$ blow up at $t = t_{\max} < \infty$ with the blow-up rate
\begin{align*}
	u_1(t)= O(\theta(t)^{-1/2}),\quad u_2(t) = u_3(t)= O(\theta(t)^{\KMg{-1/2}})
\end{align*}
as $t\rightarrow t_{\max}-0$.

\begin{rem}
\label{fact-A2}
In prior works (e.g. \cite{AIU2017, A2002}), it is \KMg{also} observed that, if $a_1 = a_2 + a_3$, $u_1(0) > 0$ and $u_2(0) = u_3(0) > 0$, then the solution blows up in a finite time $t_{\max}$ with the blow-up rate $(\theta(t)^{-1}\log \theta(t)^{-1})^{1/2}$. 
Notice that the parameter constraint is different from our present setting.
\end{rem}
As in previous examples, we introduce 
\begin{equation}
\label{trans-Andrews2}
\tilde u = \frac{u}{\sqrt{b}},\quad \tilde v = \frac{v}{\sqrt{b}},\quad \KMg{s = b^2 t},\quad 2a = b(1+\sigma)
\end{equation}
with an auxiliary parameter $\sigma$, which transform (\ref{Andrews2}) into
\begin{align*}
\frac{d\tilde u}{ds} = \tilde u^2 \left\{ (1+\sigma) \KMm{\tilde v} - \tilde u \right\},\quad \frac{d\tilde v}{ds} = \tilde u \tilde v^2.
\end{align*}
In particular, the system becomes a {\em one}-parameter family.
Our interest here is then the blow-up solution $(\tilde u(s), \tilde v(s))$ with the following blow-up rate
\begin{equation*}
\tilde u(t) = O(\tilde \theta(s)^{-1/2}),\quad \tilde v(t) = O(\tilde \theta(s)^{-1/2}),\quad \tilde \theta(s) = s_{\max}-s = \KMg{b^2}\theta(t).
\end{equation*}
Expand the solution $(\tilde u(s), \tilde v(s))$ as the asymptotic series
\begin{align}
\notag
\tilde u(s) &= \tilde \theta(s)^{-1/2}U(s) \equiv \tilde \theta(s)^{-1/2}\sum_{n=0}^\infty U_n(s),\quad U_n(s) \ll U_{n-1}(s),\quad \lim_{s\to s_{\max}}U_n(s) = U_0, \\
\label{series-Andrews2}
\tilde v(s) &= \tilde \theta(s)^{-1/2}V(s) \equiv \tilde \theta(s)^{-1/2}\sum_{n=0}^\infty V_n(s),\quad V_n(s) \ll V_{n-1}(s), \quad \lim_{s\to s_{\max}}V_n(s) = V_0.
\end{align} 
Substituting 
\eqref{series-Andrews2} into \eqref{Andrews2}, we have
\begin{align}
\label{asymptotic-system-Andrews2}
\KMc{
\frac{dU}{ds} = \tilde \theta(s)^{-1} \left\{ -\frac{1}{2}U + U^2\{ (1+\sigma) V - U\} \right\},\quad 
\frac{dV}{ds} = \tilde \theta(s)^{-1} \left\{ -\frac{1}{2}V + UV^2\right\}.
}
\end{align}
The balance law under $(U_0, V_0)\not = (0, 0)$ requires
\begin{align*}
-\frac{1}{2} + U_0 \{(1+\sigma)V_0 - U_0 \} = 0,\quad -\frac{1}{2} + U_0V_0 = 0\KMg{,}
\end{align*}
and hence 
\begin{equation}
\label{identity-Andrews2}
1 = 2U_0\{(1+\sigma) V_0 - U_0 \},\quad 1 = 2U_0V_0\KMg{,}
\end{equation}
which are used below.
These identities yield the following consequence:
\begin{equation*}
U_0 = \sqrt{\frac{\sigma}{2}},\quad V_0 = \frac{1}{\sqrt{2\sigma}}\KMg{,}
\end{equation*}
and we \KMg{have} the first order asymptotic expansion of blow-up solutions:
\begin{equation*}
\tilde u(s)\sim \sqrt{\frac{\sigma}{2}} \tilde \theta(s)^{-1/2},\quad 
\tilde v(s)\sim \frac{1}{\sqrt{2\sigma}} \tilde \theta(s)^{-1/2}
\end{equation*}
as $s\to s_{\max}$.
\par
\bigskip
Next the second term \KMc{$(U_1(s), V_1(s))$} is calculated.
From the linearization of the vector field (\ref{asymptotic-system-Andrews2}) at $(U_0, V_0)$, the governing system for \KMc{$(U_1(s), V_1(s))$} is
\begin{align}
\label{asymptotic-2nd-Andrews2}
\frac{d}{ds}\begin{pmatrix}
U_1 \\ V_1 
\end{pmatrix} &= \tilde \theta(s)^{-1}\left\{ -\frac{1}{2}I_2 +D(U_0, V_0; \sigma) \right\} \begin{pmatrix}
U_1 \\ V_1 
\end{pmatrix}
\end{align}
where
\begin{align*}
&\KMc{
D(U_0, V_0; \sigma) =  
\begin{pmatrix}
 2(1+\sigma)UV - 3U^2 & (1+\sigma)U^2 \\
 V^2 & 2UV
\end{pmatrix}_{(U,V) = (U_0, V_0)} = 
 \begin{pmatrix}
 1 - \frac{\sigma}{2} & \frac{\sigma(1+\sigma)}{2} \\
 \frac{1}{2\sigma} & 1
\end{pmatrix}
}
\end{align*}
under the identity (\ref{identity-Andrews2}).
In this example, the blow-up power-determining matrix $A$ is
\begin{equation*}
A = -\frac{1}{2}I_2 + \KMc{D(U_0, V_0; \sigma)} 
\end{equation*}
and blow-up power eigenvalues are easily \KMg{calculated} through eigenvalues of $\KMc{D(U_0, V_0; \sigma)}$, which \KMg{solve}
\begin{align*}
\KMc{
\left( \mu - 1 + \frac{\sigma}{2} \right)(\mu - 1) - \frac{1+\sigma}{4} 
= \mu^2 + \left(\frac{\sigma}{2}-2\right)\mu + \frac{3}{4}(1-\sigma) = 0.
}
\end{align*}
The roots of this equation are
\begin{equation*}
\mu = \frac{3}{2},\quad \KMc{\frac{1}{2}(1-\sigma)}
\end{equation*}
and hence the blow-up power eigenvalues are
\begin{equation*}
\lambda \equiv \mu - \frac{1}{2} = 1,\quad \KMc{-\frac{\sigma}{2}}.
\end{equation*}
The value $1$ is an eigenvalue of $A$\KMj{.}
By our assumption $2a > b > 0$, $\sigma > 0$ always holds\KMg{,} and hence another blow-up power eigenvalue is always negative.
Therefore the second term $(U_1(s), V_1(s))$ has the following form:
\begin{equation*}
U_1(s) = C_{1u}\KMc{ \tilde \theta(s)^{\sigma/2} },\quad 
V_1(s) = C_{1v}\KMc{ \tilde \theta(s)^{\sigma/2} },
\end{equation*}
while the constraint among constants are determined by (\ref{asymptotic-2nd-Andrews2}):
\begin{align*}
-\KMg{\rho} C_{1u} &= -\frac{1}{2}C_{1u} + \KMc {\left( 1 - \frac{\sigma}{2} \right)C_{1u} + \frac{\sigma(1+\sigma)}{2}  C_{1v}} ,\\
-\KMg{\rho} C_{1v} &= -\frac{1}{2}C_{1v} + \KMc{ \frac{1}{2\sigma} C_{1u} + C_{1v}},
\end{align*}
where $\KMg{\rho} = \sigma / 2 > 0$.
We thus have
\begin{equation*}
C_{1v} = -\frac{1}{\sigma(1+\sigma)}C_{1u}
\end{equation*}
and the second order asymptotic expansion of the blow-up solution \KMc{$(\tilde u(s), \tilde v(s))$} is obtained as follows:
\begin{align*}
\tilde u(s)&\sim \KMc{
 \sqrt{\frac{\sigma}{2}} \tilde \theta(s)^{-1/2} + C_{1u} \tilde \theta(s)^{\frac{\sigma}{2}- \frac{1}{2}},
 } \\
\tilde v(s)&\sim \KMc{
\frac{1}{\sqrt{2\sigma}} \tilde \theta(s)^{-1/2}  -\frac{1}{\sigma(1+\sigma)} C_{1u}\tilde \theta(s)^{\frac{\sigma}{2}- \frac{1}{2}}
}
\end{align*}
as \KMc{$s\to s_{\max}$}.

\par
\bigskip
Finally, we shall derive the third term $(U_2(s), V_2(s))$ \KMk{to derive an interesting observation (Remark \ref{rem-Andrews2-multi})}.
From the balance law and (\ref{asymptotic-2nd-Andrews2}), the governing equation for $(U_2(t), V_2(t))$ is 
\begin{align}
\notag
\KMc{\frac{d}{ds}}\begin{pmatrix}
U_2 \\ V_2
\end{pmatrix} &= 
\KMc{
\tilde \theta(s)^{-1}\left\{ -\frac{1}{2}I_2 +D(U_0, V_0; \sigma) \right\} \begin{pmatrix}
U_2 \\ V_2 
\end{pmatrix}
}\\
\label{3rd-Andrews2}
	&\quad 
	+ \KMc{
	\tilde \theta(s)^{-1}
	\begin{pmatrix}
	2(1+\sigma)U_0 U_1V_1 + ((1+\sigma)V_0 - 3U_0) U_1^2 + (1+\sigma)U_1^2V_1 - U_1^3 \\
	 U_0 V_1^2 + 2V_0 U_1V_1 + U_1V_1^2 
\end{pmatrix}
},
\end{align}
according to (\ref{jth}) for $j=3$ with $f_{\rm res} \equiv 0$.
Here the inhomogeneous term consists of terms of order $\tilde \theta(s)^{2\KMg{\rho}-1}$ and $\tilde \theta(s)^{3\KMg{\rho}-1}$ with $\KMg{\rho} = \sigma/2$.
\KMg{Because $\rho > 0$ under} our assumption, the terms of order $\tilde \theta(s)^{3\KMg{\rho}-1}$ is negligible compared with those of the order $\tilde \theta(s)^{2\KMg{\rho}-1}$, which indicates that the governing equation determining the third order term is essentially
\begin{align}
\label{3rd-essential-Andrews2}
\KMc{\frac{d}{ds}}
\begin{pmatrix}
U_2 \\ V_2
\end{pmatrix} &\approx 
\KMc{
\tilde \theta(s)^{-1} \left[ \left\{ -\frac{1}{2}I_2 +D(U_0, V_0; \sigma) \right\} \begin{pmatrix}
U_2 \\ V_2 
\end{pmatrix}
	+
	\begin{pmatrix}
	2(1+\sigma)U_0 U_1V_1 + ((1+\sigma)V_0 - 3U_0) U_1^2 \\
	U_0 V_1^2 + 2V_0 U_1V_1
\end{pmatrix} \right].
}
\end{align}
Integration of the system (\ref{3rd-Andrews2}) yields that the essential terms in \KMc{$(U_2(s), V_2(s))$} are of order \KMc{$\tilde \theta(s)^{2\KMg{\rho}}$} and hence these possess the following forms, respectively:
\begin{equation*}
\KMc{
C_{2u}\tilde \theta(s)^{2\KMg{\rho}},\quad C_{2v}\tilde \theta(s)^{2\KMg{\rho}}.
}
\end{equation*}
Substituting them into (\ref{3rd-Andrews2}) and comparing coefficients of $\tilde \theta(s)^{2\KMg{\rho}}$, equivalent to coefficient balance for (\ref{3rd-essential-Andrews2}), we have the following constraint:
\begin{align*}
-2\KMg{\rho}\begin{pmatrix}
C_{2u} \\ C_{2v}
\end{pmatrix}
 &= \left[ -\frac{1}{2}I_2 + 
\KMc{
 \begin{pmatrix}
 1 - \frac{\sigma}{2} & \frac{\sigma(1+\sigma)}{2} \\
 \frac{1}{2\sigma} & 1
\end{pmatrix} 
}
\right] \begin{pmatrix}
C_{2u} \\ C_{2v}
\end{pmatrix} + \begin{pmatrix}
\KMc{
\left\{  -\frac{2}{\sigma}U_0 +  (1+\sigma)V_0 - 3U_0)  \right\}C_{1u}^2 
}
	\\ 
\KMc{
	\left\{ U_0 - 2\sigma(1+\sigma)V_0 \right\}\frac{1}{\sigma^2(1+\sigma)^2}C_{1u}^2
}
\end{pmatrix},
\end{align*}
equivalently
\begin{align*}
\left[ \left(2\KMg{\rho}-\frac{1}{2}\right)I_2 + 
\KMc{
\begin{pmatrix}
 1 - \frac{\sigma}{2} & \frac{\sigma(1+\sigma)}{2} \\
 \frac{1}{2\sigma} & 1
\end{pmatrix} 
}
\right] \begin{pmatrix}
C_{2u} \\ C_{2v}
\end{pmatrix}
 &= -\begin{pmatrix}
\KMc{
\left\{  -\frac{2}{\sigma}U_0 +  (1+\sigma)V_0 - 3U_0)  \right\}C_{1u}^2 
}
	\\ 
\KMc{
	\left\{ U_0 - 2\sigma(1+\sigma)V_0 \right\}\frac{1}{\sigma^2(1+\sigma)^2}C_{1u}^2
}
\end{pmatrix}.
\end{align*}
The inverse of the coefficient matrix in the left hand side is
\begin{align*}
&\KMc{
\frac{1}{(2\KMg{\rho} + \frac{1}{2} - \frac{\sigma}{2})(2\KMg{\rho} + \frac{1}{2}) - \frac{1+\sigma}{4}}\begin{pmatrix}
2\KMg{\rho} + \frac{1}{2} & -\frac{\sigma(1+\sigma)}{2}\\
-\frac{1}{2\sigma} & 2\KMg{\rho} +\frac{1}{2} - \frac{\sigma}{2}
\end{pmatrix}
}\\
&= \KMc{
\frac{1}{\sigma(\sigma+1)}\begin{pmatrix}
2\sigma + 1 & -\sigma(1+\sigma)\\
-\frac{1}{\sigma} & 1+\sigma
\end{pmatrix}
}\KMg{.}\quad (\text{from }\rho = \sigma/2)
\end{align*}
Using \KMc{$U_0 = \sigma V_0$ and $U_0 = \sqrt{\sigma / 2}$},
we have
\begin{align*}
\begin{pmatrix}
C_{2u} \\ C_{2v}
\end{pmatrix}
 &= \KMc{ 
 -\frac{1}{\sigma(\sigma+1)}\begin{pmatrix}
2\sigma + 1 & -\sigma(1+\sigma)\\
-\frac{1}{\sigma} & 1+\sigma
\end{pmatrix}
\begin{pmatrix}
\left\{  -\frac{2}{\sigma}U_0 +  (1+\sigma)V_0 - 3U_0)  \right\}C_{1u}^2 \\
	\left\{ U_0 - 2\sigma(1+\sigma)V_0 \right\}\frac{1}{\sigma^2(1+\sigma)^2}C_{1u}^2
\end{pmatrix}
}\\
 &= \KMc{ 
 -\frac{1}{\sigma(\sigma+1)}\begin{pmatrix}
2\sigma + 1 & -\sigma(1+\sigma)\\
-\frac{1}{\sigma} & 1+\sigma
\end{pmatrix}
\begin{pmatrix}
\left\{  -\frac{2}{\sigma} +  (1+\sigma)\frac{1}{\sigma} - 3)  \right\} \\
	\left\{ 1 - 2(1+\sigma) \right\}\frac{1}{\sigma^2(1+\sigma)^2}
\end{pmatrix} U_0 C_{1u}^2
}
\\
 &= \KMc{ 
 \frac{1+2\sigma}{\sigma^2(\sigma+1)} \sqrt{\frac{\sigma}{2}} \begin{pmatrix}
2\sigma + 1 & -\sigma(1+\sigma)\\
-\frac{1}{\sigma} & 1+\sigma
\end{pmatrix}
\begin{pmatrix}
1 \\
\frac{1}{\sigma(1+\sigma)^2}
\end{pmatrix} C_{1u}^2
}
\\
 &= \KMc{ 
 \frac{1+2\sigma}{\sigma^2(\sigma+1)^2}\sqrt{\frac{\sigma}{2}}
\begin{pmatrix}
\sigma (2\sigma + 3) \\
-1
\end{pmatrix} C_{1u}^2
}.
\end{align*}

As a summary, \KMc{
the third order asymptotic expansion of the blow-up solution $(\tilde u(s), \tilde v(s))$ is calculated as follows:
\begin{align*}
\notag
\tilde u(s)&\sim \sqrt{\frac{\sigma}{2}} \tilde \theta(s)^{-1/2} + C_{1u} \tilde \theta(s)^{\frac{\sigma}{2}- \frac{1}{2}} +  \frac{1+2\sigma}{\sigma^2(\sigma+1)^2}\sqrt{\frac{\sigma}{2}} \sigma (2\sigma + 3) C_{1u}^2 \tilde \theta(s)^{\sigma -1/2}, \\
\tilde v(s)&\sim \frac{1}{\sqrt{2\sigma}} \tilde \theta(s)^{-1/2}  -\frac{1}{\sigma(1+\sigma)} C_{1u}\tilde \theta(s)^{\frac{\sigma}{2}- \frac{1}{2}} - \frac{1+2\sigma}{\sigma^2(\sigma+1)^2}\sqrt{\frac{\sigma}{2}} C_{1u}^2 \tilde \theta(s)^{\sigma -1/2}
\end{align*}
as $s\to s_{\max}$.
\KMd{
Now the constant $\delta$ defined in (\ref{suff-asym}) is estimated as follows.
First, $\alpha = (1,1)$ and $k=2$.
Because $f_{\rm res} \equiv 0$, we have $\gamma_1 = \gamma_2 = \KMg{+\infty}$ and hence
\begin{equation*}
\delta = \min \left\{ \min_{i=1,2} \left\{ \frac{\alpha_i}{k} + \gamma_i \right\} + 1, \frac{\sigma}{2} \right\} \KMg{= \frac{a}{b} - \frac{1}{2}.}
\end{equation*}
\KMg{Proposition} \ref{prop-order-increase} indicates that ${\rm ord}_\theta(Y_m) \geq m\left( \frac{a}{b} - \frac{1}{2} \right)$ for $m\geq 2$ and hence there is no term $\theta(t)^{m\sigma / 2}$ with $m\leq 2$ in $Y_n(t)$, $n\geq 3$.
}
Finally we obtain the third order asymptotic expansion of the blow-up solution in the original timescale $(u(t), v(t))$ through (\ref{trans-Andrews2}).
}
\begin{thm}
The system (\ref{Andrews2}) with $0 < b < 2a$ admits a blow-up solution with the following third order asymptotic expansion with $t\to t_{\max}$:
\begin{align}
\notag
u(t)&\sim \frac{\sqrt{2a-b}}{\sqrt{2}b}\theta(t)^{-1/2} + \KMg{C}\theta(t)^{\frac{a}{b}-1} +  \frac{b(4a-b)(4a+b)}{4\sqrt{2}a^2 \sqrt{2a - b}}\KMg{C^2} \theta(t)^{\frac{2a}{b} - \frac{3}{2}}, \\
\label{Andrews2-sol}
v(t)&\sim \frac{1}{\sqrt{2(2a-b)}}\theta(t)^{-1/2} -\frac{b^2}{2a(2a-b)}\KMg{C}\theta(t)^{\frac{a}{b}-1} - \frac{b^3(4a-b)}{4\sqrt{2}a^2 \left(2a - b\right)^{3/2}} \KMg{C^2} \theta(t)^{\frac{2a}{b} - \frac{3}{2}}
\end{align}
with $C = b^{(4a-3b)/2b}C_1$ and a constant $C_1\in \mathbb{R}$.
\end{thm}
Notice that the power of $t_{\max}- t$ in the second term is $\KMg{\rho} - \frac{1}{2}$ and that in the third term is $2\KMg{\rho} - \frac{1}{2}$, where $\KMg{\rho} = \frac{\sigma}{2} = \frac{a}{b} - \frac{1}{2} > 0$.

\begin{rem}
\label{rem-Andrews2-multi}
Similar to multi-order asymptotics of blow-up solutions for (\ref{Andrews1}), powers of $\theta(t)$ can be both positive and negative, depending of the quantity $\KMg{\rho} = \frac{a}{b} - \frac{1}{2}$, which itself is positive.
Observe that the second terms in (\ref{Andrews2-sol}) have the negative power of $\theta(t)$ if and only if
\begin{equation*}
\frac{a}{b} - 1 < 0\quad \Leftrightarrow \quad a < b < 2a.
\end{equation*}
Similarly, the third terms in (\ref{Andrews2-sol}) have the negative power of $\theta(t)$ if and only if
\begin{equation*}
\frac{2a}{b} - \frac{3}{2} < 0\quad \Leftrightarrow \quad a < \frac{3}{4}b < \frac{3}{2}a.
\end{equation*}
We know that the $n$-th terms of blow-up solutions have the negative power of $\theta(t)$ if and only if
\begin{equation*}
(n-1)\left(\frac{a}{b} - \frac{1}{2} \right)- \frac{1}{2} < 0\quad \Leftrightarrow \quad a < \frac{n}{2(n-1)}b < \frac{n}{n-1}a.
\end{equation*}
As the parameter $b (< 2a)$ approaches to $2a$, the number
\begin{equation*}
\sharp\left\{n \in \mathbb{N} \mid n\KMg{\rho} < \frac{1}{2}\right\}
\end{equation*}
determining blow-up solutions of the form
\begin{equation*}
u(t) = \theta(t)^{-1/2}\sum_{n=0}^\infty u_n\theta(t)^{\KMg{\rho}n},\quad v(t) = \theta(t)^{-1/2}\sum_{n=0}^\infty v_n\theta(t)^{\KMg{\rho}n}
\end{equation*}
for (\ref{Andrews2}) increases and the expansion would \lq\lq converge" to the blow-up solution with the blow-up rate $O(\theta^{-1/2}(\log \theta(t)^{-1})^{-1/2})$ \KMj{mentioned} in Remark \ref{fact-A2}.
This expectation imply that multi-order asymptotic expansion of blow-up solutions through the present methodology could contribute to investigate the parameter dependence of blow-up rates.
\end{rem}

\subsection{Keyfitz-Kranser-type system}
\label{section-ex-KK}

The next example is a $2$-dimensional system
\begin{equation}
\label{KK}
	u' = u^2 - v, \quad v' = \frac{1}{3} u^3 - u\KMg{,}
\end{equation}
originated from the Keyfitz-Kranser system \cite{KK1990} which is a system of conservation laws admitting a singular shock.
Blow-up solutions of (\ref{KK}) are studied in \cite{Mat2018} by means of dynamics at infinity\footnote{
\KMj{Detailed treatments are also reviewed in \cite{asym2}.}}.
Observe that the system is asymptotically quasi-homogeneous of type $\alpha=(\alpha_1, \alpha_2)=(1,2)$ and order $k+1 = 2$, consisting of the quasi-homogeneous part $f_{\alpha, k}$ and the lower-order part $f_{\rm res}$ given as follows: 
\begin{equation}
	f_{\alpha, k}(u,v) = \begin{pmatrix}
 u^2 - v\\
 \frac{1}{3} u^3	
 \end{pmatrix}, \quad f_{\mathrm{res}}(u,v) = \begin{pmatrix}
 0\\
 -u	
 \end{pmatrix}.
\end{equation}
It is proved in \cite{Mat2018} that the system (\ref{KK}) admits the following solutions blowing up as $t \to t_{\max} - 0$ associated with {\em two} different equilibria \KMj{at infinity}:
\begin{equation}
	u(t) = O(\theta(t)^{-1}), \quad v(t) = O (\theta(t)^{-2}), \quad \text{as} \quad t \to t_{\max} - 0.
\end{equation}

\subsubsection{Blow-up asymptotics for the quasi-homogeneous part}

Before treating the system (\ref{KK}), we consider a simpler system
\begin{equation}
\label{KK0}
	u' = u^2 - v, \quad v' = \frac{1}{3} u^3,
\end{equation}
namely the quasi-homogeneous part of (\ref{KK}).
As in previous examples, expand the solution $(u(t),v(t))$ as the asymptotic series
\begin{align}
\notag
u(t) &= \theta(t)^{-1}U(t) \equiv \theta(t)^{-1}\sum_{n=0}^\infty U_n(t),\quad U_n(t) \ll U_{n-1}(t),\quad \lim_{t\to t_{\max}}U_n(t) = U_0, \\
\label{series-KK}
v(t) &= \theta(t)^{-2}V(t) \equiv \theta(t)^{-2}\sum_{n=0}^\infty V_n(t),\quad V_n(t) \ll V_{n-1}(t), \quad \lim_{t\to t_{\max}}V_n(t) = V_0.
\end{align}
Note that the quasi-homogeneity induces the different rate of blow-ups among components.
The balance law for (\ref{KK0}), and hence for (\ref{KK}) is
\begin{equation*}
	U_0 = U_0^2 - V_0,\quad 2V_0 = \frac{1}{3} U_0^3,
\end{equation*}
which yields
\begin{equation}
\label{balance-KK}
	U_0 = 3 \pm \sqrt{3}, \quad V_0 = \frac{1}{6} U_0^3.
\end{equation}
In particular, we have two different solutions of the balance law\KMj{.}
\par
Next we consider the second term $(U_1(t), V_1(t))$.
Substituting (\ref{series-KK}) into (\ref{KK0}) and using the balance law, we have
\begin{align}
\notag
\frac{d}{dt}\left(\sum_{n=1}^\infty U_n \right) &= \theta(t)^{-1}\left\{ -\left( \sum_{n=1}^\infty U_n \right) + 2U_0\left( \sum_{n=1}^\infty U_n \right) + \left( \sum_{n=1}^\infty U_n \right)^2 - \left( \sum_{n=1}^\infty V_n \right)\right\},\\
\label{asymptotic-system-KK0}
\frac{d}{dt}\left(\sum_{n=1}^\infty V_n \right) &= \theta(t)^{-1}\left\{ - 2\left( \sum_{n=1}^\infty V_n \right) + U_0^2\left( \sum_{n=1}^\infty U_n \right) + U_0\left( \sum_{n=1}^\infty U_n \right)^2 + \frac{1}{3}\left( \sum_{n=1}^\infty U_n \right)^3 \right\}.
\end{align}
The governing system for $(U_1(t), V_1(t))$ is
\begin{align}
\label{2nd-KK0}
\frac{d}{dt}\begin{pmatrix}
U_1 \\
V_1
\end{pmatrix} &= \theta(t)^{-1}\left\{ \begin{pmatrix}
-1 & 0 \\
0 & -2
\end{pmatrix} + \begin{pmatrix}
2U_0 & -1 \\
U_0^2 & 0
\end{pmatrix} \right\}\begin{pmatrix}
U_1 \\
V_1
\end{pmatrix}.
\end{align}
Depending on the choice of $U_0$, the eigenstructure of the blow-up power-determining matrix \KMg{changes}.
\begin{description}
\item[Case 1. $U_0 = 3-\sqrt{3}$.] 
\end{description}
In this case, the blow-up power-determining matrix $A$ is
\begin{equation*}
A = \begin{pmatrix}
2(3-\sqrt{3})-1 & -1 \\
(3-\sqrt{3})^2 & -2
\end{pmatrix} = \begin{pmatrix}
5 - 2\sqrt{3} & -1 \\
12 - 6\sqrt{3} & -2
\end{pmatrix}
\end{equation*}
and the eigenvalues solve
\begin{align*}
(\lambda -5+2\sqrt{3})(\lambda +2) + (12-6\sqrt{3}) 
	&= \lambda^2 + (-3+2\sqrt{3})\lambda  + 2 -2\sqrt{3} = 0.
\end{align*}
The eigenvalues are
\begin{equation*}
\lambda = 1, \quad 2-2\sqrt{3}
\end{equation*}
and associating eigenvectors are
\begin{align*}
\begin{pmatrix}
5 - 2\sqrt{3} & -1 \\
12 - 6\sqrt{3} & -2
\end{pmatrix}\begin{pmatrix}
a \\ b
\end{pmatrix} = \begin{pmatrix}
a \\ b
\end{pmatrix}\quad &\Rightarrow \quad \begin{pmatrix}
a \\ b
\end{pmatrix} = \begin{pmatrix}
1 \\ 4-2\sqrt{3}
\end{pmatrix},\\
\begin{pmatrix}
5 - 2\sqrt{3} & -1 \\
12 - 6\sqrt{3} & -2
\end{pmatrix}\begin{pmatrix}
a \\ b
\end{pmatrix} = (2-2\sqrt{3})\begin{pmatrix}
a \\ b
\end{pmatrix}\quad &\Rightarrow \quad \begin{pmatrix}
a \\ b
\end{pmatrix} = \begin{pmatrix}
1 \\ 3
\end{pmatrix}.
\end{align*}
Because $2-2\sqrt{3} < 0$, this eigenvalue contributes the power of $\theta(t)$ in the asymptotic expansion, while $\lambda = 1$ makes no contribution \KMk{to asymptotic expansions}.
Introducing the nonsingular matrix 
\begin{equation}
\label{P-KK0}
P = \begin{pmatrix}
1 & 1 \\
4-2\sqrt{3} & 3
\end{pmatrix}\quad \Leftrightarrow \quad P^{-1} \equiv \begin{pmatrix}
p^{11} & p^{12} \\
p^{21} & p^{22}
\end{pmatrix} 
= \frac{1}{-1+2\sqrt{3}}\begin{pmatrix}
3 & -1 \\
-4+2\sqrt{3} & 1
\end{pmatrix},
\end{equation}
the second term $(U_1(t), V_1(t))$ is calculated as
\begin{equation*}
\begin{pmatrix}
U_1(t) \\ V_1(t)
\end{pmatrix} = P\begin{pmatrix}
0 \\ C\theta(t)^{-2+2\sqrt{3}}
\end{pmatrix} = \begin{pmatrix}
C\theta(t)^{-2+2\sqrt{3}} \\ 3C\theta(t)^{-2+2\sqrt{3}}
\end{pmatrix} 
\end{equation*}
as $t\to t_{\max}$.
Letting $C_{1u} \equiv C$, we have the second order asymptotic expansion of blow-up solutions $(u(t), v(t))$ with $U_0 = 3-\sqrt{3}$ as follows:
\begin{align*}
u(t) &\sim (3 - \sqrt{3})\theta(t)^{-1} + C_{1u}\theta(t)^{-3+2\sqrt{3}},\\
v(t) &\sim (9 - 5\sqrt{3})\theta(t)^{-2} + 3C_{1u}\theta(t)^{-4+2\sqrt{3}}
\end{align*}
as $t\to t_{\max}$.


Next the third term $(U_2(t), V_2(t))$ is considered with $U_0 = 3-\sqrt{3}$.
From (\ref{asymptotic-system-KK0}) and (\ref{2nd-KK0}), the governing equation for $(U_2(t), V_2(t))$ is
\begin{align}
\label{3rd-KK0}
\frac{d}{dt}\begin{pmatrix}
U_2 \\ V_2 
\end{pmatrix} = \theta(t)^{-1}\left\{ A\begin{pmatrix}
U_2 \\ V_2 
\end{pmatrix} + 
\begin{pmatrix}
U_1^2 \\ U_0 U_1^2 + \frac{1}{3} U_1^3
\end{pmatrix} \right\}.
\end{align}
From the asymptotic assumption $U_1(t) \ll U_0$, the essential term of $(U_2(t), V_2(t))$ is approximately governed by
\begin{align*}
\frac{d}{dt}\begin{pmatrix}
U_2 \\ V_2 
\end{pmatrix} \approx \theta(t)^{-1} \left\{ A\begin{pmatrix}
U_2 \\ V_2 
\end{pmatrix} + 
\begin{pmatrix}
U_1^2 \\ U_0 U_1^2
\end{pmatrix} \right\}.
\end{align*}
The general solution of this approximate system is
\begin{align*}
\begin{pmatrix}
U_2(t) \\ V_2(t) 
\end{pmatrix} \approx P
\begin{pmatrix}
\theta(t)^{-1} \left\{ \KMd{ -\int_t^{t_{\max}} } \theta(s)^{1-1}(p^{11}U_1(s)^2 + p^{12}U_0U_1(s)^2 ) ds\right\}\\
\theta(t)^{-2+2\sqrt{3}} \left\{ \KMd{ \tilde c_2 + \int_0^t } \theta(s)^{2-2\sqrt{3}-1}(p^{21}U_1(s)^2 + p^{22}U_0U_1(s)^2 ) ds\right\}\\
\end{pmatrix}
\end{align*}
Letting $r = 2\sqrt{3} - 2 (>0)$, we calculate the integrals.
First, 
\begin{align*}
&\KMd{ -\int_t^{t_{\max}} } (p^{11}U_1(s)^2 + p^{12}U_0U_1(s)^2 ) ds\\
&\quad = \KMd{ - p^{11}C_{1u}^2 \int_t^{t_{\max}} } \theta(s)^{2r} ds \KMd{ - U_0p^{12}C_{1u}^2\int_t^{t_{\max}}} \theta(s)^{2r} ds\\
&\quad = \frac{-1}{2r+1}\left\{ p^{11} + U_0p^{12} \right\} C_{1u}^2\theta(t)^{2r+1} \\
%
%
%
%
%
&\quad = \frac{-(9\sqrt{3}+10)}{ 143} C_{1u}^2\theta(t)^{2r+1}.
\end{align*}
Similarly, 
\begin{align*}
&\KMd{\int_0^t} \theta(s)^{-r-1} (p^{21}U_1(s)^2 + p^{22}U_0U_1(s)^2 ) ds\\
&\quad = p^{21}C_{1u}^2 \KMd{\int_0^t} \theta(s)^{r-1} ds + U_0p^{22}C_{1u}^2 \KMd{\int_0^t} \theta(s)^{r-1} ds\\
&\quad = \frac{-1}{r}\left\{ p^{21} + U_0p^{22} \right\} C_{1u}^2\theta(t)^{r} + \tilde C\\
&\quad = -\frac{1+2\sqrt{3}}{22} C_{1u}^2\theta(t)^{r} + \tilde C,
\end{align*}
\KMd{where the integral constant $\tilde C$ is chosen so that $\tilde c_2 + \tilde C = 0$ (cf. (\ref{const-asym-j})).}
Therefore
\begin{align*}
\begin{pmatrix}
U_2(t) \\ V_2(t) 
\end{pmatrix} &\approx P
\begin{pmatrix}
\frac{-(9\sqrt{3}+10)}{ 143} C_{1u}^2\theta(t)^{2r}\\
-\frac{1+2\sqrt{3}}{22} C_{1u}^2\theta(t)^{2r}
\end{pmatrix}
=
\begin{pmatrix}
 -\frac{3+4\sqrt{3}}{ 26}  C_{1u}^2\theta(t)^{2r} \\
 -\frac{1+10\sqrt{3}}{ 26}  C_{1u}^2\theta(t)^{2r}
\end{pmatrix}.
\end{align*}
As a consequence, 
we have the third order asymptotic expansion of blow-up solutions $(u(t), v(t))$ with $U_0 = 3-\sqrt{3}$ as follows:
\begin{align*}
u(t) &\sim (3 - \sqrt{3})\theta(t)^{-1} + C_{1u}\theta(t)^{-3+2\sqrt{3}} -\frac{3+4\sqrt{3}}{26}  C_{1u}^2\theta(t)^{-5+{4}\sqrt{3}},\\
v(t) &\sim (9 - 5\sqrt{3})\theta(t)^{-2} + 3C_{1u}\theta(t)^{-4+2\sqrt{3}} -\frac{1+10\sqrt{3}}{26}  C_{1u}^2\theta(t)^{-6+{4}\sqrt{3}}
\end{align*}
as $t\to t_{\max}$.

\begin{description}
\item[Case 2. $U_0 = 3+\sqrt{3}$.] 
\end{description}
In this case, the blow-up power-determining matrix $A$ is
\begin{equation*}
A = \begin{pmatrix}
2(3+\sqrt{3})-1 & -1 \\
(3+\sqrt{3})^2 & -2
\end{pmatrix} = \begin{pmatrix}
5 + 2\sqrt{3} & -1 \\
12 + 6\sqrt{3} & -2
\end{pmatrix}
\end{equation*}
and the eigenvalues solve
\begin{align*}
(\lambda -5-2\sqrt{3})(\lambda +2) + (12+6\sqrt{3}) 
	&= \lambda^2 + (-3-2\sqrt{3})\lambda -2(5+2\sqrt{3}) + 12 + 6\sqrt{3}\\
	&= \lambda^2 + (-3-2\sqrt{3})\lambda  + 2 +2\sqrt{3} = 0.
\end{align*}
The eigenvalues are
\begin{equation*}
\lambda = 1, \quad 2+2\sqrt{3},
\end{equation*}
both of which are positive.
Therefore these eigenvalues make {\em no} contributions to multi-order asymptotic expansions.
In particular, we know that $U_n(t)\equiv 0$, $V_n(t) \equiv 0$ for all $n\geq 1$ and hence the first order asymptotic expansion
\begin{align*}
u(t) &\sim (3 + \sqrt{3})\theta(t)^{-1},\quad 
v(t) \sim (9 + 5\sqrt{3})\theta(t)^{-2}
\end{align*}
indeed expresses the {\em exact solution} of (\ref{KK0}).
As a summary, we obtain the following result.

\begin{thm}
The quasi-homogeneous system (\ref{KK0}) admits blow-up solutions with the following third-order asymptotic expansions as $t\to t_{\max}$:
\begin{align*}
u(t) &\sim (3 - \sqrt{3})\theta(t)^{-1} + C\theta(t)^{-3+2\sqrt{3}} -\frac{3+4\sqrt{3}}{26}  C^2\theta(t)^{-5+{4}\sqrt{3}},\\
v(t) &\sim (9 - 5\sqrt{3})\theta(t)^{-2} + 3C\theta(t)^{-4+2\sqrt{3}} -\frac{1+10\sqrt{3}}{26}  C^2\theta(t)^{-6+{4}\sqrt{3}}
\end{align*}
with a free parameter $C$, and
\begin{align*}
u(t) &\sim (3 + \sqrt{3})\theta(t)^{-1},\quad v(t) \sim (9 + 5\sqrt{3})\theta(t)^{-2}.
\end{align*}
The latter solution is indeed the exact solution of (\ref{KK0}).
\end{thm}

\subsubsection{Blow-up asymptotics for (\ref{KK})}

We go back to the original system (\ref{KK}).
Notice first that the balance law for (\ref{KK}) is identical with that of (\ref{KK0}) \KMg{because} the quasi-homogeneous components are identical.
Substituting the asymptotic form (\ref{series-KK}) into (\ref{KK}), we have
\begin{align}
\notag
\frac{d}{dt}\left(\sum_{n=1}^\infty U_n \right) &= \theta(t)^{-1}\left\{ -\left( \sum_{n=1}^\infty U_n \right) + 2U_0\left( \sum_{n=1}^\infty U_n \right) + \left( \sum_{n=1}^\infty U_n \right)^2 - \left( \sum_{n=1}^\infty V_n \right)\right\},\\
\label{asymptotic-system-KK}
\frac{d}{dt}\left(\sum_{n=1}^\infty V_n \right) &= \theta(t)^{-1}\left\{ - 2\left( \sum_{n=1}^\infty V_n \right) + U_0^2\left( \sum_{n=1}^\infty U_n \right) + U_0\left( \sum_{n=1}^\infty U_n \right)^2 + \frac{1}{3}\left( \sum_{n=1}^\infty U_n \right)^3 \right\} \\
\notag
&\quad - \theta(t)\left(\sum_{n=0}^\infty U_n \right),
\end{align}
where the balance law (\ref{balance-KK}) is applied to eliminating the principal terms.
The governing system for $(U_1(t), V_1(t))$ is
\begin{align}
\label{2nd-KK}
\frac{d}{dt}\begin{pmatrix}
U_1 \\
V_1
\end{pmatrix} &= \theta(t)^{-1}\left\{ \begin{pmatrix}
-1 & 0 \\
0 & -2
\end{pmatrix} + \begin{pmatrix}
2U_0 & -1 \\
U_0^2 & 0
\end{pmatrix} \right\}\begin{pmatrix}
U_1 \\
V_1
\end{pmatrix} - \theta(t)\begin{pmatrix}
0 \\
U_0
\end{pmatrix},
\end{align}
where the lower order term $-\theta(t)(0,U_0)^T$ is added to (\ref{2nd-KK0}).
\KMd{Because} the linear (homogeneous) part is determined by the blow-up power-determining matrix $A$, the nonsingular matrix $P$ diagonalizing $A$ can be applied to (\ref{2nd-KK}). 
The concrete \KMg{forms} of solutions depend on the choice of $U_0$, like (\ref{KK0}).

\begin{description}
\item[Case 1. $U_0 = 3-\sqrt{3}$.] 
\end{description}
In this case, the matrix $P$ is given in (\ref{P-KK0})\KMg{,} and $\KMg{\rho} = 2\sqrt{3}-2 (>0)$ \KMg{is used to determine the order of $\theta(t)$}.
\KMd{ 
The general solution \KMg{of (\ref{2nd-KK})} is then
\begin{align*}
\begin{pmatrix}
U_1(t) \\ V_1(t) 
\end{pmatrix} = P
\begin{pmatrix}
\theta(t)^{-1} \left\{ \KMd{-\int_t^{t_{\max}}} \theta(s)^{1+1}(p^{11}0 + p^{12}(-U_0) ) ds\right\}\\
\theta(t)^r \left\{ \tilde c_2 + \KMd{ \int_0^t } \theta(s)^{-\KMg{\rho}+1}(p^{21}0 + p^{22}(-U_0) ) ds\right\}
\end{pmatrix}.
\end{align*}
}
The integrals are calculated to obtain
\begin{align*}
&\KMd{ -\int_t^{t_{\max}} } \theta(s)^{1+1}(p^{11}0 + p^{12}(-U_0) ) ds =  \KMd{ -p^{12}(-U_0)\int_t^{t_{\max}} }\theta(s)^2 ds\\
&\quad = -\frac{3-\sqrt{3}}{3(-1+2\sqrt{3})}\theta(t)^3\\
&\quad = -\frac{-3+5\sqrt{3}}{33}\theta(t)^3,
\end{align*}
and
\begin{align*}
&\KMd{ \int_0^t } \theta(s)^{-\KMg{\rho}+1}(p^{21}0 + p^{22}(-U_0) ) ds =  p^{22}(-U_0) \KMd{ \int_0^t } \theta(s)^{-\KMg{\rho}+1} ds\\
&\quad = \frac{3-\sqrt{3}}{(-1+2\sqrt{3})(-\KMg{\rho}+2)}\theta(t)^{-\KMg{\rho}+2} + C\\
&\quad = \frac{3-\sqrt{3}}{(-1+2\sqrt{3})(4-2\sqrt{3})}\theta(t)^{-\KMg{\rho}+2} + C\\
&\quad = \frac{9 + 7\sqrt{3}}{22}\theta(t)^{-r+2} + C
\end{align*}
\KMd{with the integral constant $C$.}
Therefore
\begin{align*}
\begin{pmatrix}
U_1(t) \\ V_1(t) 
\end{pmatrix} = P
\begin{pmatrix}
\theta(t)^{-1} \left\{ - \frac{-3+5\sqrt{3}}{33}\theta(t)^3 \right\}\\
\theta(t)^\KMg{\rho} \left\{ \KMg{C_{1u}} + \frac{9 + 7\sqrt{3}}{22}\theta(t)^{-\KMg{\rho}+2} \right\}
\end{pmatrix}
\end{align*}
\KMd{with a constant $\KMg{C_{1u}}$. }
In particular, 
\begin{align}
\label{sol-2nd-KK-sink}
\begin{pmatrix}
U_1(t) \\ V_1(t) 
\end{pmatrix} &=  \begin{pmatrix}
1 & 1 \\
4-2\sqrt{3} & 3
\end{pmatrix}
\begin{pmatrix}
\frac{3-5\sqrt{3}}{33}\theta(t)^2 \\
\KMg{C_{1u}}\theta(t)^\KMg{\rho} + \frac{9 + 7\sqrt{3}}{22}\theta(t)^2
\end{pmatrix}\\
\notag
&=  
\begin{pmatrix}
\KMg{C_{1u}} \theta(t)^\KMg{\rho} + (\frac{3-5\sqrt{3}}{33} +\frac{9 + 7\sqrt{3}}{22}) \theta(t)^2\\
3\KMg{C_{1u}} \theta(t)^\KMg{\rho} +  (\frac{(3-5\sqrt{3})(4-2\sqrt{3})}{33} +\frac{3(9 + 7\sqrt{3})}{22}) \theta(t)^2
\end{pmatrix}\\
\notag
&=  
\begin{pmatrix}
\KMg{C_{1u}} \theta(t)^\KMg{\rho} + (\frac{2(3-5\sqrt{3})}{66} +\frac{3(9 + 7\sqrt{3})}{66}) \theta(t)^2\\
3\KMg{C_{1u}} \theta(t)^\KMg{\rho} +  (\frac{2(42 - 26\sqrt{3})}{66} +\frac{9(9 + 7\sqrt{3})}{66}) \theta(t)^2
\end{pmatrix}\\
\notag
&=  
\begin{pmatrix}
\KMg{C_{1u}} \theta(t)^\KMg{\rho} + (\frac{(6-10\sqrt{3})}{66} +\frac{(27 + 21\sqrt{3})}{66}) \theta(t)^2\\
3\KMg{C_{1u}} \theta(t)^\KMg{\rho} +  (\frac{(84 - 52\sqrt{3})}{66} +\frac{(81 + 63\sqrt{3})}{66}) \theta(t)^2
\end{pmatrix}\\
\notag
&=  
\begin{pmatrix}
\KMg{C_{1u}} \theta(t)^\KMg{\rho} + \frac{3+\sqrt{3}}{6} \theta(t)^2\\
3\KMg{C_{1u}} \theta(t)^\KMg{\rho} +  \frac{15 + \sqrt{3}}{6} \theta(t)^2
\end{pmatrix}.
\end{align}
As a consequence, 
we have the second order asymptotic expansion of blow-up solutions $(u(t), v(t))$ with $U_0 = 3-\sqrt{3}$ as follows:
\begin{align}
\notag
u(t) &\sim (3 - \sqrt{3})\theta(t)^{-1} + C_{1u}\theta(t)^{-3+2\sqrt{3}} + \frac{3+\sqrt{3}}{6}\theta(t),\\
\label{2nd-asym-KK}
v(t) &\sim (9 - 5\sqrt{3})\theta(t)^{-2} + 3C_{1u}\theta(t)^{-4+2\sqrt{3}} + \frac{15 + \sqrt{3}}{6}
\end{align}
as $t\to t_{\max}$ with a constant $C_{1u}$.
In the present case, the additional terms of $O(\theta(t))$ in $u$ and of $O(1)$ in $v$, respectively, are added as the contribution of the lower-order term $f_{\rm res}$.
We observe that $\KMg{\rho}$ and $2 = k+\alpha_2 - 1$ are exponents of $\theta(t)$ determining $(U(t),V(t))$.
The exponent $\KMg{\rho}$ comes from blow-up power eigenvalues (\KMg{recall that $-\rho$ is one of such eigenvalues}), while $2$ comes from $f_{\rm res}$.
\KMg{Because} $\KMg{\rho}$ and $2$ are $\mathbb{Z}$-linearly independent, coefficients of all terms in (\ref{2nd-asym-KK}) are uniquely determined.
In fact, when we consider the governing system for the third term $(U_2(t), V_2(t))$ given by
\begin{align}
\label{3rd-KK}
\frac{d}{dt}\begin{pmatrix}
U_2 \\ V_2 
\end{pmatrix} \approx \theta(t)^{-1} \left\{ A\begin{pmatrix}
U_2 \\ V_2 
\end{pmatrix} + 
\begin{pmatrix}
U_1^2 \\ U_0 U_1^2
\end{pmatrix} \right\} - \theta(t)\begin{pmatrix}
0 \\
U_1
\end{pmatrix},
\end{align}
we see that the solution has the form
\begin{align*}
&P^{-1}\begin{pmatrix}
U_2(t) \\ V_2(t) 
\end{pmatrix}\\
&= \begin{pmatrix}
\theta(t)^{-1} [ -\int_t^{t_{\max}} \theta(s)^{1-1} (c_{11} \theta(s)^{2\KMg{\rho}} + c_{12} \theta(s)^{2+\KMg{\rho}} + c_{13} \theta(s)^{4}) ds -\int_t^{t_{\max}} \theta(s)^{1+1} (c_{14} \theta(s)^{\KMg{\rho}} + c_{15} \theta(s)^{2}) ds]\\
\theta(t)^{\KMg{\rho}} [ \KMd{\tilde c_2 + \int_0^t} \theta(s)^{-\KMg{\rho}-1} (c_{21} \theta(s)^{2\KMg{\rho}} + c_{22} \theta(s)^{2+\KMg{\rho}} + c_{23} \theta(s)^{4}) ds + \int_0^t \theta(s)^{-\KMg{\rho}+1} (c_{24} \theta(s)^{\KMg{\rho}} + c_{25} \theta(s)^{2}) ds]\\
\end{pmatrix}\\
&= \begin{pmatrix}
\tilde c_{11} \theta(t)^{2\KMg{\rho}} + \tilde c_{12} \theta(t)^{2+\KMg{\rho}} + \tilde c_{13} \theta(t)^{4} + \tilde c_{14} \theta(t)^{2+\KMg{\rho}} + \tilde c_{15} \theta(t)^{4} \\
\tilde c_{21} \theta(t)^{2\KMg{\rho}} + \tilde c_{22} \theta(t)^{2+\KMg{\rho}} + \tilde c_{23} \theta(t)^{4} + \tilde c_{24} \theta(t)^{2+\KMg{\rho}} + \tilde c_{25} \theta(t)^{4}\\
\end{pmatrix},
\end{align*}
where the constant $\tilde c_2$ is chosen so that the constant term vanishes in the second component (cf. (\ref{const-asym-j})).
All possible exponents appeared in the above solution does not contain $2$ and $\KMg{\rho}$. 
In fact, all components appeared here has the form $2\beta_1 + \KMg{\rho}\beta_2$, where $\beta_1, \beta_2\in \mathbb{Z}_{\geq 0}$ satisfying $\beta_1 + \beta_2 = 2$.
By induction, we can prove that all possible exponents appeared in ${\bf Y}_n(t) \equiv (U_n(t), V_n(t))$ as the form $2\beta_1 + \KMg{\rho} \beta_2$, where $\beta_1, \beta_2\in \mathbb{Z}_{\geq 0}$ satisfying $\beta_1 + \beta_2 = n$.
For example, the governing system determining ${\bf Y}_3(t) \equiv (U_3(t), V_3(t))$ is, combining the neglected terms in the derivation of ${\bf Y}_2(t) = (U_2(t), V_2(t))$,
\begin{align}
\label{4th-KK}
\frac{d}{dt}\begin{pmatrix}
U_3 \\ V_3 
\end{pmatrix} \approx \theta(t)^{-1} \left\{ A\begin{pmatrix}
U_3 \\ V_3 
\end{pmatrix} + 
\begin{pmatrix}
2U_1U_2 + U_2^2 \\ U_1^3 + 2U_0U_1U_2 + U_0U_2^2 + \frac{1}{3}U_1^3 + U_1^2U_2 + U_1 U_2^2
\end{pmatrix} \right\} - \theta(t)\begin{pmatrix}
0 \\
U_2
\end{pmatrix},
\end{align}
and all integrands appeared in solving ${\bf Y}_3(t)$ consist of $\theta(t)^\mu$ with $\mu = 2\beta_1 + \KMg{\rho} \beta_2-1$, where $\beta_1, \beta_2 \in \mathbb{Z}_{\geq 0}$ with $\beta_1 + \beta_2 = 3$.
Therefore the integration of (\ref{4th-KK}) yields our claim with $n=3$.
The general case can also be treated in the similar way.
By the $\mathbb{Z}$-linear independence of $2$ and $\KMg{\rho}$, all terms appeared in $(U_n(t), V_n(t))$ are uniquely determined for each $n \geq 0$.
In particular, (\ref{2nd-asym-KK}) is the second order exact asymptotic expansion of $(u(t), v(t))$ as $t\to t_{\max}$.

\begin{description}
\item[Case 2. $U_0 = 3+\sqrt{3}$.] 
\end{description}
In the similar way to the case $U_0 = 3- \sqrt{3}$, the asymptotic \KMg{expansion} of the blow-up solution with $U_0 = 3+\sqrt{3}$ can be \KMg{calculated}.
Unlike the fully quasi-homogeneous case, the lower-order part $f_{\rm res}$ makes a contribution to determine the (nontrivial) asymptotic behavior.
\par
To this end, calculate the nonsingular matrix $P_+$ diagonalizing $A$ with $U_0 = 3+\sqrt{3}$, which is constructed through calculations of eigenvectors:
\begin{align*}
\begin{pmatrix}
5 + 2\sqrt{3} & -1 \\
12 + 6\sqrt{3} & -2
\end{pmatrix}\begin{pmatrix}
a \\ b
\end{pmatrix} = \begin{pmatrix}
a \\ b
\end{pmatrix}\quad &\Rightarrow \quad \begin{pmatrix}
a \\ b
\end{pmatrix} = \begin{pmatrix}
1 \\ 4+2\sqrt{3}
\end{pmatrix},\\
\begin{pmatrix}
5 + 2\sqrt{3} & -1 \\
12 + 6\sqrt{3} & -2
\end{pmatrix}\begin{pmatrix}
a \\ b
\end{pmatrix} = (2+2\sqrt{3})\begin{pmatrix}
a \\ b
\end{pmatrix}\quad &\Rightarrow \quad \begin{pmatrix}
a \\ b
\end{pmatrix} = \begin{pmatrix}
1 \\ 3
\end{pmatrix}.
\end{align*}
Let $P_+$ be then the following matrix:
\begin{equation*}
P_+ = \begin{pmatrix}
1 & 1 \\ 
4+2\sqrt{3} & 3
\end{pmatrix}\quad \Leftrightarrow \quad 
P_+^{-1} \equiv \begin{pmatrix}
p_+^{11} & p_+^{12} \\ 
p_+^{21} & p_+^{22}
\end{pmatrix}
= \frac{-1}{1+2\sqrt{3}}\begin{pmatrix}
3 & -1 \\ 
-(4+2\sqrt{3}) & 1
\end{pmatrix}.
\end{equation*}
The general solution \KMg{of (\ref{2nd-KK})} is written by
\begin{align*}
\begin{pmatrix}
U_1(t) \\ V_1(t) 
\end{pmatrix} = P_+
\begin{pmatrix}
\theta(t)^{-1} \left\{ -\int_t^{t_{\max}} \theta(s)^{1+1}(p_+^{11}0 + p_+^{12}(-U_0) ) ds\right\}\\
\theta(t)^{-\KMg{\rho}_+} \left\{ -\int_t^{t_{\max}} \theta(s)^{\KMg{\rho}_+ + 1}(p_+^{21}0 + p_+^{22}(-U_0) ) ds\right\}
\end{pmatrix},
\end{align*}
where $\KMg{\rho}_+ = 2+2\sqrt{3}$.
The integrals are calculated to obtain
\begin{align*}
&\KMd{-\int_t^{t_{\max}}} \theta(s)^{1+1}(p_+^{11}0 + p_+^{12}(-U_0) ) ds =  \KMd{-} p_+^{12}(-U_0)\KMd{\int_t^{t_{\max}}} \theta(s)^2 ds\\
&\quad = \frac{(3+\sqrt{3})}{3(1+2\sqrt{3})}\theta(t)^3\\
&\quad = \frac{3+5\sqrt{3}}{33}\theta(t)^3
\end{align*}
and
\begin{align*}
&-\int_t^{t_{\max}} \theta(s)^{\KMg{\rho}_+ +1}(p_+^{21}0 + p_+^{22}(-U_0) ) ds = - p_+^{22}(-U_0) \KMd{\int_t^{t_{\max}}} \theta(s)^{\KMg{\rho}_+ +1} ds\\
&\quad = -\frac{3+\sqrt{3}}{(1+2\sqrt{3})(\KMg{\rho}_+ +2)}\theta(t)^{\KMg{\rho}_+ +2} \\
&\quad = -\frac{3+\sqrt{3}}{(1+2\sqrt{3})(4+2\sqrt{3})}\theta(t)^{\KMg{\rho}_+ +2} \\
&\quad = \frac{9-7\sqrt{3}}{22}\theta(t)^{\KMg{\rho}_+ +2}.
\end{align*}
Therefore
\begin{align*}
\begin{pmatrix}
U_1(t) \\ V_1(t) 
\end{pmatrix} &= P_+
\begin{pmatrix}
\theta(t)^{-1} \left\{ \frac{3+5\sqrt{3}}{33}\theta(t)^3 \right\}\\
\theta(t)^{-\KMg{\rho}_+} \left\{\frac{9 - 7\sqrt{3}}{22}\theta(t)^{\KMg{\rho}_+ +2} \right\}
\end{pmatrix}\\
 &=  \begin{pmatrix}
1 & 1 \\
4+2\sqrt{3} & 3
\end{pmatrix}
\begin{pmatrix}
\frac{3+5\sqrt{3}}{33}\theta(t)^2 \\
 \frac{9 - 7\sqrt{3}}{22}\theta(t)^2
\end{pmatrix}\\
\notag
&=  
\begin{pmatrix}
(\frac{3+5\sqrt{3}}{33} +\frac{9 - 7\sqrt{3}}{22}) \theta(t)^2\\
 (- \frac{(3+5\sqrt{3})(4+2\sqrt{3})}{33} +\frac{3(9 - 7\sqrt{3})}{22}) \theta(t)^2
\end{pmatrix}\\
\notag
&=  
\begin{pmatrix}
\frac{3-\sqrt{3}}{6} \theta(t)^2\\
 \frac{15 - \sqrt{3}}{6} \theta(t)^2
\end{pmatrix}.
\end{align*}
As a consequence, 
we have the second order asymptotic expansion of blow-up solutions $(u(t), v(t))$ with $U_0 = 3+\sqrt{3}$ as follows:
\begin{align*}
u(t) &\sim (3 + \sqrt{3})\theta(t)^{-1} + \frac{3 - \sqrt{3}}{6}\theta(t),\quad
v(t) \sim (9 + 5\sqrt{3})\theta(t)^{-2} + \frac{15 - \sqrt{3}}{6}
\end{align*}
as $t\to t_{\max}$.

\begin{thm}
The system (\ref{KK}) admits the following two families of blow-up solutions as $t\to t_{\max}$:
\begin{align}
\notag
u(t) &\sim (3 - \sqrt{3})\theta(t)^{-1} + C\theta(t)^{-3+2\sqrt{3}} + \frac{3+\sqrt{3}}{6}\theta(t),\\
\label{KK-asym-sink-2nd}
v(t) &\sim (9 - 5\sqrt{3})\theta(t)^{-2} + 3C\theta(t)^{-4+2\sqrt{3}} + \frac{15 + \sqrt{3}}{6}
\end{align}
with a free parameter $C\in \mathbb{R}$ and
\begin{equation}
\label{KK-asym-saddle-2nd}
u(t) \sim (3 + \sqrt{3})\theta(t)^{-1} + \frac{3 - \sqrt{3}}{6}\theta(t),\quad
v(t) \sim (9 + 5\sqrt{3})\theta(t)^{-2} + \frac{15 - \sqrt{3}}{6}
\end{equation}
\end{thm}

\begin{rem}
Define $\iota: \mathbb{Q}[\sqrt{3}]\to \mathbb{Q}[\sqrt{3}]$ by
\begin{equation*}
\iota(1) = 1,\quad \iota(\sqrt{3}) = -\sqrt{3}
\end{equation*}
and extend linearly on $\mathbb{Q}[\sqrt{3}] \equiv \{a+b\sqrt{3} \mid a,b\in \mathbb{Q}\}$.
Then $\iota^2 = {\rm id}$ holds. That is, $\iota$ is an involution.
We first observe in (\ref{KK}) that roots of the balance law $(U_0, V_0) = (3\pm \sqrt{3}, 9\pm 5\sqrt{3})$ are mapped to each other via $\iota$.
(\ref{KK-asym-sink-2nd}) with $C=0$ and (\ref{KK-asym-saddle-2nd}) indicate that the second order asymptotic expansions are also mapped via $\iota$ by mapping individual coefficients, which would expect the involution, or more general symmetry correspondence among different blow-up solutions, possibly with a specific choice of free parameters.
\end{rem}

\subsection{An artificial system in the presence of Jordan blocks}
\label{section-ex-log}

The next example \KMg{concerns} with an artificial system \KMg{such that} the blow-up power-determining matrix has a non-trivial Jordan block.
We see through this example that multi-order asymptotic expansion of blow-ups can be executed regardless of the presence of Jordan blocks.

\subsubsection{The presence of terms of order $k+\alpha_i - 1$}
First we consider
\begin{equation}
\label{log}
	u' = u^2 + v, \quad v' = au^3 + 3uv - u^2\KMg{,}
\end{equation}
where $a\in\mathbb{R}$ is a parameter.
This system is asymptotically quasi-homogeneous of type $\alpha = (1,2)$ and order $k+1=2$, 
consisting of the quasi-homogeneous part $f_{\alpha, k}$ and the lower-order part $f_{\rm res}$ given as follows: 
\begin{equation}
\label{log-vf}
	f_{\alpha, k}(u,v) = \begin{pmatrix}
u^2 + v\\
 au^3 + 3uv
 \end{pmatrix}, \quad f_{\mathrm{res}}(u,v) = \begin{pmatrix}
 0\\
 -u^2	
 \end{pmatrix}.
\end{equation}
Assume that the system (\ref{log}) admits a blow-up solution with the following asymptotic behavior:
\begin{equation*}
u(t) = O(\theta(t)^{-1}),\quad v(t) = O(\theta(t)^{-2})\quad \text{ as }\quad t\to t_{\max}-0.
\end{equation*}
Under this assumption and the associated ansatz
\begin{align}
\notag
u(t) &= \theta(t)^{-1}U(t) \equiv \theta(t)^{-1}\sum_{n=0}^\infty U_n(t),\quad U_n(t) \ll U_{n-1}(t),\quad \lim_{t\to t_{\max}}U_n(t) = U_0, \\
\label{series-log}
v(t) &= \theta(t)^{-2}V(t) \equiv \theta(t)^{-2}\sum_{n=0}^\infty V_n(t),\quad V_n(t) \ll V_{n-1}(t), \quad \lim_{t\to t_{\max}}V_n(t) = V_0,
\end{align}
we solve the balance law and investigate the eigenstructure of the associated blow-up power eigenvalues.
The balance law is
\begin{equation*}
\begin{pmatrix}
U_0 \\ 2V_0
\end{pmatrix} = \begin{pmatrix}
U_0^2 + V_0\\
aU_0^3 + 3U_0V_0
 \end{pmatrix}.
\end{equation*}
Our particular interest here is the case $a=0$, in which case the root is $(U_0, V_0) = (1,0)$\KMk{.} 
We fix $a=0$ for a while.
The blow-up power-determining matrix at $(U_0, V_0)$ is
\begin{equation*}
A = \begin{pmatrix}
-1 & 0 \\ 0 & -2
\end{pmatrix} + \begin{pmatrix}
2U_0 & 1\\
3V_0 & 3U_0
 \end{pmatrix} = \begin{pmatrix}
1& 1\\
0 & 1
 \end{pmatrix},
\end{equation*}
that is, the matrix $A$ has nontrivial Jordan block.
The governing system for determining $(U_1(t), V_1(t))$ is therefore
\begin{equation}
\label{2nd-log}
\frac{d}{dt}\begin{pmatrix}
U_1 \\
V_1
 \end{pmatrix} = \theta(t)^{-1}\begin{pmatrix}
1& 1\\
0 & 1
 \end{pmatrix}\begin{pmatrix}
U_1 \\
V_1
 \end{pmatrix} - \begin{pmatrix}
0 \\
U_0^2
\end{pmatrix} =  \theta(t)^{-1}\begin{pmatrix}
1& 1\\
0 & 1
 \end{pmatrix}\begin{pmatrix}
U_1 \\
V_1
 \end{pmatrix} - \begin{pmatrix}
0 \\
1
\end{pmatrix}.
\end{equation}
The fundamental matrix of the homogeneous part and its inverse are
\begin{equation*}
\Phi_2(t;1) = \begin{pmatrix}
\theta(t)^{-1} & \theta(t)^{-1}\ln(\theta(t)^{-1}) \\
0 & \theta(t)^{-1} 
\end{pmatrix},\quad 
\Phi_2(t;1)^{-1} = \begin{pmatrix}
\theta(t) & -\theta(t)\ln(\theta(t)^{-1}) \\
0 & \theta(t)
\end{pmatrix},
\end{equation*}
respectively.
\KMg{Because} all eigenvalues of $A$ are positive, \KMg{${\rm Spec}(A)$ makes} no contribution \KMg{to determine the order of $\theta(t)$}, while the combination of $\Phi_2(t;1)$ and the inhomogeneous term $(0,-1)^T$ induce the successive terms of $(U(t),V(t))$.
The general solution of (\ref{2nd-log}) satisfying the asymptotic assumption (\ref{series-log}) is
\begin{align*}
\begin{pmatrix}
U_1(t) \\ V_1(t)
\end{pmatrix} &= \KMd{-} \Phi_2(t;1) \KMd{\int_t^{t_{\max}}} \Phi_2(s;1)^{-1}\begin{pmatrix}
0 \\ -1
\end{pmatrix} ds\\
	&= \KMd{-} \Phi_2(t;1) \KMd{\int_t^{t_{\max}}} \begin{pmatrix}
\theta(s)\ln(\theta(s)^{-1})  \\ -\theta(s)
\end{pmatrix}ds\\
	&= \begin{pmatrix}
\theta(t)^{-1} & \theta(t)^{-1}\ln(\theta(t)^{-1}) \\
0 & \theta(t)^{-1} 
\end{pmatrix}\begin{pmatrix}
-\frac{1}{4}\theta(t)^2 \{ 2\ln(\theta(t)^{-1})+1 \} \\ \frac{1}{2}\theta(t)^2
\end{pmatrix}\\
	&= \begin{pmatrix}
-\frac{1}{4}\theta(t) \{ 2\ln(\theta(t)^{-1})+1 \}  + \frac{1}{2}  \theta(t)\ln(\theta(t)^{-1}) \\ 
\frac{1}{2}\theta(t)
\end{pmatrix}\\
	&= \frac{1}{4}\begin{pmatrix}
-\theta(t) \\ 
2\theta(t)
\end{pmatrix}.
\end{align*}
\KMd{
Now the constant $\delta$ defined in (\ref{suff-asym}) is estimated as follows.
First, $\alpha = (1,2)$ and $k=1$.
From (\ref{log-vf}) we know that $(\gamma_1, \gamma_2) = (+\infty, -2)$.
Because there are no negative blow-up power eigenvalues, $\delta$ is evaluated as
\begin{equation*}
\delta = \min \left\{ \frac{1}{1} +\infty, \frac{2}{1} - 2 \right\} + 1 = 1.
\end{equation*}
\KMg{Proposition} \ref{prop-order-increase} indicates that ${\rm ord}_\theta({\bf Y}_m) \geq m$ for $m\geq 2$ and hence there is no term $\theta(t)^m$ with $m\leq 1$ in ${\bf Y}_n(t) = (U_n(t), V_n(t))$, $n\geq 2$.
}
In particular, we have the second order asymptotic expansion of $(u(t), v(t))$ as $t\to t_{\max}-0$:
\begin{equation*}
u(t) \sim \theta(t)^{-1} - \frac{1}{4},\quad v(t) \sim \frac{1}{2}\theta(t)^{-1}.
\end{equation*}
The governing equation for the third term $(U_2(t), V_2(t))$ is
\begin{align}
\notag
\frac{d}{dt}\begin{pmatrix}
U_2 \\
V_2
 \end{pmatrix} &= \theta(t)^{-1} \left\{ \begin{pmatrix}
1& 1\\
0 & 1
 \end{pmatrix}\begin{pmatrix}
U_2 \\
V_2
 \end{pmatrix} + \begin{pmatrix}
U_1^2 \\
3U_1V_1
 \end{pmatrix}\right\} - \begin{pmatrix}
0 \\
2U_1
\end{pmatrix}\\
\notag
	&=  \theta(t)^{-1} \left\{ \begin{pmatrix}
1& 1\\
0 & 1
 \end{pmatrix}\begin{pmatrix}
U_2 \\
V_2
 \end{pmatrix} + \begin{pmatrix}
\frac{1}{16}\theta(t)^2 \\
-\frac{3}{8}\theta(t)^2
 \end{pmatrix}\right\} + \frac{1}{2}\begin{pmatrix}
0 \\
\theta(t)
\end{pmatrix}\\
\label{3rd-log}
&=  \theta(t)^{-1}  \begin{pmatrix}
1& 1\\
0 & 1
 \end{pmatrix}\begin{pmatrix}
U_2 \\
V_2
 \end{pmatrix} + \frac{1}{16}\begin{pmatrix}
\theta(t) \\
2\theta(t)
 \end{pmatrix}.
\end{align}

The solution is therefore
\begin{align*}
\begin{pmatrix}
U_2(t) \\ V_2(t)
\end{pmatrix} &=  \KMd{-}\frac{1}{16}\Phi_2(t;1) \KMd{\int_t^{t_{\max}}} \begin{pmatrix}
\theta(s) & -\theta(s)\ln(\theta(s)^{-1}) \\
0 & \theta(s)
\end{pmatrix}\begin{pmatrix}
\theta(s) \\ 2\theta(s)
\end{pmatrix} ds\\
	&= \KMd{-\frac{1}{16}}\Phi_2(t;1) \KMd{\int_t^{t_{\max}}} \begin{pmatrix}
\theta(s)^2 -2\theta(s)^2\ln(\theta(s)^{-1})  \\ 2\theta(s)^2
\end{pmatrix}ds\\
	&= \KMd{\frac{1}{16}}\begin{pmatrix}
\theta(t)^{-1} & \theta(t)^{-1}\ln(\theta(t)^{-1}) \\
0 & \theta(t)^{-1} 
\end{pmatrix}\begin{pmatrix}
-\frac{1}{3}\theta(t)^3 -2 \left[ -\frac{1}{9} \theta(t)^3 \{ 3\ln(\theta(t)^{-1})+1\}\right] \\ -\frac{2}{3}\theta(t)^3
\end{pmatrix}\\
	&= \KMd{\frac{1}{16}}\begin{pmatrix}
-\frac{1}{3}\theta(t)^2 -2 \left[ -\frac{1}{9} \theta(t)^2 \{ 3\ln(\theta(t)^{-1})+1\}\right] -\frac{2}{3}\theta(t)^2 \ln(\theta(t)^{-1}) \\ 
-\frac{2}{3}\theta(t)^2
\end{pmatrix}\\
	&= \KMg{\frac{-1}{144}\begin{pmatrix}
\theta(t)^2 \\ 
6\theta(t)^2
\end{pmatrix}}.
\end{align*}
As a summary, we obtain the following result.

\begin{thm}
\label{thm-asym-log}
The system (\ref{log}) with $a=0$ admits a blow-up solution with the following third-order asymptotic expansion as  $t\to t_{\max}-0$:
\begin{equation*}
u(t) \sim \theta(t)^{-1} - \frac{1}{4} - \KMd{\frac{1}{144}}\theta(t),\quad v(t) \sim \frac{1}{2}\theta(t)^{-1} \KMg{-} \frac{1}{24}.
\end{equation*}
\end{thm}

\subsubsection{The absence of terms of order $k+\alpha_i - 1$}

Here we consider
\begin{equation}
\label{log2}
	u' = u^2 + v, \quad v' = au^3 + 3uv - u
\end{equation}
with a real parameter $a\in\mathbb{R}$, instead of (\ref{log}).
The difference from (\ref{log}) is the replacement of $-u^2$ by $-u$ in the second component of the vector fields.
Because the quasi-homogeneous part is unchanged, \KMg{the balance law and blow-up power-determining matrix are the same} as those of (\ref{log})\KMj{.}
Now we consider the asymptopic expansion (\ref{series-log}) for (\ref{log2}):
\begin{equation}
\label{2nd-log2}
\frac{d}{dt}\begin{pmatrix}
U_1 \\
V_1
 \end{pmatrix} =  \theta(t)^{-1}\begin{pmatrix}
1& 1\\
0 & 1
 \end{pmatrix}\begin{pmatrix}
U_1 \\
V_1
 \end{pmatrix} - \begin{pmatrix}
0 \\
\theta(t)
\end{pmatrix}.
\end{equation}
The solution satisfying the asymptotic assumption (\ref{series-log}) is
\begin{align*}
\begin{pmatrix}
U_1(t) \\ V_1(t)
\end{pmatrix} &=  \KMd{-}\Phi_2(t;1)  \KMd{\int_t^{t_{\max}}} \Phi_2(s;1)^{-1}\begin{pmatrix}
0 \\ -\theta(s)
\end{pmatrix} ds\\
	&=  \KMd{-}\Phi_2(t;1) \KMd{\int_t^{t_{\max}}} \begin{pmatrix}
\theta(s)^2\ln(\theta(s)^{-1})  \\ -\theta(s)^2
\end{pmatrix}ds\\
	&= \begin{pmatrix}
\theta(t)^{-1} & \theta(t)^{-1}\ln(\theta(t)^{-1}) \\
0 & \theta(t)^{-1} 
\end{pmatrix}\begin{pmatrix}
-\frac{1}{9}\theta(t)^3 \{ 3\ln(\theta(t)^{-1})+1 \} \\ \frac{1}{3}\theta(t)^3
\end{pmatrix}\\
	&= \begin{pmatrix}
-\frac{1}{9}\theta(t)^2 \{ 3\ln(\theta(t)^{-1})+1 \}  + \frac{1}{3} \theta(t)^2\ln(\theta(t)^{-1}) \\ 
\frac{1}{3}\theta(t)^2
\end{pmatrix}\\
	&= \frac{1}{9}\begin{pmatrix}
-\theta(t)^2 \\ 
3\theta(t)^2
\end{pmatrix}.
\end{align*}
Similar to (\ref{log}), the constant $\delta$ in (\ref{suff-asym}) is estimated as 
\begin{equation*}
\delta = \min \left\{ \frac{1}{1} + \infty, \frac{2}{1} - 1 \right\} + 1 = 2,
\end{equation*}
where we have used $(\gamma_1, \gamma_2) = (+\infty, -1)$.
\KMg{Proposition} \ref{prop-order-increase} indicates that ${\rm ord}_\theta({\bf Y}_m) \geq 2m$ for $m\geq 2$ and hence there is no term $\theta(t)^m$ with $m\leq 2$ in ${\bf Y}_n(t) = (U_n(t), V_n(t))$, $n\geq 2$.
As a summary, we have the following result. 
\begin{thm}
\label{thm-asym-log2}
The system (\ref{log2}) with $a=0$ admits a blow-up solution with the following second-order asymptotic expansion as  $t\to t_{\max}-0$:
\begin{equation*}
u(t) \sim \theta(t)^{-1} - \frac{1}{9} \theta(t),\quad v(t) \sim \frac{1}{3}.
\end{equation*}
\end{thm}
In contrast to Theorem \ref{thm-asym-log} for (\ref{log}), $v(t)$ remains bounded as $t\to t_{\max}$, while $u(t)$ blows up\footnote{
In some context in partial differential equations, this kind of blow-up behavior is referred to as {\em nonsimultaneous blow-up} (cf. \cite{Ha2016, Ha2017}).
}.
Another interesting observation here is the absence of logarithmic terms in the asymptotic expansion, although the fundamental matrix \KMg{$\Phi_2(t;1)$} includes the logarithmic terms.
In practical calculations, there are cancellations of logarithmic terms in integration of inhomogeneous terms, which can also occur in calculations of higher-order terms $(U_n(t), V_n(t))$, $n\geq 3$.

%% file: appendix_A.tex
\subsection{Right-sided compactifications for asymptotically autonomous systems}
\label{section-appendix-right-cpt}

In several results, {\em time-variable compactifications for asymptotically autonomous systems} introduced in e.g. \cite{WXJ2021} are applied to completing our proofs discussed in the next section.
These advanced notions in dynamical systems are briefly summarized in the present section.
Note that simpler proofs can be constructed, which will yield the same statement in more general setting.
\par
Recall that our main issue is (\ref{blow-up-basic}). 
Now we introduce the new time variable
\begin{equation}
\label{time-trans-t-s}
s = -\ln\theta(t)\quad \Leftrightarrow \quad \theta(t) = e^{-s},\quad \frac{df}{dt} = \frac{df}{ds}\frac{ds}{dt} = \theta(t)^{-1}\frac{df}{ds}.
\end{equation}
Then the system (\ref{blow-up-basic}) is transformed into
\begin{equation}
\label{blow-up-basic-s}
\frac{d}{ds}{\bf Y} = -\KMf{ \frac{1}{k}\Lambda_\alpha }{\bf Y} + f_{\alpha, k}({\bf Y}) + e^{-s(\KMg{I}+ \KMf{ \frac{1}{k}\Lambda_\alpha })}f_{{\rm res}}(e^{ \KMf{ \frac{s}{k}\Lambda_\alpha }} {\bf Y}).
\end{equation}

By the asymptotic quasi-homogeneity of $f$, the nonautonomous system (\ref{blow-up-basic-s}) is {\em asymptotically autonomous} as $s\to +\infty$ (corresponding to $t\to t_{\max}-0$) in the sense that (\ref{blow-up-basic-s}) admits the {\em future limit system} as $s\to \infty$ (cf. \cite{WXJ2021}):  
\begin{equation}
\label{blow-up-basic-s-limit}
\frac{d}{ds}{\bf Y} = - \KMf{ \frac{1}{k}\Lambda_\alpha } {\bf Y} + f_{\alpha, k}({\bf Y}).
\end{equation}
In particular, $e^{-s(\KMg{I}+ \KMf{ \frac{1}{k}\Lambda_\alpha } )}f_{{\rm res}}(e^{\KMf{ \frac{s}{k}\Lambda_\alpha }} {\bf Y}) \to 0$ as $s\to +\infty$, which is locally uniform in ${\bf Y}$.
The root of the balance law ${\bf Y}_0$ is an equilibrium of the future limit system (\ref{blow-up-basic-s-limit}).
Here we briefly summarize the technique of {\em compactification of for asymptotically autonomous systems in the time variable} introduced in \cite{WXJ2021}.
In particular, {\em right-sided compactification} is reviewed.

\begin{ass}[\cite{WXJ2021}, Assumption 2.2]
The general system
\begin{equation*}
\frac{d{\bf y}}{ds} = f({\bf y}, \Gamma(s)),\quad \Gamma: \mathbb{R}\to V:\, C^1
\end{equation*}
is considered with an open interval $V\subset \mathbb{R}$, where $f:U\times V\to \mathbb{R}^n$ with an open subset $U\subset \mathbb{R}^n$ is $C^r$ with \KMc{$r \geq 1$}. 
\KMl{The function $\Gamma$ is assumed to be asymptotically constant in the sense that the future limit $\displaystyle{\lim_{t\to +\infty}\Gamma(t) = \Gamma_+ \in V}$ exists.}
A coordinate transform $\tilde s = \KMf{\phi}(s)$ satisfying all the following properties is chosen:
\begin{align*}
&\KMf{\phi}: [s_- , \infty) \to [\tilde s_-, 1),\quad \KMf{\phi}\in C^{k\geq 2},\quad \lim_{s\to +\infty} \KMf{\phi}(s) = 1,\\
&\frac{d\KMf{\phi}}{ds}(s) > 0\quad (s\geq s_-),\quad \lim_{s\to +\infty} \frac{d\KMf{\phi}}{ds}(s) = 0,
\end{align*}
with the inverse $h(\tilde s) = \KMf{\phi}^{-1}(\tilde s)$.
\end{ass}
The resulting autonomous right-sided compactified system is the following:
\begin{equation}
\label{comp-right}
\frac{d{\bf y}}{ds} = f({\bf y}, \Gamma(h(\tilde s))),\quad \frac{d\tilde s}{ds} = \gamma(\tilde s),
\end{equation}
where 
\begin{align*}
 f({\bf y}, \Gamma(h(\KMl{\tilde s}))) &= \begin{cases}
	f({\bf y}, \Gamma(h(\tilde s))) & \tilde s \in [\tilde s_-, 1),\\
	f({\bf y}, \Gamma^+) & \tilde s = 1,\\
\end{cases}\\
\gamma(\tilde s) &= \begin{cases}
	1/h'(\tilde s) & \tilde s \in [\tilde s_-, 1),\\
	0 & \tilde s = 1.\\
\end{cases}
\end{align*}
Then it follows from Proposition 2.1 and Theorem 2.2 in \cite{WXJ2021} that (\ref{comp-right}) is $C^1$-smooth on $U\times (\tilde s_-, 1]$ if and only if the following limits exist:
\begin{equation}
\label{limit-comp}
\lim_{s\to +\infty}\frac{\frac{d}{ds}\Gamma(s)}{\frac{d}{ds}\KMf{\phi}(s)} = \lim_{\tilde s\to +1-0}\frac{d}{d\tilde s}\Gamma(h(\tilde s)),\quad 
\lim_{s\to +\infty}\frac{\frac{d^2}{ds^2}\KMf{\phi}(s)}{\frac{d}{ds}\KMf{\phi}(s)} = -\lim_{\tilde s\to +1-0}\frac{\frac{d^2}{ds^2}h(\tilde s)}{ \left(\frac{d}{ds}h(\tilde s) \right)^2 }\KMi{.}
\end{equation}
In particular, under the existence of the above limits, the system (\ref{blow-up-basic-s}) is $C^1$-extended to the future-limit system (\ref{blow-up-basic-s-limit}) through the compactified system (\ref{comp-right}).
\par
\bigskip
By assumption of the blow-up power-determining matrix $A$, we have
\begin{equation*}
\sharp \{\lambda\in {\rm Spec}(A) \mid {\rm Re}\lambda < 0\} = m_A
\end{equation*}
and hence ${\bf Y}_0$ is a hyperbolic saddle of (\ref{blow-up-basic-s-limit}) admitting the $m_A$-dimensional stable manifold.
Here we apply the right-sided compactification in $s$ to (\ref{blow-up-basic-s}).
As a simple one, we apply the right-side exponential compactification (cf. Section 4.2 in \cite{WXJ2021})
\begin{equation}
\label{R-exp-comp}
\tilde s = \KMf{\phi}(s)\equiv \KMf{\phi}_{(\nu)}(s) = 1 - e^{-\nu s},\quad s = h(\tilde s)\equiv \KMf{\phi}_{(\nu)}^{-1}(\tilde s) = -\frac{1}{\nu}\ln (1-\tilde s)
\end{equation}
with $\nu \in (0,1)$, and let $\Gamma(s) = e^{-s}$.
\KMg{In particular, functions $\Gamma$, $\phi$ and $h$ are $C^\infty$-smooth for $s < \infty$, equivalently $\tilde s < 1$.}
Then limits in (\ref{limit-comp}) exist and the one-sided compactified {\em autonomous} system
\begin{align}
\label{Y-1side-compactified}
\frac{d}{ds}{\bf Y} &= - \KMf{ \frac{1}{k}\Lambda_\alpha } {\bf Y} + f_{\alpha, k}({\bf Y}) + (1-\tilde s)^{\nu^{-1}(\KMg{I} +\KMf{ \frac{1}{k}\Lambda_\alpha })} f_{{\rm res}}\left( (1-\tilde s)^{-\KMf{ \frac{1}{k\nu }\Lambda_\alpha }} {\bf Y} \right),\\
\notag
\frac{d\tilde s}{ds} &= \nu (1-\tilde s)
\end{align}
is $C^1$-smooth on $\mathbb{R}^n \times (-\infty,1]$ \KMl{jointly in $({\bf Y}, \tilde s)$} because $\nu < 1$, which is in fact $C^r$ in the present case by choosing $\nu$ sufficiently small, if necessary, where $r$ is the differentiability of $f$.
\KMi{This transformation} can be applied to asymptotically autonomous systems admitting the future limit system, such as (\ref{2nd}).

%% file: appendix_B.tex
\label{section-appendix-proof}

Proofs of statements in \KMc{Section \ref{section-asym}} are collected here.

\subsection{\KMd{Proof of Proposition \ref{prop-conv-2nd}}}
\label{section-app-proof-conv-2nd}

Apply the right-sided compactification to (\ref{2nd}) in the similar way to what we have discussed in Appendix \ref{section-appendix-right-cpt}, which yields
\begin{align}
\label{Y1-1side-compactified}
\frac{d}{ds}{\bf Y}_1 &= A{\bf Y}_1 + (1-\tilde s)^{\nu^{-1}(\KMg{I} + \KMf{ \frac{1}{k}\Lambda_\alpha } )} f_{{\rm res}}\left( (1-\tilde s)^{-\KMf{ \frac{1}{k\nu}\Lambda_\alpha } }  {\bf Y}_0 \right),\\
\notag
\frac{d\tilde s}{ds} &= \nu (1-\tilde s).
\end{align}
From the asymptotic quasi-homogeneity of $f$, this system possesses an equilibrium $\tilde {\bf Y}_1 \KMg{(}\equiv ({\bf Y}_1, \tilde s) \KMg{)}= ({\bf 0}, 1)$ on the future limit space $\{\tilde s = 1\}$, which is the only possible equilibrium so that the asymptotic relation (\ref{Y-asym}) is satisfied.
Moreover, we see  from the form of the system that eigenvalues of the linearized matrix at the equilibrium consists of ${\rm Spec}(A)$ and $\{-\nu\}$.
By assumption, the equilibrium $\tilde {\bf Y}_1$ is hyperbolic and hence, by the Stable Manifold Theorem (e.g. \cite{Rob}),  the $(m_A+1)$-dimensional local stable manifold $W^s_{\rm loc}(\tilde {\bf Y}_1)$ for (\ref{Y1-1side-compactified}) is constructed.
\KMi{Now} one variable generating $W^s_{\rm loc}(\tilde {\bf Y}_1)$ is $\tilde s$, which is identical with the time $s$ in the nonautonomous system
\begin{align}
\label{Y1-nonaut}
\frac{d}{ds}{\bf Y}_1 &= A{\bf Y}_1 + \tilde {\bf g}_1(s)
\end{align}
\KMi{before applying the right-sided compactification to obtaining (\ref{Y1-1side-compactified}),}
where
\begin{equation*}
\tilde {\bf g}_1(s) = e^{-s (\KMg{I} + \KMf{\frac{1}{k}\Lambda_\alpha })} f_{{\rm res}}\left( e^{ \KMf{\frac{s}{k}\Lambda_\alpha } } {\bf Y}_0 \right)\KMi{.}
\end{equation*}
\KMi{Therefore,} the remaining $m_A$ parameters generate the family of solutions ${\bf Y}_1(s)$.
\par
On the other hand, the solution ${\bf Y}_1(s)$ of (\ref{Y1-nonaut}) possesses the expression
\begin{equation}
\label{Y1_variation}
{\bf Y}_1(s) = e^{(s-s_0)A} \left\{ {\bf Y}_1^0 + \int_{s_0}^{s}e^{(s_0-\eta)A} \tilde {\bf g}_1(\eta) d\eta \right\}
\end{equation}
by the variation of constants formula, where $s_0 = -\ln \theta(t_0)$.
(\ref{Y1_variation}) is also written as
\begin{align}
\notag
{\bf Y}_1(s) &= e^{(s-s_0)A} \left\{ (P_- + P_+) {\bf Y}_1^0 + \int_{s_0}^{s}e^{-(\eta - s_0) A} (P_- + P_+)  \tilde {\bf g}_1(\eta) d\eta \right\}\\
\label{Y1_variation-2}
	&= e^{(s-s_0)A} \left\{ (P_- + P_+) {\bf Y}_1^0 + (P_- + P_+) \int_{s_0}^{s}e^{-(\eta - s_0) A} \tilde {\bf g}_1(\eta) d\eta \right\}
\end{align}
where we have used the $A$-invariance of $P_\pm$ and the fact that $P_\pm$ are independent of the time variable $s$ in the second equality.
Note that solutions ${\bf Y}_1(s)$ determining trajectories on $W^s_{\rm loc}(\tilde {\bf Y}_1)$ are bounded on $s\in [s_0, \infty)$.
In particular, such solutions ${\bf Y}_1(s)$ of the system (\ref{Y1-nonaut}) satisfy all requirements so that the {\em Lyapunov-Perron's method} (cf. Section III. 6 in \cite{Hale1969} or Proposition 6.4 in \cite{KR2011} for linear inhomogeneous systems) can be applied.
Consequently, a bounded solution ${\bf Y}_1(s)$ on $[s_0, \KMk{\infty})$ also admits the form
\begin{align}
\notag
\label{Y1_Perron}
{\bf Y}_1(s) &= e^{(s-s_0)A} \left\{ P_-{\bf Y}_1^0 + \int_{s_0}^s e^{-(\eta -s_0) A} P_-  \tilde {\bf g}_1(\eta) d\eta  -\int_{s}^{\infty} e^{ -(\eta -s_0) A} P_+  \tilde {\bf g}_1(\eta) d\eta \right\}
\end{align}
for {\em any} ${\bf Y}_1^0\in \mathbb{R}^n$.
This is also written by
\begin{align*}
\notag
{\bf Y}_1(s) &= e^{(s-s_0)A} \left\{ (P_- + P_+){\bf Y}_1^0 + \int_{s_0}^s e^{-(\eta -s_0) A} \tilde {\bf g}_1(\eta) d\eta \right. \\ 
	&\quad \left. - \left( P_+{\bf Y}_1^0 + P_+\int_{s_0}^s e^{-(\eta -s_0) A} \tilde {\bf g}_1(\eta) d\eta \right)  -\int_{s}^{\infty} e^{- (\eta -s_0) A} P_+ \tilde {\bf g}_1(\eta) d\eta \right\},
\end{align*}
which follows from (\ref{Y1_variation}) that
\begin{equation*}
P_+{\bf Y}_1^0 = -\int_{s_0}^{\infty} e^{-(\eta -s_0) A} P_+  \tilde {\bf g}_1(\eta) d\eta.
\end{equation*}
Substituting this into (\ref{Y1_variation-2}), we have
\begin{align}
\notag
{\bf Y}_1(s) &= e^{(s-s_0)A} \left\{ P_- \left( {\bf Y}_1^0 + \int_{s_0}^{s} e^{-(\eta -s_0)A} \tilde {\bf g}_1(\eta) d\eta \right) - P_+  \int_{s}^{\infty} e^{-(\eta -s_0) A} \tilde {\bf g}_1(\eta) d\eta \right\},\quad {\bf Y}_1^0 \in \mathbb{R}^n.
\end{align}
Back to the original time-scale $t$, we have
\begin{equation*}
{\bf Y}_1(t) = \left(\frac{\theta (t)}{\theta (t_0)}\right)^{-A} \left\{ P_- \left( {\bf Y}_1^0 + \int_{t_0}^t \left(\frac{\theta (\eta)}{\theta (t_0)}\right)^{A} {\bf g}_1(\eta) d\eta \right) - P_+  \int_{t}^{t_{\max}} \left(\frac{\theta (\eta)}{\theta (t_0)}\right)^{A} {\bf g}_1(\eta) d\eta \right\},
\end{equation*}
where $t_0 = t_{\max} - e^{-s_0}$.
By the uniqueness of solutions, this solution corresponds to a solution trajectory on $W^s_{\rm loc}(\tilde {\bf Y}_1)$ for (\ref{Y1-1side-compactified}) we have obtained first, whenever $\| P_- {\bf Y}_1^0\|$ is sufficiently small.

\subsection{Proof of Proposition \ref{prop-conv-(j+1)th}}
\label{section-app-proof-conv-(j+1)th}
The basic idea for proving the existence of a solution converging to ${\bf 0}\in \mathbb{R}^n$ is the same as the proof of Proposition \ref{prop-conv-2nd}.
That is, the right-sided compactification to the time-transformed system
\begin{equation}
\label{Yj-nonaut}
\frac{d}{ds}{\bf Y}_{j} = A {\bf Y}_j + \tilde {\bf g}_j(s).
\end{equation}
of (\ref{jth}) through $s = -\ln \theta(t)$ is discussed, where
\begin{align*}
 \tilde {\bf g}_j(s) &=  \left\{ R_{\alpha, k}({\bf S}_{j-1}({\bf Y})(s)) - R_{\alpha, k}({\bf S}_{j-2}({\bf Y})(s)) \right\}\\
	&\quad  + e^{-s(\KMg{I}+ \KMf{ \frac{1}{k}\Lambda_\alpha })}  \left\{ f_{{\rm res}}(e^{ \KMf{ \frac{s}{k}\Lambda_\alpha }}{\bf S}_{j-1}({\bf Y})(s)) - f_{{\rm res}}(e^{ \KMf{ \frac{s}{k}\Lambda_\alpha }} {\bf S}_{j-2}({\bf Y})(s) \right\}.
\end{align*}
The corresponding compactified system becomes
\begin{align}
\notag
\frac{d}{ds}{\bf Y}_j &= A{\bf Y}_j +\left\{ R_{\alpha, k}({\bf S}_{j-1}({\bf Y})(s)) - R_{\alpha, k}({\bf S}_{j-2}({\bf Y})(s)) \right\} \\
\label{Yj-1side-compactified}
	&\quad +  (1-\tilde s)^{\nu^{-1}(\KMg{I} + \KMf{ \frac{1}{k}\Lambda_\alpha })} \left\{ \KMg{ f_{{\rm res}} \left( (1-\tilde s)^{-\KMf{ \frac{1}{k\nu }\Lambda_\alpha }} {\bf S}_{j-1}({\bf Y})(s) \right) - f_{{\rm res}} \left( (1-\tilde s)^{-\KMf{ \frac{1}{k\nu }\Lambda_\alpha }}{\bf S}_{j-2}({\bf Y})(s) \right)}  \right\},\\
\notag
\frac{d\tilde s}{ds} &= \nu (1-\tilde s).
\end{align}
%
In particular, the $(m_A+1)$-dimensional local stable manifold $W^s_{\rm loc}(\tilde {\bf Y}_j)$ of the hyperbolic saddle $\tilde {\bf Y}_j = ({\bf Y}_j, \tilde s) = ({\bf 0}, 1)$ for (\ref{Yj-1side-compactified}) is constructed. 
Here recall that the second term ${\bf Y}_1(t)$ includes $m_A$ free parameters chosen from $\mathbb{E}_A^s$ and the power decay term \KMk{$\theta(\theta)^{-A}P_- {\bf Y}_1^0$}.
This \KMk{power decay} term also appears in ${\bf Y}_j(t)$ through the same arguments in those obtaining ${\bf Y}_1(t)$.
In contrast, this term has to \KMi{vanish so that} the asymptotic relation (\ref{Y-asym}) holds.

Now go back to the proof of Proposition \ref{prop-conv-(j+1)th}.
By the same arguments as Section \ref{section-app-proof-conv-2nd}, we know that (\ref{Yj-nonaut}) possesses the solution
\begin{equation}
\label{Yj_variation-stable}
{\bf Y}_j(s) = e^{(s-s_0)A} \left\{ P_- \left( {\bf Y}_j^0 + \int_{s_0}^{s} e^{-(\eta - s_0) A} \tilde {\bf g}_j(\eta) d\eta \right) - P_+  \int_{s}^{\infty} e^{-(\eta -s_0)A} \tilde {\bf g}_j(\eta) d\eta \right\},
\end{equation}
where ${\bf Y}_j^0 \in \mathbb{R}^n$, but the exponentially decaying terms characterized by ${\rm Spec}(A) \cap \{{\rm Re}\,\lambda < 0\}$ should be vanished by the asymptotic relation
\begin{equation*}
{\bf Y}_j(\KMk{s}) \ll {\bf Y}_1(\KMk{s}) \quad \text{ as }s\to \infty\KMi{,}
\end{equation*}
and hence the constant vector ${\bf Y}_j^0$ should be specified so that this requirement is satisfied.
\KMl{This constraint is considered below.}
\par
\bigskip
Back to the original time-scale $t$, we have
\begin{equation}
\label{Yj_variation-stable-proof}
{\bf Y}_j(t) = \left(\frac{\theta (t)}{\theta (t_0)}\right)^{-A} \left\{ P_- \left( {\bf Y}_j^0 + \int_{t_0}^t \left(\frac{\theta (\eta)}{\theta (t_0)}\right)^{A} {\bf g}_j(\eta) d\eta \right) - P_+  \int_{t}^{t_{\max}} \left(\frac{\theta (\eta)}{\theta (t_0)}\right)^{A} {\bf g}_j(\eta) d\eta \right\},
\end{equation}
where $t_0 = t_{\max} - e^{-s_0}$\KMl{.}
We shall determine the appropriate ${\bf Y}_j^0 = (Y_{j,1}^0, \ldots, Y_{j,n}^0)^T$.
To this end, consider the primitive function ${\bf H}_j(t)$ satisfying (\ref{integral-jth-H}).
Let 
\begin{equation*}
\nu_{j,i} := \KMj{\deg_\theta} \left( \left\{ \KMk{ \left(\frac{\theta (t)}{\theta (t_0)}\right)^{\Lambda} P^{-1} \KMj{P_-} {\bf g}_j(t) } \right\}_i \right),
\end{equation*}
\KMl{Now our determination of ${\bf Y}_j^0$ consists of two cases.}
\begin{description}
\item[Case 1. $\nu_{j,i} \leq -1$.] 
\end{description}
In this case, define $Z_{j,i} := ({\bf H}_j(\KMl{t_0}))_i$.
\begin{description}
\item[Case 2. $\nu_{j,i} > -1$.] 
\end{description}
In this case, similar to arguments in the proof of Proposition \ref{prop-deg-integral-theta-F}, we rewrite the rightmost integral in (\ref{integral-jth-H}) by
\begin{align}
\label{integral-jth-H-2}
&\left\{\int_{t_0}^{t_{\max}} \left(\frac{\theta (\eta)}{\theta (t_0)}\right)^{\Lambda} P^{-1}\KMk{P_-}  {\bf g}_j(\eta) d\eta - \int_{t}^{t_{\max}} \left(\frac{\theta (\eta)}{\theta (t_0)}\right)^{\Lambda} P^{-1}\KMk{P_-} {\bf g}_j(\eta) d\eta \right\}_i \\
\notag
	&\equiv I_{0; j,i} - \left\{\int_{t}^{t_{\max}} \left(\frac{\theta (\eta)}{\theta (t_0)}\right)^{\Lambda} P^{-1} \KMk{P_-} {\bf g}_j(\eta) d\eta \right\}_i \\
\notag
	&\KMl{= I_{0; j,i} + ({\bf H}_j(t))_i,}
\end{align}
\KMl{where we have used the fact that ${\bf H}_j(t)$ is chosen so that $\displaystyle{\lim_{t\to t_{\max}-0}} ({\bf H}_j(t)) = 0$. }
By the similar arguments to Proposition \ref{prop-ord-fundamental}, the first term converges to a finite value $I_{0;j,i}$ as the improper integral.
By definition of improper integrals and (\ref{integral-jth-H}), we have
\begin{equation*}
I_{0; j,i} = \lim_{t\to t_{\max}-0} ({\bf H}_j(t))_i - ({\bf H}_j(t_0))_i = -({\bf H}_j(t_0))_i.
\end{equation*}
In this case, define \KMl{$Z_{j,i} := -I_{0; j,i} = {\bf H}_j(t_0)$}.
Finally, define the vector ${\bf Z}_j$ by ${\bf Z}_j := (Z_{j,1}, \ldots, Z_{j,n})^T \equiv {\bf H}_j(t_0)$.
\par
\bigskip
Because ${\bf Z}_j$\KMl{, namely ${\bf H}_j(t_0)$,} is uniquely determined once \KMl{$t_0$,} ${\bf g}_j(\eta)$ is, in particular the lower terms $\{{\bf Y}_l(t)\}_{l=0}^{j-1}$ are determined, and hence letting ${\bf Y}_j^0$ so that
\begin{equation*}
P_- \KMi{{\bf Y}_j^0} =  P_- (P{\bf H}_j(t_0)),
\end{equation*}
the $\mathbb{E}_A^s$-component of the constant term in \KMk{(\ref{Yj_variation-stable-proof})} becomes $0$ \KMk{for all $t\in (t_0, t_{\max})$}.
The initial point at $t=t_0$ is uniquely determined by
\begin{equation*}
{\bf Y}_j(t_0) =  P_- {\bf Y}_j^0 - P_+  \int_{t_0}^{t_{\max}} \left(\frac{\theta (\eta)}{\theta (t_0)}\right)^{A} {\bf g}_j(\eta) d\eta,
\end{equation*}
which contains free parameters determined by ${\bf Y}_1(t)$ in ${\bf g}_j(t)$.
The proof is therefore completed.

\subsection{Proof of Theorem \ref{thm-smoothness-Y}}
\label{section-app-proof-smoothness-Y}

By the same argument as Theorem 3.2 in \cite{WXJ2021}, we know that the point $\tilde {\bf Y}_0 \equiv ({\bf Y}_0, 1)$ is a hyperbolic saddle for \KMk{the one-sided} compactified system \KMk{(\ref{Y-1side-compactified})}.
Then the saddle $\tilde {\bf Y}_0$ admits the $({m_A}+1)$-dimensional local stable manifold $W^{s}_{\rm loc}(\tilde {\bf Y}_0)$ in $\mathbb{R}^n \times (-\infty,1]$ by the Stable Manifold Theorem (e.g. \cite{Rob}).
From the $C^\infty$-smoothness of $\KMf{\phi_{(\nu)}}$, $\Gamma(s) = e^{-s}$ on $\mathbb{R}$ and $f$ on $\mathbb{R}^n$, the smooth dependence of solutions of (\ref{blow-up-basic-s}) on initial points and time as well as the uniqueness of solutions yield that $W^{s}_{\rm loc}(\tilde {\bf Y}_0)\cap (\mathbb{R}^n \times (-\infty,1))$ is indeed $C^r$-smooth in the $s$-\KMm{timescale}.
Moreover, the Stable Manifold Theorem implies that the convergence of solutions to ${\bf Y}_0$ is exponential, that is, for each $i = 1,\ldots, n$, we have
\begin{equation}
\label{Ysys-exp-decay}
Y_i(s) - Y_{i,0} = C_i e^{-\mu_i s}(1+o(1))\quad \text{ as }\quad s\to +\infty
\end{equation}
with a constant $C_i$ and an exponent $\mu_i$ with ${\rm Re}\,\mu_i > 0$.
Finally we go back to the system in the original $t$-\KMm{timescale}.
The exponential decay (\ref{Ysys-exp-decay}) implies that 
\begin{equation*}
Y_i(t) - Y_{i,0} = \tilde C_i (t_{\max}-t)^{\mu_i}(1+o(1))\quad \text{ as }\quad t\to t_{\max}\KMi{,}
\end{equation*}
which is continuous in $t$ including $t\to t_{\max}$, and $C^r$-smooth for $t < t_{\max}$.

\subsection{Proof of Proposition \ref{prop-convergence-Y}}
\label{section-app-proof-convergence-Y}

The basic idea is the same as the proof of Theorem \ref{thm-smoothness-Y} except the system we consider is 
\begin{align}
\notag
\frac{d}{dt} {\bf Y}_N^c &= \theta(t)^{-1} \left[ \left( -\KMf{ \frac{1}{k}\Lambda_\alpha } + Df_{\alpha, k}({\bf Y}_0) \right){\bf Y}_N^c + \left\{ R_{\alpha, k}( {\bf S}_N({\bf Y}) + {\bf Y}_N^c ) - R_{\alpha, k}( {\bf S}_N({\bf Y}) )\right\} \right]\\
\label{asym-eq-remainder}
	&\quad + \theta(t)^{\KMf{ \frac{1}{k}\Lambda_\alpha }} \left\{ f_{{\rm res}}\left( \theta(t)^{\KMf{ -\frac{1}{k}\Lambda_\alpha }} \left\{ {\bf S}_N({\bf Y}) +  {\bf Y}_N^c \right\} \right) - f_{{\rm res}}\left( \theta(t)^{\KMf{ -\frac{1}{k}\Lambda_\alpha }} {\bf S}_N({\bf Y}) \right) \right\}
\end{align}
\KMe{instead of (\ref{asym-eq}).}
Using (\ref{time-trans-t-s}), (\ref{asym-eq-remainder}) is transformed into
\begin{align}
\notag
\frac{d}{ds} {\bf Y}_N^c &=  \left( -\KMf{ \frac{1}{k}\Lambda_\alpha } + Df_{\alpha, k}({\bf Y}_0) \right) {\bf Y}_N^c + \left\{ R_{\alpha, k}( {\bf S}_N({\bf Y}) + {\bf Y}_N^c ) - R_{\alpha, k}( {\bf S}_N({\bf Y}) )\right\} \\
\label{blow-up-basic-remainder-s}
	&\quad + e^{-s(\KMg{I} + \KMf{ \frac{1}{k}\Lambda_\alpha } )} \left\{ f_{{\rm res}}\left( e^{ \KMf{ -\frac{s}{k}\Lambda_\alpha }} \{ {\bf S}_N({\bf Y}) + {\bf Y}_N^c \} \right) - f_{{\rm res}}\left( e^{ \KMf{ -\frac{s}{k}\Lambda_\alpha }} {\bf S}_N({\bf Y}) \right) \right\}\KMf{.}
\end{align}
Applying the right-side exponential compactification (\ref{R-exp-comp}) to (\ref{blow-up-basic-remainder-s}), the following system is obtained:
\begin{align*}
\frac{d}{ds}{\bf Y}_N^c &=  \left( -\KMf{ \frac{1}{k}\Lambda_\alpha } + Df_{\alpha, k}({\bf Y}_0) \right) {\bf Y}_N^c + \left\{ R_{\alpha, k}( {\bf S}_N({\bf Y}) + {\bf Y}_N^c ) - R_{\alpha, k}( {\bf S}_N({\bf Y}) )\right\} \\
	&\quad + (1-\tilde s)^{ \nu^{-1} (\KMg{I} + \frac{1}{k}\Lambda_\alpha ) } \left\{ f_{{\rm res}}\left( (1-\tilde s)^{ \frac{1}{k\nu}\Lambda_\alpha }  \{ {\bf S}_N({\bf Y}) + {\bf Y}_N^c \} \right) - f_{{\rm res}}\left( (1-\tilde s)^{ \frac{1}{k\nu}\Lambda_\alpha }  {\bf S}_N({\bf Y}) \right) \right\},\\
\notag
\frac{d\tilde s}{ds} &= \nu (1-\tilde s).
\end{align*}
By the definition of $R_{\alpha, k}$, ${\bf Y}_N^c = {\bf 0} \in \mathbb{R}^n$ is an equilibrium of the future limit system
\begin{align*}
\notag
\frac{d}{ds}{\bf Y}_N^c &=  \left( -\KMf{ \frac{1}{k}\Lambda_\alpha } + Df_{\alpha, k}({\bf Y}_0) \right) {\bf Y}_N^c + \left\{ R_{\alpha, k}( {\bf S}_N({\bf Y}) + {\bf Y}_N^c ) - R_{\alpha, k}( {\bf S}_N({\bf Y}) )\right\}
\end{align*}
whose linearized matrix at ${\bf Y}_N^c = {\bf 0}$ is $A = -\KMf{ \frac{1}{k}\Lambda_\alpha } + Df_{\alpha, k}({\bf Y}_0)$.
Now ${\bf Y}_j (t) \ll {\bf Y}_0 \in \mathbb{R}^n$ as $t\to t_{\max}$ is required for all $j\geq 1$ by our construction of (\ref{formula-asym}), which indicates that ${\bf 0}\in \mathbb{R}^n$ is the only possible limit of the remainder ${\bf Y}_N^c$. 
By assumption, \KMe{$({\bf Y}_N^c, \tilde s) = \KMk{({\bf 0}, 1)}$} is saddle and hence \KMe{it} admits $(m_A+1)$-dimensional local stable manifold\KMi{. In particular, any solution ${\bf Y}_N^c$ on the manifold converges to } ${\bf 0} \in \mathbb{R}^n$ as $s\to \KMk{\infty}$, equivalently $t\to t_{\max}$ for small free parameters $(C_1, \ldots, C_{m_A})$.

\subsection{\KMd{Proof of Theorem \ref{thm-asym-QH}}}
\label{section-app-proof-asym-QH}

Note that there is a nonsingular matrix $P$ such that $P^{-1}AP = \Lambda$ is the Jordan normal form.
It follows from this transformation that
\begin{equation*}
\theta(t)^A \equiv e^{(\ln \theta(t))A} = e^{(\ln \theta(t))P\Lambda P^{-1}} = Pe^{(\ln \theta(t))\Lambda }P^{-1} = P\theta(t)^\Lambda P^{-1}.
\end{equation*}
First assume that $A$ is diagonalizable, in which case
\begin{equation*}
\theta(t)^{\Lambda} = {\rm diag}\left( \theta(t)^{\lambda_1}, \ldots, \theta(t)^{\lambda_n} \right),\quad \{\lambda_i\}_{i=1}^n = {\rm Spec}(A).
\end{equation*}
All possible orders of $\theta(t)$ are obviously listed as \KMe{$-\beta\cdot_A \lambda$ when} $\beta\in \mathbb{Z}_{\geq 0}^n$ with $|\beta| = 1$.
Assume that, for given $N\geq 1$, all possible orders of $\theta(t)$ are listed as the collection \KMe{$\{-\beta\cdot_A \lambda\}$} for $\beta\in \mathbb{Z}_{\geq 0}^n$.
The system for determining ${\bf Y}_{N+1}$ is 
\begin{align*}
\frac{d}{dt} {\bf Y}_{N+1} &= \theta(t)^{-1} \left[ A{\bf Y}_{N+1} +  \left\{ R_{\alpha, k}({\bf S}_N({\bf Y})(t)) - R_{\alpha, k}({\bf S}_{N-1}({\bf Y})(t)) \right\} \right].
\end{align*}
\KMk{Because} $R_{\alpha,k}$ is the remainder of quasi-homogeneous vector field, then the remainder $R_{\alpha, k}({\bf S}_N({\bf Y})) - R_{\alpha, k}({\bf S}_{N-1}({\bf Y}))$ consists of powers of the form $\theta(t)^{\KMl{-\beta\cdot_A \lambda}}$ with $\beta\in \mathbb{Z}_{\geq 0}^n$.
The solution ${\bf Y}_{N+1} = {\bf Y}_{N+1}(t)$ is 
\begin{align}
\notag
&{\bf Y}_{N+1}(t) \\
\notag
&= P\left(\frac{\theta (t)}{\theta (t_0)}\right)^{-\Lambda}P^{-1} \left\{ P_- \left( {\bf Y}_{N+1}^0 + \int_{t_0}^t P \left(\frac{\theta (\eta)}{\theta (t_0)}\right)^{\Lambda}P^{-1} {\bf g}_{N+1} (\eta) d\eta \right) - P_+ \int_{t}^{t_{\max}} P \left(\frac{\theta (\eta)}{\theta (t_0)}\right)^{\Lambda}P^{-1} {\bf g}_{N+1}(\eta) d\eta \right\} \\
\notag
&=  P\left(\frac{\theta (t)}{\theta (t_0)}\right)^{-\Lambda}P^{-1} P_- \left( {\bf Y}_{N+1}^0 + \int_{t_0}^t P\left(\frac{\theta (\eta)}{\theta (t_0)}\right)^{\Lambda}P^{-1}  \left\{ \theta(\eta)^{-1} \left(R_{\alpha, k}({\bf S}_{N}({\bf Y})(\eta)) - R_{\alpha, k}({\bf S}_{N-1}({\bf Y})(\eta)) \right) \right\} d\eta \right) \\
\label{YNp1}
&\quad - P\left(\frac{\theta (t)}{\theta (t_0)}\right)^{-\Lambda}P^{-1} \left( P_+ \int_{t}^{t_{\max}} P\left(\frac{\theta (\eta)}{\theta (t_0)}\right)^{\Lambda}P^{-1}  \left\{ \theta(\eta)^{-1} \left(R_{\alpha, k}({\bf S}_{N}({\bf Y})(\eta)) - R_{\alpha, k}({\bf S}_{N-1}({\bf Y})(\eta)) \right) \right\}   d\eta \right).
\end{align}
Integrating all terms individually, we know that the resulting functions also consist of powers of the form \KMe{$\theta(t)^{-\beta\cdot_A \lambda}$} with $\beta\in \mathbb{Z}_{\geq 0}^n$.
\par
Next consider the general case where $A$ admits non-trivial Jordan blocks.
From the form of the fundamental matrix \KMe{$\theta(t)^A$}, we know that all components of ${\bf Y}_1(t)$ are linear combinations of functions of the form \KMe{$\theta(t)^{-\beta\cdot_A \lambda}(\ln \theta(t))^M$} with $\beta\in \mathbb{Z}_{\geq 0}^n$ satisfying $|\beta| = 1$ and $M\in \mathbb{Z}_{\geq 0}$ not greater than $\sum_{i=1}^d \sum_{l=1}^{\mu_i}(m_{i,l}-1)$.
Recall that $\{m_{i,l}\}_{l=1}^{\mu_i}$ denotes the collection of natural numbers corresponding to the size of Jordan blocks $J_\ast(\tilde \lambda_i)$, for $i=1,\ldots, d$, associated with mutually disjoint eigenvalues $\{\tilde \lambda_i\}_{i=1}^d$ of $A$\KMi{,} and $\mu_i$ denotes the geometric multiplicity of $\tilde \lambda_i$.
The exponent $M$ is determined by the number of eigenvalues admitting non-trivial Jordan blocks and their size.
Assume that, for $N\geq 1$, ${\bf Y}_N(t)$ consists of the linear combination of functions of the form \KMe{$\theta(t)^{-\beta\cdot_A \lambda}(\ln \theta(t))^M$} with $\beta\in \mathbb{Z}_{\geq 0}^n$ and $M\in \mathbb{Z}_{\geq 0}$.
Then the remainder $R_{\alpha, k}({\bf S}_N({\bf Y})) - R_{\alpha, k}({\bf S}_{N-1}({\bf Y}))$ also consists of linear combinations of powers of the form
\begin{equation*}
\KMe{\theta(t)^{-\beta\cdot_A \lambda} } (\ln \theta(t))^M,\quad \beta\in \mathbb{Z}_{\geq 0}^n,\quad M\in \mathbb{Z}_{\geq 0}.
\end{equation*}
Similar to the diagonalizable case, the solution ${\bf Y}_{N+1}(t)$ is written by (\ref{YNp1})
and we see that it is sufficient to verify \KMe{integrals} of \KMe{$\theta(\eta)^{ -\beta\cdot_A \lambda -1} (\ln \theta(\eta))^M$} \KMl{in $\eta$}, which are
\begin{align*}
\notag
\int_{t_0}^t \theta(\eta)^{-\beta\cdot_A \lambda -1} (\ln \theta(\eta))^M d\eta &= \left[ \frac{1}{\beta\cdot_A \lambda}\KMe{ \theta(\eta)^{-\beta\cdot_A \lambda}} (\ln \theta(\eta))^M\right]_{t_0}^t \KMe{-} \frac{1}{\beta\cdot_A \lambda} \int_{t_0}^t \KMe{ \theta(\eta)^{-\beta\cdot_A \lambda} } \left\{ \frac{d}{d\eta}(\ln \theta(\eta))^M \right\} d\eta \\
\notag
	&\quad + \KMe{ \frac{M(M-1)}{(\beta\cdot_A \lambda)^2} \int_{t_0}^t \theta(\eta)^{-\beta\cdot_A \lambda-1} } (\ln \theta(\eta))^{M-2} d\eta \\
\notag
	&= \cdots \\
	&= \left\{ \sum_{l=0}^M \frac{M!}{(M-l)!} \left( \KMe{ \frac{1}{\beta\cdot_A \lambda} } \right)^{l+1} (\ln \theta(t))^{M-l}
\right\} \KMe{\theta(t)^{-\beta\cdot_A \lambda}} + \text{(constant)}
\end{align*}
and
\begin{align*}
\int_t^{t_{\max}} \theta(\eta)^{-\beta\cdot_A \lambda -1} (\ln \theta(\eta))^M ds &= \left[ \frac{1}{\beta\cdot_A \lambda}\KMe{ \theta(\eta)^{-\beta\cdot_A \lambda}} (\ln \theta(\eta))^M\right]_t^{t_{\max}} \KMe{-} \frac{1}{\beta\cdot_A \lambda} \int_t^{t_{\max}} \KMe{ \theta(\eta)^{-\beta\cdot_A \lambda} } \left\{ \frac{d}{d\eta}(\ln \theta(\eta))^M \right\} d\eta \\
	&=  \frac{-1}{\beta\cdot_A \lambda} \KMe{ \theta(\eta)^{-\beta\cdot_A \lambda}} (\ln \theta(\eta))^M  \KMe{+ \frac{M}{\beta\cdot_A \lambda}} \int_t^{t_{\max}} \KMe{\theta(\eta)^{-\beta\cdot_A \lambda-1} } (\ln \theta(\eta))^{M-1} d\eta \\
	&= \cdots \\
	&= \KMe{-}\left\{ \sum_{l=0}^M \frac{M!}{(M-l)!} \left( \KMe{ \frac{1}{\beta\cdot_A \lambda} } \right)^{l+1} (\ln \theta(t))^{M-l}
\right\} \theta(t)^{-\beta\cdot_A \lambda}
\end{align*}
whenever $\beta\cdot_A \lambda \not = 0$, where we have used 
\begin{equation*}
\lim_{\KMf{t\to} t_{\max}-0} \frac{\log (t_{\max}- t)}{(t_{\max} - t)^a} = 0 \quad \text{ for }\KMl{a < 0}.
\end{equation*}
When $\beta\cdot_A \lambda = 0$ with $\beta \not \equiv 0$, this indicates that ${\rm Spec}(A)\subset \{{\rm Re}\,\lambda > 0\}$ and only integrals over $[t, t_{\max})$ are involved. 
The integral is then calculated as follows:
\begin{align*}
\int_t^{t_{\max}} \theta(\eta)^{\beta\cdot_A \lambda -1} (\ln \theta(\eta))^M ds &= \int_t^{t_{\max}} \theta(\eta)^{-1} (\ln \theta(\eta))^M d\eta \\
	&= \int_t^{t_{\max}} -\frac{d}{d\eta} (\ln \theta(\eta)) (\ln \theta(\eta))^M d\eta \\
	&= \KMi{\frac{-1}{M+1}}(\ln \theta(t))^{M+1}.
\end{align*}
In every cases, the integral is described by the linear combination of functions of the form $\KMe{ \theta(\eta)^{-\beta\cdot_A \lambda }} (\ln \theta(\eta))^M$ with $\beta\in \mathbb{Z}_{\geq 0}^n$ and $M\in \mathbb{Z}_{\geq 0}$ and hence all components of ${\bf Y}_{N+1}(t)$ are again linear combinations of functions of the form \KMi{$\theta(t)^{-\beta\cdot_A \lambda}(\ln \theta(\eta))^M$} with $\beta\in \mathbb{Z}_{\geq 0}^n$ and $M\in \mathbb{Z}_{\geq 0}$.
Exponents of $\ln \theta(t)$ in ${\bf Y}_{N+1}(t)$ is determined by $N+1, \alpha, k$ and $\{m_{i,l}\}_{l=1}^{\mu_i}$, $i=1,\ldots, d$.

\subsection{Proof of Proposition \ref{prop-order-increase}}
First we have
$\KMj{\deg_\theta}(Y_{0, i}) = 0$ for all $i = 1, \ldots, n$ when $Y_{0,i}\not = 0$, otherwise $\KMj{\deg_\theta}(Y_{0, i}) = +\infty$, by Example \ref{ex-ord}.
These identities imply (\ref{estimate-order-asym}) with $N=0$.
\par
\bigskip
Next consider ${\bf Y}_1(t)$.
By \KMi{definition of $\gamma_i$}, we have
\begin{align}
\label{estimate-g1-deg}
\KMj{\deg_\theta}\left( \{{\bf g}_1(t)\}_i \right)  \equiv \KMj{\deg_\theta}\left( \theta(t)^{\alpha_i/k} f_{i;{\rm res}}( \theta(t)^{-\KMf{ \frac{1}{k}\Lambda_\alpha }} {\bf Y}_0) \right) &\geq \frac{\alpha_i}{k} + \gamma_i 
\end{align}
and hence $\deg_\theta \left( {\bf g}_1(t) \right) >-1$ by assumption.
By (\ref{deg-integral-theta-F-minus}) in Proposition \ref{prop-deg-integral-theta-F}, we have
\begin{align}
\label{ord-Y1-last-Pminus}
\KMj{\deg_\theta} \left( \int_{t_0}^t \left( \frac{\theta(\eta)}{\theta(t)}\right)^{A} P_- {\bf g}_1(\eta) d\eta  \right)  &\geq \min \left\{ \min_{l=1,\ldots, n}\left\{ \frac{\alpha_l}{k} +  \gamma_l \right\} + 1, \min_{\substack{\lambda \in {\rm Spec}(A)\\ {\rm Re}\,\lambda < 0}} (-{\rm Re}\,\lambda) \right\}
\end{align}
as the estimate of $\KMj{\deg_\theta}$ for vector-valued functions.
Therefore we have
\begin{align}
\notag
&\KMj{\deg_\theta} \left( \left( \frac{\theta(t)}{\theta(t_0)}\right)^{-A} P_- \KMi{\left( {\bf Y}_1^0 + \int_{t_0}^t \left( \frac{\theta(\eta)}{\theta(t_0)}\right)^{A} {\bf g}_1(\eta) d\eta \right)} \right) \\
\notag
&\quad \geq \min\left\{  \KMj{\deg_\theta} \left( \left( \frac{\theta(t)}{\theta(t_0)}\right)^{-A} P_- {\bf Y}_1^0 \right),  \quad\KMj{\deg_\theta} \left(  \int_{t_0}^t \left( \frac{\theta(\eta)}{\theta(t)}\right)^{A} P_- {\bf g}_1(\eta) d\eta \right) \right\}\\
\notag
&\quad \quad \text{(from Proposition \ref{prop-ord-fundamental}-2)}\\
\label{ord-Y1-last-Pminus-2}
&\quad \geq \min\left\{ \min_{i=1,\ldots, n}\left\{ \frac{\alpha_i}{k} +  \gamma_i \right\} + 1,\quad  \min_{ \substack{ i=1,\ldots, n\\ {\rm Re}\, \lambda_i < 0}} {\rm Re}\, (-\lambda_i)  \right\} = \delta \quad \text{(from (\ref{ord-Y1-last-Pminus}))},
\end{align}
where we have used the commutativity $AP_- = P_-A$.
Similarly, from (\ref{deg-integral-theta-F-plus}) in Proposition \ref{prop-deg-integral-theta-F}, we have
\begin{align}
\label{ord-Y1-last-Pplus}
\KMj{\deg_\theta} \left( \int_{t}^{t_{\max}} \left( \frac{\theta(\eta)}{\theta(t)}\right)^{A} P_+ {\bf g}_1(\eta) d\eta  \right)  &\geq  \min_{l=1,\ldots, n}\left\{ \frac{\alpha_l}{k} +  \gamma_l \right\} + 1\geq \delta.
\end{align}
Combining (\ref{ord-Y1-last-Pminus-2}) and (\ref{ord-Y1-last-Pplus}), we have
\begin{align*}
\deg_\theta &({\bf Y}_1(t)) \\
&\geq \min\left\{  \KMj{\deg_\theta} \left( \left( \frac{\theta(t)}{\theta(t_0)}\right)^{-A} P_- \KMi{\left( {\bf Y}_1^0 + \int_{t_0}^t \left( \frac{\theta(\eta)}{\theta(t_0)}\right)^{A} {\bf g}_1(\eta) d\eta \right)}  \right), \KMl{ \deg_\theta \left(  -\int_{t}^{t_{\max}} \left( \frac{\theta(\eta)}{\theta(t)}\right)^{A} P_+ {\bf g}_1(\eta) d\eta \right)} \right\}\\
	&\geq \delta
\end{align*}
and hence our claim (\ref{estimate-order-asym}) is proved when $N=1$.
\par
\bigskip
Now we assume that (\ref{estimate-order-asym}) holds true for some $N\geq 1$.
Our next \KMl{issue is the case $N+1$.} 
\KMl{The} main investigation is the magnitude of $\theta(t)$ in the following two vector-valued functions:
\begin{align*}
{\bf g}_{N+1,1}(t)&= \theta(t)^{-1}\left\{ R_{\alpha, k}({\bf S}_N ({\bf Y})(t)) - R_{\alpha, k}({\bf S}_{N-1}({\bf Y})(t)) \right\},\\
{\bf g}_{N+1,2}(t) &= \theta(t)^{ \frac{1}{k}\Lambda_\alpha }  \left\{ f_{{\rm res}}\left( \theta(t)^{-\frac{1}{k}\Lambda_\alpha} {\bf S}_N({\bf Y})(t) \right) - f_{{\rm res}} \left( \theta(t)^{-\frac{1}{k}\Lambda_\alpha}  {\bf S}_{N-1}({\bf Y})(t) \right) \right\}.
\end{align*}
First note that, by the Taylor's theorem, there are constants $\eta_1, \eta_2 \in (0,1)$ such that
\begin{align*}
&R_{\alpha, k}({\bf S}_N ({\bf Y})(t)) - R_{\alpha, k}({\bf S}_{N-1}({\bf Y})(t)) \\
&= f_{\alpha, k}({\bf S}_N ({\bf Y})(t)) - f_{\alpha, k}({\bf S}_{N-1}({\bf Y})(t)) \KMl{-} Df_{\alpha, k}({\bf Y}_0){\bf Y}_N(t) \\
 	&= \{ Df_{\alpha, k}({\bf S}_{N-1}({\bf Y})(t) + \eta_1 {\bf Y}_N(t) ) - Df_{\alpha, k}({\bf Y}_0)\}{\bf Y}_N(t)\\
 	&= \left[ \frac{1}{2}D^2f_{\alpha, k} \left( {\bf Y}_0 + \eta_2 \left({\bf S}_{N-1}({\bf Y})(t) + \eta_1 {\bf Y}_N(t)- {\bf Y}_0\right) \right) \left\{ {\bf S}_{N-1}({\bf Y})(t) + \eta_1 {\bf Y}_N(t) - {\bf Y}_0 \right\} \right] {\bf Y}_N(t).
 \end{align*}
By assertions in (\ref{order-QH-derivatives}), we observe that $D^2f_{\alpha, k}({\bf Y}_0 + \eta_2 \left({\bf S}_{N-1}({\bf Y})(t) + \eta_1 {\bf Y}_N(t)- {\bf Y}_0\right) )$ is $O(1)$ as $t\to t_{\max}$ as a bilinear map.
By the assumption of induction and Proposition \ref{prop-ord-fundamental}-2 and 3, we have 
\begin{equation*}
\KMj{\deg_\theta}\left( R_{\alpha, k}({\bf S}_{\KMl{N}}({\bf Y})) - R_{\alpha, k}({\bf S}_{N-1}({\bf Y})) \right) \geq (N+1)\delta
\end{equation*}
and hence 
\begin{equation*}
\KMj{\deg_\theta}\left( {\bf g}_{N+1,1}(t) \right) \geq (N+1)\delta - 1.
\end{equation*}
Next consider ${\bf g}_{N+1, 2}(t)$.
First we have 
\begin{align*}
&f_{{\rm res}}( \theta(t)^{-\frac{1}{k}\Lambda_\alpha} {\bf S}_N ({\bf Y})(t)) - f_{{\rm res}}( \theta(t)^{-\frac{1}{k}\Lambda_\alpha} {\bf S}_{N-1}({\bf Y})(t)) \\
&= Df_{{\rm res}} \left( \theta(t)^{-\frac{1}{k}\Lambda_\alpha} {\bf S}_{N-1}({\bf Y})(t) + \eta_3 \theta^{- \frac{1}{k}\Lambda_\alpha } {\bf Y}_N(t) \right) \theta^{-\KMf{ \frac{1}{k}\Lambda_\alpha }}{\bf Y}_N(t)
\end{align*}
by the Taylor's theorem for some constant $\eta_3\in (0,1)$.
Under the constraint (\ref{order-fres-diff}) for $f_{\rm res}$, we have
\begin{equation*}
\KMj{\deg_\theta} \left( f_{i,{\rm res}}( \theta(t)^{-\frac{1}{k}\Lambda_\alpha} {\bf S}_N ({\bf Y})) - f_{i,{\rm res}}( \theta(t)^{-\frac{1}{k}\Lambda_\alpha} {\bf S}_{N-1}({\bf Y})) \right)  \geq \gamma_i + N\delta,
\end{equation*}
where we have also used the assumption of the induction $\KMj{\deg_\theta}({\bf Y}_N(t)) \geq N\delta$.
In particular, we have
\begin{align*}
\notag
\KMj{\deg_\theta} \left( {\bf g}_{N+1,2}(t) \right) & \geq \min_{l=1,\ldots, n}\left\{ \frac{\alpha_l}{k} + \gamma_l \right\} + N\delta \\
	&\geq (N+1)\delta - 1.
\end{align*}
In particular, we have
\begin{align}
\label{estimate-gN+1-deg}
\deg_\theta \left( {\bf g}_{N+1}(t) \right) &\geq \min \left\{\deg_\theta \left( {\bf g}_{N+1, 1}(t) \right), \deg_\theta \left( {\bf g}_{N+1,2}(t) \right)\right\} 
	\geq (N+1)\delta - 1,
\end{align}
which is larger than $-1$ because $\delta > 0$ by assumption.
Now the inequality (\ref{deg-integral-theta-F-minus-2}) in Proposition \ref{prop-deg-integral-theta-F} with ${\bf g}_{N+1,1}(t)$, ${\bf g}_{N+1,2}(t)$, and the projector $P_-$ yields
\begin{align*}
\deg_\theta \left( \KMl{\left(\frac{\theta (t)}{\theta (t_0)}\right)^{-A}} P_- \left( {\bf Y}_{N+1}^0 + \int_{t_0}^t \left(\frac{\theta (\eta)}{\theta (t_0)}\right)^{A} {\bf g}_{N+1}(\eta) d\eta \right) \right) \geq (N+1)\delta,
\end{align*}
where we have used the commutativity $AP_- = P_-A$.
Note that the constant ${\bf Y}_{N+1}^0$ is chosen \KMl{so that the requirements stated in Proposition \ref{prop-conv-(j+1)th} are satisfied.}
\par
Finally, similar to the case $N=1$, we have
\begin{align*}
\KMj{\deg_\theta} \left( \left( \frac{\theta(t)}{\theta(t_0)}\right)^{-A} \int_{t}^{t_{\max}} \left( \frac{\theta(\eta)}{\theta(t_0)}\right)^{A} P_+ {\bf g}_{N+1}(\eta) d\eta  \right)  \geq (N+1)\delta.
\end{align*}
As a summary, we have
\begin{equation*}
\deg_\theta ({\bf Y}_{N+1}(t)) \geq (N+1)\delta
\end{equation*}
and the proof is completed.

%% file: blow_up_asymptotic_1.bbl
\begin{thebibliography}{10}

\bibitem{AIU2017}
K.~Anada, T.~Ishiwata, and T.~Ushijima.
\newblock A numerical method of estimating blow-up rates for nonlinear
  evolution equations by using rescaling algorithm.
\newblock {\em Japan Journal of Industrial and Applied Mathematics}, pages
  1--15, 2017.

\bibitem{A2002}
B.~Andrews.
\newblock Singularities in crystalline curvature flows.
\newblock {\em Asian J. Math.}, 6(1):101--122, 2002.

\bibitem{AIK2014}
T.~Arakawa, T.~Ibukiyama, M.~Kaneko, and D.~Zagier.
\newblock {\em Bernoulli numbers and zeta functions}.
\newblock Springer, 2014.

\bibitem{BO1999}
C.M. Bender and S.A. Orszag.
\newblock {\em Advanced mathematical methods for scientists and engineers {I}:
  {A}symptotic methods and perturbation theory}.
\newblock McGraw-Hill, 1999.

\bibitem{BFGK2011}
E.~Berchio, A.~Ferrero, F.~Gazzola, and P.~Karageorgis.
\newblock Qualitative behavior of global solutions to some nonlinear fourth
  order differential equations.
\newblock {\em Journal of Differential Equations}, 251(10):2696--2727, 2011.

\bibitem{BK1988}
M.~Berger and R.V. Kohn.
\newblock A rescaling algorithm for the numerical calculation of blowing-up
  solutions.
\newblock {\em Communications on pure and applied mathematics}, 41(6):841--863,
  1988.

\bibitem{Cha2011}
F.~Chatelin.
\newblock {\em Spectral approximation of linear operators}.
\newblock SIAM, 2011.

\bibitem{CHO2007}
C.-H. Cho, S.~Hamada, and H.~Okamoto.
\newblock On the finite difference approximation for a parabolic blow-up
  problem.
\newblock {\em Japan Journal of Industrial and Applied Mathematics},
  24(2):131--160, 2007.

\bibitem{D1985}
J.W. Dold.
\newblock Analysis of the early stage of thermal runaway.
\newblock {\em The Quarterly Journal of Mechanics and Applied Mathematics},
  38(3):361--387, 1985.

\bibitem{D1993}
F.~Dumortier.
\newblock Techniques in the theory of local bifurcations: Blow-up, normal
  forms, nilpotent bifurcations, singular perturbations.
\newblock In {\em Bifurcations and Periodic Orbits of Vector Fields}, pages
  19--73. Springer, 1993.

\bibitem{E1979}
W.~Eckhaus.
\newblock {\em Asymptotic analysis of singular perturbations}.
\newblock North-Holland, Amsterdam, 1979.

\bibitem{EG2006}
U.~Elias and H.~Gingold.
\newblock Critical points at infinity and blow up of solutions of autonomous
  polynomial differential systems via compactification.
\newblock {\em Journal of mathematical analysis and applications},
  318(1):305--322, 2006.

\bibitem{FM2002}
M.~Fila and H.~Matano.
\newblock Blow-up in nonlinear heat equations from the dynamical systems point
  of view.
\newblock {\em Handbook of dynamical systems}, 2:723--758, 2002.

\bibitem{FP1991}
A.~Fordy and A.~Pickering.
\newblock Analysing negative resonances in the {P}ainlev{\'e} test.
\newblock {\em Physics Letters A}, 160(4):347--354, 1991.

\bibitem{GV2002}
V.A. Galaktionov and J.-L. V{\'a}zquez.
\newblock The problem of blow-up in nonlinear parabolic equations.
\newblock {\em Discrete \& Continuous Dynamical Systems-A}, 8(2):399, 2002.

\bibitem{Hale1969}
J.K. Hale.
\newblock {\em {O}rdinary {D}ifferential {E}quations}.
\newblock John Wiley, New York, 1969.

\bibitem{Ha2016}
J.~Harada.
\newblock Blowup profile for a complex valued semilinear heat equation.
\newblock {\em Journal of Functional Analysis}, 270(11):4213--4255, 2016.

\bibitem{Ha2017}
J.~Harada.
\newblock Nonsimultaneous blowup for a complex valued semilinear heat equation.
\newblock {\em Journal of Differential Equations}, 263(8):4503--4516, 2017.

\bibitem{HV1994}
M.A. Herrero and J.J.L. Vel{\'a}zquez.
\newblock Explosion de solutions d'{\'e}quations paraboliques semilin{\'e}aires
  supercritiques.
\newblock {\em C. R. Acad. Sci. S\'{e}r. 1 Math.}, 319:141--145, 1994.

\bibitem{IY2003}
T.~Ishiwata and S.~Yazaki.
\newblock On the blow-up rate for fast blow-up solutions arising in an
  anisotropic crystalline motion.
\newblock {\em Journal of Computational and Applied Mathematics},
  159(1):55--64, 2003.

\bibitem{KSS2003}
B.L. Keyfitz, R.~Sanders, and M.~Sever.
\newblock Lack of hyperbolicity in the two-fluid model for two-phase
  incompressible flow.
\newblock {\em DISCRETE AND CONTINUOUS DYNAMICAL SYSTEMS SERIES B},
  3(4):541--564, 2003.

\bibitem{KR2011}
P.E. Kloeden and M.~Rasmussen.
\newblock {\em Nonautonomous dynamical systems}.
\newblock Number 176. Amer. Math. Soc., 2011.

\bibitem{asym2}
H.~Kodani, K.~Matsue, H.~Ochiai, and A.~Takayasu.
\newblock Multi-order asymptotic expansion of blow-up solutions for autonomous
  {ODE}s. {II} - {D}ynamical {C}orrespondence.
\newblock {\em submitted}, 2022.

\bibitem{KK1990}
H.C. Kranzer and B.L. Keyfitz.
\newblock A strictly hyperbolic system of conservation laws admitting singular
  shocks.
\newblock In {\em Nonlinear evolution equations that change type}, pages
  107--125. Springer, 1990.

\bibitem{K1992}
M.D. Kruskal.
\newblock Flexibility in applying the {P}ainlev{\'e} test.
\newblock In {\em Painlev{\'e} Transcendents}, pages 187--195. Springer, 1992.

\bibitem{KJH1997}
M.D. Kruskal, N.~Joshi, and R.~Halburd.
\newblock Analytic and asymptotic methods for nonlinear singularity analysis: a
  review and extensions of tests for the {P}ainlev{\'e} property.
\newblock {\em Integrability of nonlinear systems}, pages 171--205, 1997.

\bibitem{LMT2021}
J.-P. Lessard, K.~Matsue, and A.~Takayasu.
\newblock Saddle-type blow-up solutions with computer-assisted proofs:
  Validation and extraction of global nature.
\newblock {\em arXiv:2103.12390}, 2021.

\bibitem{Mat2018}
K.~Matsue.
\newblock On blow-up solutions of differential equations with
  {P}oincar\'{e}-type compactifications.
\newblock {\em SIAM Journal on Applied Dynamical Systems}, 17(3):2249--2288,
  2018.

\bibitem{Mat2019}
K.~Matsue.
\newblock Geometric treatments and a common mechanism in finite-time
  singularities for autonomous {ODE}s.
\newblock {\em Journal of Differential Equations}, 267(12):7313--7368, 2019.

\bibitem{MT2020_1}
K.~Matsue and A.~Takayasu.
\newblock Numerical validation of blow-up solutions with quasi-homogeneous
  compactifications.
\newblock {\em Numerische Mathematik}, 145:605--654, 2020.

\bibitem{MT2020_2}
K.~Matsue and A.~Takayasu.
\newblock Rigorous numerics of blow-up solutions for {ODE}s with exponential
  nonlinearity.
\newblock {\em Journal of Computational and Applied Mathematics}, 374:112607,
  2020.

\bibitem{Rob}
C.~Robinson.
\newblock {\em Dynamical systems - {S}tability, {S}ymbolic {D}ynamics, and
  {C}haos}.
\newblock Studies in Advanced Mathematics. CRC Press, Boca Raton, FL, second
  edition, 1999.

\bibitem{TMSTMO2017}
A.~Takayasu, K.~Matsue, T.~Sasaki, K.~Tanaka, M.~Mizuguchi, and S.~Oishi.
\newblock Numerical validation of blow-up solutions for ordinary differential
  equations.
\newblock {\em Journal of Computational and Applied Mathematics}, 314:10--29,
  2017.

\bibitem{WXJ2021}
S.~Wieczorek, C.~Xie, and C.K.R.T. Jones.
\newblock Compactification for asymptotically autonomous dynamical systems:
  theory, applications and invariant manifolds.
\newblock {\em Nonlinearity}, 34(5):2970, 2021.

\end{thebibliography}
